\documentclass[11pt]{article}

\usepackage{amssymb,amsmath,amsthm,enumitem,mathrsfs,mathabx,accents,xy,pstricks}
\xyoption{all}

\makeatletter
\renewenvironment{proof}[1][\proofname]{\par
  \pushQED{\qed}%
  \normalfont \topsep-5\p@\@plus6\p@\relax
  \trivlist
  \item[\hskip\labelsep
    #1\@addpunct{.}]\ignorespaces
}{%
  \popQED\endtrivlist\@endpefalse
}
\makeatother

\usepackage{titlesec}
\titlespacing\section{0pt}{-3pt plus 2pt minus 1pt}{-7pt plus 2pt minus 1pt}
\titlespacing\subsection{0pt}{-3pt plus 2pt minus 1pt}{-7pt plus 2pt minus 1pt}
\titlespacing\subsubsection{0pt}{10pt plus 4pt minus 2pt}{1pt plus 2pt minus 2pt}

\numberwithin{equation}{section}

\newtheoremstyle{reduced_space}{}{-0.5\baselineskip}{}{}{\bfseries}{}{.5em}{}
\theoremstyle{reduced_space}

\newtheorem{thm}{Theorem}[section]
\newtheorem{prp}[thm]{Proposition}
\newtheorem{lmm}[thm]{Lemma}
\newtheorem{crl}[thm]{Corollary}

\theoremstyle{definition}
\newtheorem{dfn}[thm]{Definition}

\theoremstyle{remark}
\newtheorem{rmk}[thm]{Remark}

\def\BE#1{\begin{equation}\label{#1}}
\def\EE{\end{equation}}
\def\eref#1{(\ref{#1})}

\def\BEnum#1{\begin{enumerate}[label=#1,leftmargin=*,topsep=-10pt,itemsep=-3pt]}
\def\EEnum{\end{enumerate}}

\def\ov#1{\overline{#1}}
\def\sf#1{\textsf{#1}}
\def\wt#1{\widetilde{#1}}
\def\tn#1{\textnormal{#1}} 
\def\lr#1{\langle{#1}\rangle}
\def\blr#1{\big\langle{#1}\big\rangle}

\def\wh#1{\widehat{#1}}

\def\unbr#1#2{\underset{#2}{\underbrace{#1}}}
\def\sm#1{\begin{small}#1\end{small}}
\def\flr#1{\left\lfloor{#1}\right\rfloor}

\def\lra{\longrightarrow}
\def\Lra{\Longrightarrow}
\def\xlra#1{\xrightarrow{{#1}}}

\def\cB{\mathcal B}
\def\C{\mathbb C}
\def\cC{\mathcal C}
\def\bfC{\mathbf C}
\def\D{\mathbb D}
\def\cD{\mathcal D}

\def\cJ{\mathcal J}

\def\M{\mathfrak M}
\def\cM{\mathcal M}
\def\cN{\mathcal N}

\def\P{\mathbb P}
\def\cP{\mathcal P}
\def\R{\mathbb R}
\def\Q{\mathbb Q}

\def\cS{\mathcal S}

\def\cZ{\mathcal Z}
\def\Z{\mathbb Z}

\def\al{\alpha}
\def\be{\beta}

\def\ep{\epsilon}
\def\ga{\gamma}
\def\io{\iota}

\def\la{\lambda}

\def\si{\sigma}
\def\om{\omega}
\def\th{\theta}

\def\vt{\vartheta}

\def\De{\Delta}
\def\Ga{\Gamma}
\def\La{\Lambda}
\def\Om{\Omega}
\def\Si{\Sigma}

\def\fb{\mathfrak b}
\def\fbb{\mathfrak{bb}}

\def\fd{\mathfrak d}
\def\ff{\mathfrak f}
\def\fj{\mathfrak j}

\def\fo{\mathfrak o}

\def\fs{\mathfrak s}

\def\u{\mathbf u}

\def\codim{\tn{codim}}

\def\nd{\tn{d}}
\def\dim{\tn{dim}}
\def\dom{\tn{dom}}
\def\ev{\tn{ev}}
\def\evb{\tn{evb}}
\def\evi{\tn{evi}}
\def\id{\tn{id}}
\def\Id{\tn{Id}}
\def\Im{\tn{Im}}
\def\lk{\tn{lk}}
\def\nod{\tn{nd}}
\def\PD{\tn{PD}}

\def\pt{\tn{pt}}

\def\sgn{\tn{sgn}}

\def\fiber{\times_\tn{fb}}

\def\Dks{\tn{Dcs}}
\def\Bds{\tn{Bds}}
\def\FPt{\tn{FPt}}
\def\FPC{\tn{FPC}}
\def\PC{\tn{PC}}

\def\i{\infty}

\def\eset{\emptyset}
\def\prt{\partial}
\def\dbar{\ov\partial}
\def\st{\bigstar}

\def\bt{\mathbf t}

\def\os{\mathfrak{os}}

\def\bu{\bullet}
\def\v{\vee}

\def\bsl{\backslash}

\begin{document}

\title{A Geometric Depiction of\\ 
Solomon-Tukachinsky's Construction of Open GW-Invariants}
\author{Xujia Chen\thanks{Supported by NSF grant DMS 1901979}}
\date{\today}

\maketitle

\begin{abstract}
\noindent
The 2016 papers of J.~Solomon and S.~Tukachinsky use bounding chains in 
Fukaya's $A_{\i}$-algebras to define numerical disk counts relative to a Lagrangian
under certain regularity assumptions on the moduli spaces of disks.
We present a (self-contained) direct geometric analogue of their construction under
weaker topological assumptions, extend it over arbitrary rings in the process, and
sketch an extension without any assumptions over rings containing the rationals.
This implements the intuitive suggestion represented by their drawing 
and P.~Georgieva's perspective.
We also note a curious relation for the standard Gromov-Witten invariants
readily deducible from their~work.
In a sequel, we use the geometric perspective of this paper to relate 
Solomon-Tukachinsky's invariants to Welschinger's open invariants of
symplectic sixfolds, confirming their belief and G.~Tian's related expectation
concerning K.~Fukaya's earlier construction.
\end{abstract}

\tableofcontents
\setlength{\parskip}{\baselineskip}

\section{Introduction}
\label{intro_sec}

Let $(X,\om)$ be a compact symplectic manifold of real dimension~$2n$
with $n\!\not\in\!2\Z$, $Y\!\subset\!X$ be a compact Lagrangian submanifold, and 
$$\wh{H}^{2*}(X,Y;R)\equiv  H^{n+1}(X,Y;R)\oplus
\bigoplus_{\begin{subarray}{c}p\in\Z\\ 2p\neq n+1\end{subarray}} \!\!\!\!\!
H^{2p}(X;R)$$
for any commutative ring $R$ with unity~1.
Fix a relative OSpin-structure $\os\!\equiv\!(\fo,\fs)$ on~$Y$,
i.e.~a pair consisting of an orientation~$\fo$ on~$Y$ and 
a relative Spin-structure~$\fs$ on the oriented manifold~$(Y,\fo)$.

Based on $A_{\i}$-algebra considerations, 
K.~Fukaya~\cite{Fuk11} uses  \sf{bounding chains} to define counts 
$$\lr{}_{\be,0}^{\om,\os}\in\R, \qquad\hbox{with}\quad \be\!\in\!H_2(X,Y;\Z),$$
of $J$-holomorphic degree~$\be$ disks in~$X$ with boundary in~$Y$ under the assumption
that $(X,\om)$ is a Calabi-Yau threefold and the \sf{Maslov index} 
\BE{Maslovdfn_e}\mu_Y^{\om}\!:H_2(X,Y;\Z)\lra\Z\EE
of~$Y$ vanishes.
These counts do not depend on the choice of bounding chains, but
may depend on the choice of the almost complex structure~$J$
compatible with~$(X,\om)$.

Motivated by~\cite{Fuk10,Fuk11} and after some preparation  in~\cite{JS1},
J.~Solomon and S.~Tukachinsky~\cite{JS2} use bounding chains in
Fukaya's $A_{\i}$-algebras to define counts 
\BE{JSinvdfn_e}\lr{\cdot,\ldots,\cdot}_{\be,k}^{\om,\os}\!: 
\bigoplus_{l=0}^{\i}\wh{H}^{2*}(X,Y;\R)^{\oplus l}\lra\R, \quad 
\be\!\in\!H_2(X,Y;\Z),~k\!\in\!\Z^{\ge0},\EE
of $J$-holomorphic disks in~$(X,Y)$ under the assumption that 
the (uncompactified) moduli spaces $\M_{0,0}(\be;J)$ on unmarked $J$-holomorphic degree~$\be$ 
disks are regular and the evaluation~maps 
$$\evb_1\!:\M_{1,0}(\be;J)\lra Y$$
from the moduli spaces of disks with one boundary marked point are submersions.
If $Y$ is an $\R$-homology odd-dimensional sphere, the relevant bounding chains exist
and the counts~\eref{JSinvdfn_e} are independent of the choices of bounding chains.
These counts also remain invariant under deformation of the almost complex structure
that respect the above regularity assumptions.
The authors of~\cite{JS2} call the counts~\eref{JSinvdfn_e}
\sf{open Gromov-Witten invariants}.
Inline with G.~Tian's perspective on K.~Fukaya's construction in~\cite{Fuk11},
they expect these invariants 
to be related to Welschinger's open Gromov-Witten invariants~\cite{Wel13},
which count multi-disks weighted by self-linking numbers.

As informally noted by the authors of~\cite{JS1,JS2} and by P.~Georgieva
(who described the idea below to the author), 
the construction in~\cite{JS2} based on primarily algebraic considerations should
have a geometric interpretation generalizing~\cite{Wel13}
via linking numbers of arbitrary-dimensional cycles in~$Y$.
More precisely,
suppose $\bfC$ is a generic collection of constraints in~$Y$ and~$X$ so that
the (expected) dimension of the space $\Dks(\be,\bfC)$
of $J$-holomorphic degree~$\be$ disks in~$(X,Y)$ 
passing through~$\bfC$ is~0.
A relative OSpin-structure~$\os$ then determines 
a signed cardinality~$n_{\be}^{\os}(\bfC)$ of this finite~set. 
For any splittings 
\BE{DiskSplit_e}\be=\be_1\!+\!\be_2\in H_2(X,Y;\Z) \qquad\hbox{and}\qquad 
\bfC=\bfC_1\!\sqcup\!\bfC_2\EE
of the degree and constraints, the total (expected) dimension of
the circle bundles $\Bds(\be_1,\bfC_1)$ and $\Bds(\be_2,\bfC_2)$ in~$Y$
formed by the boundaries of
the disks in $\Dks(\be_1,\bfC_1)$ and $\Dks(\be_2,\bfC_2)$, respectively, is $n\!-\!1$,
the correct dimension for taking a linking number~$\lk_{\be_1,\be_2}^{\os}\!(\bfC_1,\bfC_2)$
of $\Bds(\be_1,\bfC_1)$ and $\Bds(\be_2,\bfC_2)$ in~$Y$.
A lift of a generic path of almost complex structures~$J_t$ and constraints~$\bfC_t$
to~$\Dks(\be,\bfC)$ could terminate at a nodal disk corresponding to a pair of disks in 
$\Dks(\be_1,\bfC_1)$ and $\Dks(\be_2,\bfC_2)$ intersecting along their boundaries.
Its lift to $\Dks(\be_1,\bfC_1)\!\times\!\Dks(\be_2,\bfC_2)$ also passes 
through this pair of disks,
with a change in the associated linking number~$\lk_{\be_1,\be_2}^{\os}\!(\bfC_1,\bfC_2)$.
One might thus hope that some combination of the numbers~$n_{\be}^{\os}(\bfC)$
and~$\lk_{\be_1,\be_2}^{\os}\!(\bfC_1,\bfC_2)$, with $(\be_1,\bfC_1)$ and~$(\be_2,\bfC_2)$
as in~\eref{DiskSplit_e}, remains invariant over a generic path 
of almost complex structures~$J_t$ and constraints~$\bfC_t$.

If $n\!=\!3$, the families $\Dks(\be_1,\bfC_1)$ and $\Dks(\be_2,\bfC_2)$ are compact,
as needed for defining a linking number~$\lk_{\be_1,\be_2}^{\os}\!(\bfC_1,\bfC_2)$ of
$\Bds(\be_1,\bfC_1)$ and $\Bds(\be_2,\bfC_2)$.
The reasoning in the previous paragraph then leads to the open Gromov-Witten invariants
of~\cite{Wel13} enumerating linked multi-disks. 
If $n\!>\!3$, these families are generally not compact, as the disks 
in $\Dks(\be_i,\bfC_i)$ might degenerate to nodal disks.
The resulting (codimension~1) boundaries of $\Bds(\be_i,\bfC_i)$
then need to be canceled in some~way.
J.~Solomon and S.~Tukachinsky do so by choosing auxiliary differential forms,
which are then integrated over moduli spaces of $J$-holomorphic disks,
in a consistent manner.

The present paper is a geometric ``translation'' of (some of) the definitions 
and arguments in~\cite{JS1,JS2} in terms of auxiliary bordered pseudocycles to~$Y$,
chosen in a consistent manner,
which are intersected with moduli spaces of $J$-holomorphic disks.
This ``translation''
%is motivated by the interpretation of the real Gromov-Witten invariants of~\cite{Wel4,Jake}
%as degrees of relatively orientable pseudocycles in \cite{RealWDVV,RealWDVV3}, 
makes sense of the linking number picture above and realizes \cite[Fig~1]{JS2}.
It applies over any commutative ring~$R$ with~unity under the topological assumptions~that
\BE{strongpos_e1}\om(B)>0,~~\blr{c_1(X,\om),B}\ge 3\!-\!\dim\,Y \qquad\Lra\qquad
\blr{c_1(X,\om),B}\ge 0\EE
for every spherical class $B\!\in\!H_2(X;\Z)$ and  
\BE{strongpos_e2}\om(\be)>0,~~\mu_Y^{\om}(\be)\ge 3\!-\!\dim\,Y \qquad\Lra\qquad
\mu_Y^{\om}(\be)>0\EE
for every $\be\!\in\!H_2(X,Y;\Z)$ representable by a map from~$(\D^2,S^1)$. 
In order for the regularity assumptions of~\cite{JS2} to hold 
for a fixed almost complex structure~$J$,
the last inequalities in~\eref{strongpos_e1} and~\eref{strongpos_e2} must hold for all~$B$ representable by $J$-holomorphic maps from~$S^2$ and
for all~$\be$ representable by $J$-holomorphic maps from~$(\D^2,S^1)$,
respectively.
By~\eref{strongpos_e1} and~\eref{strongpos_e2}, 
the regularity assumptions of~\cite{JS2} hold over the moduli spaces of {\it simple}
$J$-holomorphic disk maps
for  a generic $\om$-compatible almost complex structure~$J$ on~$X$ and  
the images of the multiply covered maps under evaluation maps are of codimension
at least~2.
If $Y$ is an $\R$-homology odd-dimensional sphere, the disk counts we construct
are independent of the choice of~$J$ and thus are invariants of~$(X,\om,Y,\os)$;
this is a stronger invariance property than in~\cite{JS2}.
The main statements of the present paper are Theorems~\ref{countinv_thm} and~\ref{main_thm}.
In Appendix~\ref{vfc_app}, we sketch an adaptation of the geometric construction described in
this paper compatible with standard virtual class approaches.

The idea behind the notion of \sf{bounding chain} of Definition~\ref{bndch_dfn},
which is a geometric analogue of the notions used in \cite{Fuk11,JS1,JS2}, 
can be roughly described as follows.
Let $\be\!\in\!H_2(X,Y;\Z)$ and
$\bfC$ be a generic collection of constraints in~$Y$ and~$X$.
A boundary stratum~$\cS$ of $\Bds(\be,\bfC)$ is the total space of the $S^1\!\vee\!S^1$-bundle
formed by the fibers of $\Bds(\be_1,\bfC_1)$ over $\Dks(\be_1,\bfC_1)$
and $\Bds(\be_2,\bfC_2)$ over $\Dks(\be_2,\bfC_2)$ that intersect in~$Y$,
for some~$(\be_1,\bfC_1)$ and~$(\be_2,\bfC_2)$ as in~\eref{DiskSplit_e}.
Suppose, by inductive hypothesis, that we have already defined
closed cycles~$\fbb(\be_1,\bfC_1)$ and~$\fbb(\be_2,\bfC_2)$ in~$Y$
containing~$\Bds(\be_1,\bfC_1)$ and~$\Bds(\be_2,\bfC_2)$, respectively.
If $Y$ is a homology sphere and the dimension of~$\Bds(\be_2,\bfC_2)$
is not~0 or~$n$, 
we can take a bordered pseudocycle~$\fb(\be_2,\bfC_2)$ in~$Y$
that bounds~$\fbb(\be_2,\bfC_2)$.
The fibers of~$\Bds(\be_1,\bfC_1)$ 
that intersect~$\fb(\be_2,\bfC_2)$ in~$Y$ form an $S^1$-bundle 
$\fbb(\be_1,\bfC_1;\be_2,\bfC_2)$.
The fibers of~$\Bds(\be_1,\bfC_1)$ that intersect 
\hbox{$\Bds(\be_2,\bfC_2)\!\subset\!\prt\fb(\be_2,\bfC_2)$} in~$Y$
form part of $\prt\fbb(\be_1,\bfC_1;\be_2,\bfC_2)$.
Since this part is isomorphic to~$\cS$, we can  glue $\fbb(\be_1,\bfC_1;\be_2,\bfC_2)$ 
to~$\Bds(\be,\bfC)$ along their common boundary.
Doing this for all $(\be_1,\bfC_1)$ and $(\be_2,\bfC_2)$ satisfying~\eref{DiskSplit_e},
we eliminate the boundary of~$\Bds(\be,\bfC)$.
(By the nature of this construction,
the remaining parts of the boundaries of the various pseudocycles 
$\fbb(\be_1,\bfC_1;\be_2,\bfC_2)$ cancel with each other in a similar manner.)
We thus  
obtain a closed cycle~$\fbb(\be,\bfC)$ in~$Y$ and complete the inductive step; 
see Lemma~\ref{BCpseudo_lmm} and the proof of Proposition~\ref{bndch_prp}.
If the dimension of this cycle is~$n$, we can take its degree and obtain
a count of $J$-holomorphic disks in~$(X,Y)$ as in~\eref{JSinvdfn_e}.
If $Y$ is a homology sphere, this count does not depend on the choice of~$\fb(\be_2,\bfC_2)$
above; see Section~\ref{Bc_subs}.

Analogously to~\cite{JS2}, we relate the disk counts arising from 
the bounding chains of Definition~\ref{bndch_dfn} to the real Gromov-Witten invariants 
of~\cite{Wel6,Wel6b,Jake,Penka2} in the appropriate real settings;
see Theorems~\ref{WelReal_thm} and~\ref{PenkaReal_thm}.
We also translate the statements of the WDVV-type relations for the open Gromov-Witten invariants 
obtained in~\cite{JS3} into relations for these disk counts; see Theorem~\ref{OpenWDVV_thm}.
Combining Theorems~\ref{main_thm} and~\ref{OpenWDVV_thm}, 
we obtain an intriguing relation between the standard, closed Gromov-Witten invariants
of~$(X,\om)$; see Corollary~\ref{OpenWDVV_crl}.
In~\cite{JakeSaraWel}, we show that the open invariants of Theorem~\ref{main_thm} 
reduce to 
the open Gromov-Witten invariants of~\cite{Wel13} if $n\!=\!3$, 
as expected in~\cite{JS2} and earlier envisioned by G.~Tian based on~\cite{Fuk11}.
This then yields WDVV-type relations for Welschinger's open invariants.

We hope that our geometric interpretation of Solomon-Tukachinsky's construction 
of open GW-invariants will make them accessible to a broader audience and 
will be developed further.
Via relatively orientable pseudocycles (as defined in~\cite{RealWDVV}),
this interpretation might lead to a construction of such invariants
for the cases when $n$ is even or the Lagrangian $Y$ is not orientable.
Along with a similar geometric interpretation of~\cite{JS3},
this should in turn lead to WDVV-type equations for open GW-invariants in such settings as~well.

\vspace{.1in}

\begin{rmk}\label{neven_rmk}
The present paper is based on the first version of~\cite{JS2}.
While this paper was being completed, the second version of~\cite{JS2} 
partly extended~\eref{JSinvdfn_e} to odd-dimensional cohomology on~$(X,Y)$ and even-dimensional~$Y$. 
The signs are a more delicate issue in these cases. 
\end{rmk}

\vspace{-.1in}

The author would like to thank Penka Georgieva for hosting her at 
the Institut de Math\'ematiques de Jussieu in March~2019 and
for the enlightening discussions that inspired the present paper.
She would also like to thank Sara Tukachinsky and Jake Solomon
 for clarifying some statements in~\cite{JS3} and Aleksey Zinger for many detailed discussions and help with the exposition.

\section{Setup and main statements}
\label{main_sec}

\subsection{Notation and terminology}
\label{Notation_subs}

For $k\!\in\!\Z^{\ge0}$, we define $[k]=\{1,2,\ldots,k\}$. 
We denote by $\D^2\!\subset\!\C$ the unit disk with the induced complex structure,
by  $\D^2\!\v\!\D^2$ the union of two disks joined at a pair of boundary points,
and by $S^1\!\subset\!\D^2$ and $S^1\!\v\!S^1\!\subset\!\D^2\!\v\!\D^2$ 
the respective boundaries.
We orient the boundaries counterclockwise; 
thus, starting from a smooth point~$x_0$ of $S^1\!\v\!S^1$, 
we proceed counterclockwise to the node~$\nod$,
then circle the second copy of~$S^1$ counterclockwise back to~$\nod$,
and return to~$x_0$ counterclockwise from~$\nod$.
We call smooth points $x_0,x_1,\ldots,x_k$ on $S^1$ or $S^1\!\v\!S^1$ \sf{ordered by position}
if they are traversed in counterclockwise order; 
see the first diagram in Figure~\ref{cSetavt_fig} on page~\pageref{cSetavt_fig}.

Let $Y$ be a smooth compact manifold.
%For a point $p\!\in\!Y$, we denote its inclusion into~$Y$ also by~$p$.
For a continuous map $f\!:\cZ\!\lra\!Y$,  let 
$$\Om(f)=\bigcap_{K\subset \cZ\text{~cmpt}}\!\!\!\!\!\!\!\!\ov{f(\cZ\!-\!K)}$$
be \sf{the limit set of~$f$}.
A continuous map \hbox{$f\!:\cZ\!\lra\!Y$} from a manifold, possibly with boundary,
is a \sf{$\Z_2$-pseudocycle} into~$Y$ if there exists a smooth map $h\!:\cZ'\!\lra\!Y$ such~that 
$$\dim\,\cZ'\le \dim\,\cZ\!-\!2 \qquad\hbox{and}\qquad
f(\prt\cZ),\Om(f)\subset h(\cZ')\,.$$
The \sf{codimension} of such a $\Z_2$-pseudocycle is \hbox{$\dim\,Y\!-\!\dim\,\cZ$}.
A continuous map \hbox{$\wt{f}\!:\wt\cZ\!\lra\!Y$} is a
\sf{bordered $\Z_2$-pseudocycle with boundary} \hbox{$f\!:\cZ\!\lra\!Y$} if
there exist an open subset $\cZ^*\!\subset\!\cZ$ and
a smooth map $\wt{h}\!:\wt\cZ'\!\lra\!Y$ such~that 
$$\cZ^*\subset\prt\wt\cZ, \quad \wt{f}|_{\cZ^*}=f|_{\cZ^*},\quad
\dim\,\wt\cZ'\le \dim\,\wt\cZ\!-\!2,\quad
f(\cZ\!-\!\cZ^*),\wt{f}\big(\prt\wt\cZ\!-\!\cZ^*\big),
\Om(\wt{f})\subset\wt{h}(\wt\cZ')\,.$$
Throughout the paper, we take oriented pseudocycles with coefficients
in a commutative ring~$R$ with unity.
Every $R$-homology class in a manifold can be represented
by a pseudocycle in this sense, which is unique up to equivalence;
see Theorem~1.1 in~\cite{pseudo}.

Let $(X,\om)$ be a compact symplectic manifold of dimension~$2n$,
$Y\!\subset\!X$ be a compact Lagrangian submanifold, 
\BE{H2omXYdfn_e} H_2^{\om}(X,Y)=
\big\{\be\!\in\!H_2(X,Y;\Z)\!:\om(\be)\!>\!0~\hbox{or}~\be\!=\!0\big\},\EE
and $\cJ_{\om}$ be the space of $\om$-compatible almost complex structures on~$X$.
We denote by $\PC(X)$ the collection of pseudocycles to~$X$ with coefficients in~$R$,
by $\FPC(X)$ the collection of finite subsets of~$\PC(X)$, 
and by $\FPt(Y)$ the collection of finite subsets of~$Y$.
Let 
\BE{cComYdfn_e}\cC_{\om}(Y)=\big\{(\be,K,L)\!:\,\be\!\in\!H_2^{\om}(X,Y),\,
K\!\in\!\FPt(Y),\,L\!\in\!\FPC(X),\,
(\be,K,L)\!\neq\!(0,\eset,\eset)\big\}.\EE
This collection has a natural partial order:
\BE{cComYdfn_e2}(\be',K',L')\preceq(\be,K,L) \qquad\hbox{if}\qquad
\be\!-\!\be'\in H_2^{\om}(X,Y), \quad
K'\!\subset\!K, \quad\hbox{and}\quad L'\!\subset\!L.\EE
The elements $(0,K,L)$ of $\cC_{\om}(Y)$ with $|K|\!+\!|L|\!=\!1$ are minimal 
with respect to this partial order.
For each element $\al\!\equiv\!(\be,K,L)$ of~$\cC_\om(Y)$, we define 
\begin{gather*}
\be(\al)\equiv\be,\qquad K(\al)\equiv K,\qquad L(\al)\equiv L, \\
\dim(\al)
=\mu_Y^{\om}(\be)\!+\!n\!-\!3\!-\!(n\!-\!1)|K|\!-\!\sum_{\Ga\in L}\!\!\big(\codim\,\Ga\!-\!2\big),\quad
\cC_{\om;\al}(Y)=\big\{\al'\!\in\!\cC_{\om}(Y)\!:\,\al'\!\prec\!\al\big\}.
\end{gather*}

\vspace{-.15in}

For $\al\!\equiv\!(\be,K,L)\!\in\!\cC_{\om}(Y)$, let
\BE{cDomdfn_e}\begin{split}
\cD_{\om}(\al)=\bigg\{\!\big(\be_{\bu},k_{\bu},L_{\bu},(\al_i)_{i\in[k_{\bu}]}\big)
\!\!:\be_{\bu}\!\in\!H_2^{\om}(X,Y),\,k_{\bu}\!\in\!\Z^{\ge0},\,L_{\bu}\!\subset\!L,\,
\al_i\!\in\!\cC_{\om}(Y)\,\forall\,i\!\in\![k_{\bu}],&\\
(\be_{\bu},k_{\bu},L_{\bu})\!\neq\!(0,1,\eset),\,
\be_{\bu}\!+\!\sum_{i=1}^{k_{\bu}}\be(\al_i)\!=\!\be,\,
\bigsqcup_{i=1}^{k_{\bu}}\!K(\al_i)\!=\!K,~
L_{\bu}\!\sqcup\!\bigsqcup_{i=1}^{k_{\bu}}\!L(\al_i)\!=\!L&\!\bigg\}.
\end{split}\EE
Since $\al_i\!\prec\!\al$ for every  
\BE{degenelemdfn_e} \eta\equiv\big(\be_{\bu},k_{\bu},L_{\bu},(\al_i)_{i\in[k_{\bu}]}\big)
\equiv\big(\be_{\bu},k_{\bu},L_{\bu},(\be_i,K_i,L_i)_{i\in[k_{\bu}]}\big)
\in \cD_{\om}(\al)\EE
and every $i\!\in\![k_{\bu}]$, $k_{\bu}\!=\!0$ if $\al$ is a minimal element 
of~$\cC_{\om}(Y)$.
Thus,
$$\cD_{\om}\big(0,\{\pt\},\eset\big)=\eset~~\forall\,\pt\!\in\!Y
\quad\hbox{and}\quad
\cD_{\om}\big(0,\eset,\{\Ga\}\big)=\big\{\big(0,0,\{\Ga\},()\!\big)\big\}
~~\forall\,\Ga\!\in\!\PC(X)\,.$$
For $\eta\!\in\!\cD_{\om}(\al)$ as in~\eref{degenelemdfn_e} and $i\!\in\![k_{\bu}]$, 
we define 
\begin{gather*}
\be_{\bu}(\eta)=\be_{\bu},\quad k_{\bu}(\eta)=k_{\bu}, \quad L_{\bu}(\eta)=L_{\bu}, \\
\be_i(\eta)=\be_i,\quad K_i(\eta)=K_i, \quad L_i(\eta)=L_i, \quad
\al_i(\eta)=\al_i=(\be_i,K_i,L_i).
\end{gather*}

\vspace{-.15in}

We denote by $\wt\PC(X)$ the collection of bordered pseudocycles to~$[0,1]\!\times\!X$ 
with coefficients in~$R$ and boundary in $\{0,1\!\}\!\times\!X$,
by $\wt\FPC(X)$ the collection of finite subsets of~$\wt\PC(X)$, 
and by $\wt\FPt(Y)$ the collection of finite sets of paths in $[0,1]\!\times\!Y$
from~$\{0\}\!\times\!Y$ to~$\{1\}\!\times\!Y$.
We define the partially ordered set~$\wt\cC_{\om}(Y)$ as in~\eref{cComYdfn_e}
and~\eref{cComYdfn_e2} with~$\FPt(Y)$ replaced by~$\wt\FPt(Y)$
and~$\FPC(X)$ by~$\wt\FPC(X)$.
For an element $\wt\al\!\equiv\!(\be,\wt{K},\wt{L})$
of $\wt\cC_{\om}(Y)$, 
we define the collection~$\cD_{\om}(\wt\al)$ as in~\eref{cDomdfn_e}
with $\cC_{\om}(Y)$, $K$, and $L$ replaced~by $\wt\cC_{\om}(Y)$, $\wt{K}$, and~$\wt{L}$,
respectively.
Let
\begin{gather*}
\dim(\wt\al)=\mu_Y^{\om}(\be)\!+\!n\!-\!3\!-\!(n\!-\!1)|\wt{K}|
\!-\!\sum_{\wt\Ga\in\wt{L}}\!\!\big(\codim\,\wt\Ga\!-\!2\big),\quad
\wt\cC_{\om;\wt\al}(Y)=\big\{\wt\al'\!\in\!\wt\cC_{\om}(Y)\!:\,\wt\al'\!\prec\!\wt\al\big\}.
\end{gather*}
We define $\wt\be(\wt\al),\wt{K}(\wt\al),\wt{L}(\wt\al)$ for $\wt\al\!\in\!\wt\cC_{\om}(Y)$
and $\be_{\bu}(\wt\eta),k_{\bu}(\wt\eta),\wt{L}_{\bu}(\wt\eta)$,
$\be_i(\wt\eta),\wt{K}_i(\wt\eta),\wt{L}_i(\wt\eta),\wt\al_i(\wt\eta)$ for 
\hbox{$\wt\eta\!\in\!\cD_{\om}(\wt\al)$} similarly to the analogous objects
for $\al\!\in\!\cC_{\om}(Y)$ and \hbox{$\eta\!\in\!\cD_{\om}(\al)$}.

For $\Ga_0,\Ga_1\!\in\!\FPC(X)$ and $\wt\Ga\!\in\!\wt\FPC(X)$, we write
$$\prt\wt\Ga=\{1\}\!\times\!\Ga_1\!-\!\{0\}\!\times\!\Ga_0$$
if $|\Ga_0|,|\Ga_1|\!=\!|\wt\Ga|$ and there is an ordering $\Ga_{0;i},\Ga_{1;i},\wt\Ga_i$
of the elements of $\Ga_0,\Ga_1,\wt\Ga$, respectively, so~that
$$\prt\wt\Ga_i=\{1\}\!\times\!\Ga_{1;i}-\{0\}\!\times\!\Ga_{0;i} \qquad\forall\,i.$$
For $K_0,K_1\!\in\!\FPt(Y)$ and $\wt{K}\!\in\!\wt\FPt(Y)$, we write
$$\prt\wt{K}=\{1\}\!\times\!K_1\!-\!\{0\}\!\times\!K_0$$
if the analogous condition holds.
For $\al_0,\al_1\!\in\!\cC_{\om}(Y)$ and $\wt\al\!\in\!\wt\cC_{\om}(Y)$, we write
\begin{gather*}
\prt\wt\al=\{1\}\!\times\!\al_1\!-\!\{0\}\!\times\!\al_0 \qquad\hbox{if}\quad
\be(\al_0),\be(\al_1)=\be(\wt\al),\\
\prt\wt{K}(\wt\al)=\{1\}\!\times\!K(\al_1)\!-\!\{0\}\!\times\!K(\al_0), ~~
\prt\wt{L}(\wt\al)=\{1\}\!\times\!L(\al_1)\!-\!\{0\}\!\times\!L(\al_0)\,.
\end{gather*}
If in addition $\eta_0\!\in\!\cD_{\om}(\al_0)$, 
$\eta_1\!\in\!\cD_{\om}(\al_1)$, and $\wt\eta\!\in\!\cD_{\om}(\wt\al)$,
we write
\begin{gather*}
\prt\wt\eta=\{1\}\!\times\!\eta_1\!-\!\{0\}\!\times\!\eta_0 
\qquad\hbox{if}\quad
\be_{\bu}(\eta_0),\be_{\bu}(\eta_1)=\be_{\bu}(\wt\eta),~~
k_{\bu}(\eta_0),k_{\bu}(\eta_1)=k_{\bu}(\wt\eta),\\
\prt\wt{L}_{\bu}(\wt\eta)=\{1\}\!\times\!L_{\bu}(\eta_1)
\!-\!\{0\}\!\times\!L_{\bu}(\eta_0),~~
\prt\al_i(\wt\eta)=\{1\}\!\times\!\al_i(\eta_1)
\!-\!\{0\}\!\times\!\al_i(\eta_0)~\forall\,i\!\in\!\big[k_{\bu}(\wt\eta)\!\big].
\end{gather*}

Let $k\!\in\!\Z^{\ge0}$, $L$ be a finite set, $\be\!\in\!H_2^{\om}(X,Y)$, and 
$J\!\in\!\cJ_{\om}$.
We denote~by $\M_{k,L}^{\st}(\be;J)$ the moduli space of
stable simple $J$-holomorphic degree~$\be$ maps from $(\D^2,S^1)$ and $(\D^2\!\v\!\D^2,S^1\!\v\!S^1)$
to~$(X,Y)$ with the interior marked points indexed by~$L$ and 
the boundary marked points indexed by $1,\ldots,k$ and ordered by the position.
A relative OSpin-structure $\os$ on $Y$ determines 
an orientation~$\fo_\os$ of~$\M_{k,L}^{\st}(\be;J)$; 
see Section~\ref{Ms_subs}. 
For $i\in[k]$ and $i\!\in\!L$, let
$$\evb_i\!:\M_{k,L}^{\st}(\be;J)\lra Y \qquad\hbox{and}\qquad
\evi_i\!:\M_{k,L}^{\st}(\be;J)\lra X$$
be the evaluation morphisms at the $i$-th boundary marked point and 
the $i$-th interior marked point, respectively. 
If $M\!\subset\!\M_{k,L}^\st(\be;J)$, 
we denote the restrictions of~$\evb_i$ and~$\evi_i$ to~$M$ also by~$\evb_i$ and~$\evi_i$. 

If in addition $m,m'\!\in\!\Z^{\ge0}$, 
$$\big(\fb_s\!:Z_{\fb_s}\!\lra\!Y\big)_{\!s\in[m]}
\qquad\hbox{and}\qquad
\big(\Ga_s\!:Z_{\Ga_s}\!\lra\!X\big)_{\!s\in [m']}$$
are tuples of maps and $i_1,\ldots,i_m\!\in\![k]$ and $j_1,\ldots,j_{m'}\!\in\!L$ are
distinct elements, 
let
\begin{equation*}\begin{split}
&M\!\!\fiber\!\big(\!(i_s,\fb_s)_{s\in[m]};(j_s,\Ga_s)_{s\in[m']}\big)\\
&~\equiv 
M_{(\evb_{i_1},\ldots,\evb_{i_m},\evi_{j_1},\ldots,\evi_{j_{m'}})}\!\!
\times_{\fb_1\times\ldots\times\fb_m\times\Ga_1\times\ldots\times\Ga_{m'}}\!\!
\big(Z_{\fb_1}\!\times\!\ldots\!\times\!Z_{\fb_m}\!\times\!
Z_{\Ga_1}\!\times\!\ldots\!\times\!Z_{\Ga_{m'}}\!\big)
\end{split}\end{equation*} 
be their fiber product with~$M$; see Section~\ref{Fp_subs}. 
If $M$ is an oriented manifold and~$\fb_s$ and~$\Ga_s$ are smooth maps 
from oriented manifolds satisfying the appropriate transversality conditions, 
then we orient this space as in Section~\ref{Fp_subs}. 
For $i\!\in\![k]$ with $i\!\neq\!i_s$ for any $s\!\in\![m]$ 
(resp. $i\!\in\!L$ with $i\!\neq\!j_s$ for any $s\!\in\![m']$), 
we define
$$\evb_i \ (\textnormal{resp. }\evi_i)\!:
M\!\fiber\!\big(\!(i_s,\fb_s)_{s\in[m]};(j_s,\Ga_s)_{s\in[m']}\big)
\lra Y\ (\tn{resp. }X)$$
to be the composition of the evaluation map~$\evb_i$ (resp. $\evi_i$) defined above with 
the projection to the first component.

For a path $\wt{J}\!\equiv\!(J_t)_{t\in[0,1]}$ in~$\cJ_{\om}$, let
$$\M_{k,L}^{\st}(\be;\wt{J})=
\big\{(t,\u)\!:t\!\in\![0,1],\,\u\!\in\!\M_{k,L}^{\st}(\be;J_t)\big\}.$$
For $i\in[k]$ and $i\!\in\!L$, we define 
\begin{alignat*}{2}
\wt\evb_i\!:\M_{k,L}^{\st}(\be;\wt{J})&\lra[0,1]\!\times\!Y, &\qquad
\wt\evb_i(t,\u)&=\big(t,\evb_i(\u)\!\big), \qquad\hbox{and}\\
\wt\evi_i\!:\M_{k,L}^{\st}(\be;\wt{J})&\lra[0,1]\!\times\!X, &\qquad
\wt\evi_i(t,\u)&=\big(t,\evi_i(\u)\!\big),
\end{alignat*}
respectively.
For $\wt M\!\subset\!\M_{k,L}^{\st}(\be;\wt{J})$, 
tuples $(\wt\fb_s)_{s\in[m]}$ and $(\wt\Ga_s)_{s\in[m']}$ of maps 
to $[0,1]\!\times\!Y$ and $[0,1]\!\times\!X$, respectively, 
$i_1,\ldots,i_m\!\in\![k]$, $j_1,\ldots,j_{m'}\!\in\!L$, and 
$i\!\in\![k]$ (resp. $i\!\in\!L$) as above, we define
$$\wt\evb_i \ (\textnormal{resp. }\wt\evi_i)\!:
\wt M\!\fiber\!\!\big(\!(i_s,\wt\fb_s)_{s\in[m]};(j_s,\wt\Ga_s)_{s\in[m']}\big)
\lra [0,1]\!\times\!Y\ (\tn{resp. }[0,1]\!\times\!X)$$
as in the previous paragraph.

A relative OSpin-structure~$\os$ on~$Y$ determines an orientation 
on $\M_{k,L}^{\st}(\be;\wt{J})$ with the base direction {\it first}.
In other words, the exact sequence
$$0\lra T_{\u}\M_{k,L}^{\st}(\be;J_t)\lra T_{(t,\u)}\M_{k,L}^{\st}(\be;\wt{J})
\xlra{\nd_{(t,\u)}e} T_t[0,1]\lra0$$
induced by the projection~$e$ to $[0,1]$ at a regular point $(t,\u)$ of~$e$ 
is orientation-compatible (as defined in Section~\ref{Fp_subs})
if and only if the dimension of $\M_{k,L}^{\st}(\be;J_t)$ is even.
If $\wt\fb_s$ and~$\wt\Ga_s$ are smooth maps from oriented manifolds,
then a relative OSpin-structure $\os$ on $Y$ also determines an orientation 
on the above fiber product space.

By the assumptions~\eref{strongpos_e1} and~\eref{strongpos_e2}, products of the evaluation maps 
from $\M_{k,L}^{\st}(\be;J)$ are bordered pseudocycles for 
a generic $\om$-compatible almost complex structure~$J$ on~$X$.
Since $\mu_Y^{\om}(\be)\!\in\!2\Z$ for all~$\be$,
the same applies to products of evaluation maps from $\M_{k,L}^{\st}(\be;\wt{J})$
for a generic path~$\wt{J}$ of  $\om$-compatible almost complex structures
between two generic $\om$-compatible almost complex structures~$J_0,J_1$.

\subsection{Bounding chains}
\label{Bc_subs}

Let $R$, $(X,\om,Y)$, $n$, and $\os$ be as before with $n\!\ge\!3$ odd. 
Thus,
\BE{dimeven_e}\dim(\al)\in 2\Z\qquad \forall\,\al\!\in\!\cC_{\om}(Y)\EE
if the dimension of every pseudocycle $\Ga\!\in\!L(\al)$ is even.
This implies that the pseudocycles~$\fb_{\al'}$, $\fb_{\wt\al'}$, 
$\fbb_{\al'}$,  and~$\fbb_{\wt\al'}$
of Definitions~\ref{bndch_dfn} and~\ref{psisot_dfn}, \eref{fbbdfn_e}, 
and~\eref{wtfbbdfn_e} below satisfy
\BE{bdimprp_e} \dim\,\fb_{\al'},\,\dim\,\fbb_{\wt\al'}\in2\Z  \quad\hbox{and}\quad
\dim\,\fbb_{\al'},\,\dim\,\fb_{\wt\al'}\not\in2\Z
\qquad\forall\,\al'\!\in\!\cC_{\om;\al}(Y),\,\wt\al'\!\in\!\wt\cC_{\om;\wt\al}(Y)\,.\EE 
For $\eta\in\cD_\om(\al)$ for some $\al\!\in\!\cC_{\om}(Y)$ and
$J\in\cJ_\om$, let
\BE{fMetaJdfn_e}\M_{\eta;J}\equiv\M^\st_{k_\bu(\eta),L_\bu(\eta)}(\be_\bu(\eta);J),\qquad
\M^+_{\eta;J}\equiv\M^\st_{k_\bu(\eta)+1,L_\bu(\eta)}(\be_\bu(\eta);J).\EE
For $\wt\eta\in\cD_\om(\wt\al)$ for some $\wt\al\!\in\!\wt\cC_{\om}(Y)$
and a path~$\wt J$ in~$\cJ_\om$, define 
$$\M_{\wt\eta;\wt J}\equiv\M^\st_{k_\bu(\wt\eta),\wt{L}_\bu(\wt\eta)}
(\be_\bu(\wt\eta);\wt J),
\qquad
\M^+_{\wt\eta;\wt J}\equiv\M^\st_{k_\bu(\wt\eta)+1,\wt{L}_\bu(\wt\eta)}(\be_\bu(\wt\eta);\wt J).$$
For a point $\pt\!\in\!Y$, we denote its inclusion into~$Y$ also by~$\pt$.

\begin{dfn}\label{bndch_dfn}
Let $\al\!\equiv\!(\be,K,L)\!\in\!\cC_{\om}(Y)$ be generic so that
the dimension of every pseudocycle $\Ga\!\in\!L(\al)$ is even
and $J\!\in\!\cJ_{\om}$ be generic.
A \sf{bounding chain} on~$(\al,J)$ is 
a collection $(\fb_{\al'})_{\al'\in\cC_{\om;\al}(Y)}$ of bordered pseudocycles to~$Y$ such~that
\BEnum{(BC\arabic*)}

\item\label{BCdim_it} $\dim\,\fb_{\al'}\!=\!\dim(\al')\!+\!2$ 
for all $\al'\!\in\!\cC_{\om;\al}(Y)$;

\item\label{BC0_it} $\fb_{\al'}\!=\!\eset$ if $\dim(\al')\!\le\!-2$ 
and $\al'\!\neq\!(0,\{\pt\},\eset)$ for any $\pt\!\in\!K$ or if $\dim(\al')\!\ge\!n\!-\!1$;

\item\label{BCi_it} $\fb_{(0,\{\pt\},\eset)}\!=\!\pt$ for all $\pt\!\in\!K$;

\item\label{BCprt_it} for all $\al'\!\in\!\cC_{\om;\al}(Y)$ such that 
$\dim(\al')\!\le\!n\!-\!2$,
\BE{BCprt_e}\prt\fb_{\al'}=
\bigg(\!\evb_1\!:\!\!\!\bigcup_{\eta\in\cD_{\om}(\al')}\hspace{-.12in}
(-1)^{k_\bu(\eta)}\M^+_{\eta;J}\!\fiber\!\!
\big(\!(i\!+\!1,\fb_{\al_{i}(\eta)})_{i\in[k_{\bu}(\eta)]}; 
(i,\Ga_i)_{\Ga_i\in L_{\bu}(\eta)}\big)
\lra Y\!\bigg).\EE
\EEnum
\end{dfn}

Since the dimension of every pseudocycle $\Ga\!\in\!L$
to the even-dimensional space~$X$ is even,
Lemma~\ref{fibprodflip_lmm} implies that the oriented morphism
\BE{fbbetadfn_e0}
\fbb_{\eta}\!\equiv \bigg(\!\evb_1\!:(-1)^{k_\bu(\eta)}\M^+_{\eta;J}\!\fiber\!
\big(\! (i\!+\!1,\fb_{\al_{i}(\eta)})_{i\in[k_{\bu}(\eta)]};
(i,\Ga_i)_{\Ga_i\in L_{\bu}(\eta)}\big)\lra Y\!\bigg)\EE
in~\eref{BCprt_e} does not depend on the choice of identification of~$L_{\bu}(\eta)$
with~$[|L_{\bu}(\eta)|]$; 
see the first diagram in Figure~\ref{cSetavt_fig} on page~\pageref{cSetavt_fig}.
By Lemma~\ref{BCpseudo_lmm}, the~map
\BE{fbbdfn_e}
\fbb_{\al'}\!\equiv\bigcup_{\eta\in\cD_{\om}(\al')}\hspace{-.1in}\!\!\fbb_{\eta}\EE
with orientation induced by the OSpin-structure $\os$ is a pseudocycle 
for every \hbox{$\al'\!\in\!\cC_{\om;\al}(Y)\!\cup\!\{\al\}$}. 
If in addition $\dim(\al)\!=\!n\!-\!1$, then $\fbb_{\al}$ is a pseudocycle of codimension~0.
It then has a well-defined degree, and we set
\BE{JSinvdfn_e2}\lr{L}_{\be;K}^{\om,\os}=\deg\fbb_{\al}.\EE
In general, this degree may depend on the choices of $J$, $|K|$ points in~$Y$, 
pseudocycle representatives $\Ga\!\in\!L$ for their homology classes~$[\Ga]_X$ in~$X$, 
and the bordered pseudocycles $(\fb_{\al'})_{\al'\in\cC_{\om;\al}(Y)}$.

\begin{dfn}\label{psisot_dfn}
Suppose $\al_0,\al_1\!\in\!\cC_{\om}(Y)$ and $J_0,J_1\!\in\!\cJ_{\om}$ are generic,
the dimension of every pseudocycle $\Ga\!\in\!L(\al_0)$ is even, and
$(\fb_{0;\al'})_{\al'\in\cC_{\om;\al_0}(Y)}$ and $(\fb_{1;\al'})_{\al'\in\cC_{\om;\al_1}(Y)}$
are bounding chains on $(\al_0,J_0)$ and $(\al_1,J_1)$,  respectively.
Let $\wt\al\!\in\!\wt\cC_{\om}(Y)$ be generic~with 
\BE{prtwtalcond_e}\prt\wt\al=\{1\}\!\times\!\al_1\!-\!\{0\}\!\times\!\al_0\EE
and $\wt{J}\!\equiv\!(J_t)_{t\in[0,1]}$ be a generic path in $\cJ_{\om}$ 
from~$J_0$ to~$J_1$.
A \sf{pseudo-isotopy} on $(\wt\al,\wt{J})$ 
between $(\fb_{0;\al'})_{\al'\in\cC_{\om;\al_0}(Y)}$ and 
$(\fb_{1;\al'})_{\al'\in\cC_{\om;\al_1}(Y)}$ is 
a collection $(\fb_{\wt\al'})_{\wt\al'\in\wt\cC_{\om;\wt\al}(Y)}$ of bordered pseudocycles 
to~$[0,1]\!\times\!Y$ such~that
\BEnum{(PS\arabic*)}

\item\label{isodim_it} $\dim\,\fb_{\wt\al'}\!=\!\dim(\wt\al')\!+\!3$ 
for all $\wt\al'\!\in\!\wt\cC_{\om;\wt\al}(Y)$;

\item\label{isozero_it} $\fb_{\wt\al'}\!=\!\eset$ if $\dim(\wt\al')\!\le\!-2$ 
and $\wt\al'\!\neq\!(0,\{\wt\pt\},\eset)$ for any $\wt\pt\!\in\!\wt{K}(\wt\al)$ or if 
$\dim(\wt\al')\!\ge\!n\!-\!1$;

\item\label{isoi_it} $\fb_{(0,\{\wt\pt\},\eset)}\!=\!\wt\pt$ for all $\wt\pt\!\in\!\wt{K}$;

\item\label{isoprt_it} for all $\al_0'\!\in\!\cC_{\om;\al_0}(Y)$,
$\al_1'\!\in\!\cC_{\om;\al_1}(Y)$, and $\wt\al'\!\in\!\wt\cC_{\om;\wt\al}(Y)$
such~that 
\BE{isoprtit_e} \prt\wt\al'=\{1\}\!\times\!\al_1'\!-\!\{0\}\!\times\!\al_0'\EE
and $-2\!<\!\dim(\wt\al')\!\le\!n\!-\!2$,
\begin{equation*}\begin{split}{}\hspace{-.5in}\prt\fb_{\wt\al'}=&
\bigg(\!\wt\evb_1\!\!:\!\!\!\!\bigcup_{\wt\eta\in\cD_{\om}(\wt\al')}\hspace{-.12in}
(-1)^{\binom{k_{\bu}(\wt\eta)}{2}}\M^+_{\wt\eta;\wt J}\!\fiber\!\!
\big(\!(i\!+\!1,\fb_{\wt\al_i(\wt\eta)})_{i\in[k_\bu(\wt\eta)]};
(i,\wt\Ga_i)_{\wt\Ga_i\in \wt{L}_\bu(\wt\eta)}\big)
\lra [0,1]\!\times\!Y\!\bigg)\\
&+\!\{1\}\!\times\!\fb_{1;\al_1'}\!-\!\{0\}\!\times\!\fb_{0;\al_0'}\,.
\end{split}\end{equation*}

\EEnum
\end{dfn}

\begin{dfn}\label{psisot_dfn2}
Let $\al_0,\al_1$ and $J_0,J_1$ be as in Definition~\ref{psisot_dfn}.
Bounding chains $(\fb_{0;\al'})_{\al'\in\cC_{\om;\al_0}(Y)}$ 
and $(\fb_{1;\al'})_{\al'\in\cC_{\om;\al_1}(Y)}$
on~$(\al_0,J_0)$ and~$(\al_1,J_1)$, respectively, are \sf{pseudo-isotopic} if 
there exist~$\wt{J}$  and~$\wt\al$ as in Definition~\ref{psisot_dfn}
such that $\wt\Ga\!\cap\!Y\!=\!\eset$ for every $\wt\Ga\!\in\!\wt{L}(\wt\al)$
with $\dim\,\wt\Ga\!=\!n$
and a pseudo-isotopy $(\fb_{\wt\al'})_{\al'\in\wt\cC_{\om;\wt\al}(Y)}$
on $(\wt\al,\wt{J})$ 
between $(\fb_{0;\al'})_{\al'\in\cC_{\om;\al_0}(Y)}$ and 
$(\fb_{1;\al'})_{\al'\in\cC_{\om;\al_1}(Y)}$.
\end{dfn}

\vspace{-.1in}

With the notation and setup as in Definition~\ref{psisot_dfn},
the dimension of every pseudocycle $\wt\Ga\!\in\!\wt{L}(\wt\al)$
to the odd-dimensional space $[0,1]\!\times\!X$ is odd.
Along with Lemma~\ref{fibprodflip_lmm}, this implies that the oriented morphism
\BE{wtfbbetadfn_e0}
\fbb_{\wt\eta}\!\equiv 
\bigg(\!\wt\evb_1\!\!:
(-1)^{\binom{k_{\bu}(\wt\eta)}{2}}\M^+_{\wt\eta;\wt J}\!\fiber\!\!
\big(\!(i\!+\!1,\fb_{\wt\al_i(\wt\eta)})_{i\in[k_\bu(\wt\eta)]};
(i,\wt\Ga_i)_{\wt\Ga_i\in \wt{L}_\bu(\wt\eta)}\big)
\lra [0,1]\!\times\!Y\!\bigg)\EE
in~\ref{isoprt_it} does not depend on the choice of identification of~$\wt{L}_{\bu}(\wt\eta)$
with~$[|\wt{L}_{\bu}(\wt\eta)|]$.

Let $\al_0'\!\in\!\cC_{\om;\al_0}(Y)\!\cup\!\{\al_0\}$, 
$\al_1'\!\in\!\cC_{\om;\al_1}(Y)\!\cup\!\{\al_1\}$, and 
$\wt\al'\!\in\!\wt\cC_{\om;\wt\al}(Y)\!\cup\!\{\wt\al\}$ 
be so that~\eref{isoprtit_e} holds
and $\fbb_{0;\al_0'}$ and $\fbb_{1;\al_1'}$
be the pseudocycles as in~\eref{fbbdfn_e} determined by the bounding chains
$(\fb_{0;\al'})_{\al'\in\cC_{\om;\al_0}(Y)}$ and $(\fb_{1;\al'})_{\al'\in\cC_{\om;\al_1}(Y)}$, 
respectively.
By Lemma~\ref{psisot_lmm}, the boundary of the bordered pseudocycle 
\BE{wtfbbdfn_e}\fbb_{\wt\al'}\equiv 
\bigcup_{\wt\eta\in\cD_{\om}(\wt\al')}\hspace{-.15in}\fbb_{\wt\eta}\EE
is $\{0\}\!\times\!\fbb_{0;\al_0'}\!-\!\{1\}\!\times\!\fbb_{1;\al_1'}$.
If in addition $\dim(\al_0)\!=\!n\!-\!1$ (or equivalently $\dim(\al_1)\!=\!n\!-\!1$),
then the above implies~that
$$\deg\fbb_{0;\al_0}=\deg\fbb_{1;\al_1}\,.$$
Thus, pseudo-isotopic bounding chains determine the same numbers~\eref{JSinvdfn_e2}.

Propositions~\ref{bndch_prp} and~\ref{psisot_prp} below are geometric analogues of 
the surjectivity and injectivity statements of \cite[Thm~2]{JS2}.
They guarantee the existence of bounding chains 
and their uniqueness up to pseudo-isotopy under the topological conditions
determined by the authors of~\cite{JS2}.

\begin{prp}\label{bndch_prp}
Let $\al$ and $J$ be as in Definition~\ref{bndch_dfn}.
If $Y$ is an $R$-homology sphere, then
there exists a bounding chain $(\fb_{\al'})_{\al'\in\cC_{\om;\al}(Y)}$
on~$(\al,J)$.
\end{prp}

\begin{prp}\label{psisot_prp}
Let $\al_0,\al_1,\wt\al$, $J_0,J_1,\wt{J}$, and 
$(\fb_{0;\al'})_{\al'\in\cC_{\om;\al_0}(Y)}$ and $(\fb_{1;\al'})_{\al'\in\cC_{\om;\al_1}(Y)}$
be as in Definition~\ref{psisot_dfn} so that $\wt\Ga\!\cap\!Y\!=\!\eset$ 
for every $\wt\Ga\!\in\!\wt{L}(\wt\al)$ with $\dim\,\wt\Ga\!=\!n$.
If $Y$ is an $R$-homology sphere, then there exists a pseudo-isotopy
$(\fb_{\wt\al'})_{\wt\al'\in\wt\cC_{\om;\wt\al}(Y)}$ on $(\wt\al,\wt{J})$ 
between $(\fb_{0;\al'})_{\al'\in\cC_{\om;\al_0}(Y)}$ and 
$(\fb_{1;\al'})_{\al'\in\cC_{\om;\al_1}(Y)}$. 
\end{prp}

\vspace{.2in}

\begin{rmk}\label{transverse_rmk}
By the assumption~\eref{strongpos_e2} and Definition~\ref{bndch_dfn}\ref{BC0_it},
$\fb_{\al'}\!=\!\eset$ if $K(\al'),L(\al')\!=\!\eset$.
Thus, all non-empty bordered pseudocycles~$\fb_{\al_i(\eta)}$ 
in the fiber product~\eref{BCprt_e} are distinct.
This implies that this fiber product is transverse if the bordered pseudocycles~$\fb_{\al''}$
with $\al''\!\prec\!\al'$ are chosen generically.
The same considerations apply to the fiber product in Definition~\ref{psisot_dfn}\ref{isoprt_it}.
\end{rmk}

\subsection{Definitions of open Gromov-Witten invariants}
\label{OpenGWs_subs}

Let $(X,\om,Y)$, $n$, and $\os$ be as before.
Suppose in addition that $\be\!\in\!H_2^{\om}(X,Y)$,
$K\!\subset\!Y$ is a finite subset, and 
$L\!\equiv\!\{\Ga_1,\ldots,\Ga_l\}$ are generic pseudocycles to~$X$ of even dimensions.
The genericity assumptions in particular include that each $\Ga_i$ is traverse to~$Y$
and thus disjoint from~$Y$ if the dimension of~$\Ga_i$ is less than~$n$.
We denote~by
$$[\Ga_i]\in H_*(X;R) ~~\big(\hbox{resp.}~[\Ga_i]\in H_*(X,Y;R)\big)$$ 
the homology class of~$\Ga_i$ if the dimension~$\Ga_i$ is not $n\!-\!1$
(resp.~is $n\!-\!1$). 

Let $\al\!=\!(\be,K,L)$, $J\!\in\!\cJ_{\om}$ be a generic,
and $(\fb_{\al'})_{\al'\in\cC_{\om;\al}(Y)}$ be a bounding chain on~$(\al,J)$.
For $\pt\!\in\!K$ and $\Ga\!\in\!L$, define
$$\al_{\pt}^c=\big(\be,K\!-\!\{\pt\},L\big) \qquad\hbox{and}\qquad
\al_{\Ga}^c=\big(\be,K,L\!-\!\{\Ga\}\big).$$
The bounding chain $(\fb_{\al'})_{\al'\in\cC_{\om;\al_{\pt}^c}(Y)}$ determines a count
$$\blr{L}_{\be;K-\{\pt\}}^{\om,\os}\equiv\deg\fbb_{\al_{\pt}^c}$$
as in~\eref{JSinvdfn_e2}
of $J$-holomorphic multi-disks through $k$~points in~$Y$, 
the pseudocycles~$\Ga_i$, and the auxiliary pseudocycles~$\fb_{\al'}$ 
with $\al'\!\prec\!\al_{\pt}^c$.
As noted after Definition~\ref{psisot_dfn}, this count does not depend on 
the input $(\al_{\pt}^c,J)$ and 
$(\fb_{i;\al'})_{\al'\in\cC_{\om;\al_{\pt}^c}(Y)}$ that differs by a pseudo-isotopy.
Below we provide geometric interpretations of two other versions of such counts. 
In an analogy with Lemma~4.9 in~\cite{JS3},
the three counts agree on the overlaps of the domains of their definitions
in suitable settings; see Theorem~\ref{countinv_thm}\ref{equivdfn_it}.

We denote the signed cardinality of a finite set~$S$ of signed points by~$|S|^{\pm}$. 
If $S$ is not a finite set of signed points, we set $|S|^{\pm}\!\equiv\!0$. 
For $\Ga\!\in\!L$, let
\BE{cDomGadfn_e}
\cD_\om^{\Ga}(\al)=\big\{\eta\!\in\!\cD_\om(\al)\!:\Ga\!\in\!L_{\bu}(\eta)\big\}.\EE
If in addition $\eta\!\in\!\cD_\om(\al)$, let
$$s^*(\eta)\equiv
\begin{cases}\frac{1}{k_\bu(\eta)}\!-\!\frac{1}{2},&
\hbox{if}~k_\bu(\eta)\!\neq\!0,\\
1,&\hbox{if}~k_\bu(\eta)\!=\!0;
\end{cases}\qquad
s^{\circ}(\eta)\equiv
\begin{cases}\frac{1}{k_\bu(\eta)},&
\hbox{if}~k_\bu(\eta)\!\neq\!0,\\
1,&\hbox{if}~k_\bu(\eta)\!=\!0.
\end{cases}$$
For $\wt\eta\!\in\!\cD_\om(\wt\al')$ for some $\wt\al'\!\in\!\wt\cC_{\om}(Y)$,
we define~$s^*(\wt\eta)$ similarly.
Define
\begin{gather}\label{JSinvdfn_e2b}\begin{split}
\lr{L}_{\be;K}^*
\equiv &\sum_{\eta\in\cD_{\om}(\al)}\hspace{-0.25cm}(-1)^{k_\bu(\eta)}s^*(\eta)
\Big|\M_{\eta;J}\!\fiber\!\!\big(\!(i,\fb_{\al_i(\eta)})_{i\in[k_\bu(\eta)]};
(i,\Ga_i)_{\Ga_i\in L_\bu(\eta)}\big)\Big|^{\pm}\\
&\qquad  +\frac12\sum_{\pt\in K}\!\!\blr{\Ga}_{\be;K-\{\pt\}}^{\om,\os}\,;
\end{split}\\
\label{JSinvdfn_e2c}
\blr{L\!-\!\{\Ga\}}_{\be;K}^{\Ga}
\equiv \sum_{\eta\in\cD_{\om}^{\Ga}(\al)}\hspace{-0.25cm}(-1)^{k_\bu(\eta)}
s^{\circ}(\eta)
\Big|\M_{\eta;J}\!\fiber\!\!\big(\!(i,\fb_{\al_i(\eta)})_{i\in[k_\bu(\eta)]};
(i,\Ga_i)_{\Ga_i\in L_\bu(\eta)}\big)\Big|^{\pm}\,.
\end{gather}
Both numbers above vanish unless $\dim(\al)\!=\!0$. 
Let
\BE{qYdfn_e} q_Y\!:H_2(X;\Z)\lra H_2(X,Y;\Z)\EE 
be the natural homomorphism.
By Definition~\ref{psisot_dfn}, a pseudo-isotopy between bounding chains
on pairs~$(\al_0,J_0)$ and~$(\al_1,J_1)$ determines a bijection between
the sets~$L(\al_0)$ and~$L(\al_1)$ of pseudocycles to~$X$.

\begin{thm}\label{countinv_thm}
Suppose $(X,\om)$ is a compact symplectic manifold of real dimension~$2n$ with $n\!\ge\!3$ odd,
$Y\!\subset\!X$ is a compact Lagrangian submanifold, 
$\be\!\in\!H_2(X,Y;\Z)$,  $\os$ is a relative OSpin-structure on~$Y$,
and the conditions~\eref{strongpos_e1} and~\eref{strongpos_e2} hold.
\BEnum{(\arabic*)}

\item\label{countinv_it} Let $\al_0,\al_1$, $J_0,J_1$, and 
$(\fb_{0;\al'})_{\al'\in\cC_{\om;\al_0}(Y)}$ and $(\fb_{1;\al'})_{\al'\in\cC_{\om;\al_1}(Y)}$
be as in Definition~\ref{psisot_dfn}.
If the bounding chains $(\fb_{0;\al'})_{\al'\in\cC_{\om;\al_0}(Y)}$ and 
$(\fb_{1;\al'})_{\al'\in\cC_{\om;\al_1}(Y)}$ are pseudo-homotopic,
then the numbers~\eref{JSinvdfn_e2} associated to them are the same.
If in addition
\BE{nosphbubb_e} K(\al_0)\neq\eset \qquad\hbox{or}\qquad 
\be\not\in\Im\big(q_Y\!:H_2(X;\Z)\!\lra\!H_2(X,Y;\Z)\!\big),\EE
then the numbers~\eref{JSinvdfn_e2b}
(resp.~\eref{JSinvdfn_e2c}) associated to the two bounding chains are also the same.

\item\label{equivdfn_it} Let $\al$, $J$, and $(\fb_{\al'})_{\al'\in\cC_{\om;\al}(Y)}$
be as in Definition~\ref{bndch_dfn}.
If $\pt\!\in\!K(\al)$, then 
\BE{equivdfn_e0a} \lr{L}_{\be;K-\{\pt\}}^{\om,\os}=\lr{L}_{\be;K}^*\,.\EE
If $\Ga\!\in\!L$ and the condition~\eref{nosphbubb_e} with $\al_0\!=\!\al$
is satisfied, then
\BE{equivdfn_e0c} \blr{L\!-\!\{\Ga\}}_{\be;K}^{\Ga}=\lr{L}_{\be;K}^*\,.\EE

\EEnum
\end{thm}

\vspace{.2in}

For $\al$, $J$, and $(\fb_{\al'})_{\al'\in\cC_{\om;\al}(Y)}$
as in Definition~\ref{bndch_dfn}, $\al'\!\in\!\cC_{\om;\al}(Y)$, and 
$\eta\!\in\!\cD_\om(\al')$,
define 
$$\fbb_{\eta;J}^*
=(-1)^{\binom{k_\bu(\eta)}{2}}\M_{\eta;J}\!\fiber\!\!
\big(\!(i,\fb_{\al_i(\eta)})_{i\in[k_\bu(\eta)]};(i,\Ga_i)_{\Ga_i\in L_\bu(\eta)}\big)\,.$$ 
With the assumptions as in Theorem~\ref{countinv_thm}\ref{countinv_it},
let $(\wt\al,\wt{J})$ and 
$(\fb_{\wt\al'})_{\wt\al'\in\wt\cC_{\om;\wt\al}(Y)}$ be as in Definition~\ref{psisot_dfn2}.
For $\wt\al'\!\in\!\wt\cC_{\om;\wt\al}(Y)$ and $\wt\eta\!\in\!\cD_\om(\wt\al')$, define 
$$\fbb_{\wt\eta}^*=(-1)^{\binom{k_\bu(\wt\eta)}{2}}\M_{\wt\eta;\wt{J}}\!\fiber\!\!
\big(\!(i,\fb_{\wt\al_i(\wt\eta)})_{i\in[k_\bu(\wt\eta)]};
(i,\wt\Ga_i)_{\wt\Ga_i\in\wt{L}_\bu(\wt\eta)}\big)\,.$$

\vspace{-.1in}
 
As already noted, the sentence containing~\eref{wtfbbdfn_e} implies the claim
of Theorem~\ref{countinv_thm}\ref{countinv_it} concerning the numbers~\eref{JSinvdfn_e2}.
Suppose $\dim(\wt\al)\!=\!0$.
We show in Section~\ref{countinv_subs} that
\BE{cob_e}
\prt\bigg(\bigsqcup_{\wt\eta\in\cD_\om(\wt\al)}\hspace{-.15in}
s^*(\wt\eta)\fbb_{\wt\eta}^*\bigg)
=\bigsqcup_{\eta_1\in\cD_\om(\al_1)}\hspace{-.2in}
s^*(\eta_1)\big(\{1\}\!\times\!\fbb_{\eta_1;J_1}^*\big)
-\bigsqcup_{\eta_0\in\cD_\om(\al_0)}\hspace{-.2in}
s^*(\eta_0)\big(\{0\}\!\times\!\fbb_{\eta_0;J_0}^*\big)\,.\EE
This implies the claim
of Theorem~\ref{countinv_thm}\ref{countinv_it} concerning the numbers~\eref{JSinvdfn_e2b}.
The claim
of Theorem~\ref{countinv_thm}\ref{countinv_it} concerning the numbers~\eref{JSinvdfn_e2c}
then follows from~\eref{equivdfn_e0c}.
The condition~\eref{nosphbubb_e} precludes sphere bubbling;
it ensures that the stable map compactification of~$\M_{\eta;J}$ contains no 
additional codimension~one boundary for any $\eta\!\in\!\cD_{\om}(\al)$.
The proof of Theorem~\ref{countinv_thm}\ref{equivdfn_it} in Section~\ref{countequiv_subs} 
is a fairly straightforward application of the orientation comparisons for fiber products
collected in Section~\ref{Fp_subs}.

Let $\be\!\in\!H_2^{\om}(X,Y)$, $k\!\in\!\Z^{\ge0}$, and 
$L\!\equiv\!\{\Ga_1,\ldots,\Ga_l\}$ be as above.
If the number~\eref{JSinvdfn_e2} with $|K|\!=\!k\!-\!1$
and the number~\eref{JSinvdfn_e2b} with $|K|\!=\!k$ do not depend 
on the choices of the relevant bounding chains or 
on $K\!\subset\!Y$, we will denote all three numbers~\eref{JSinvdfn_e2}, \eref{JSinvdfn_e2b},
and~\eref{JSinvdfn_e2c} by $\lr{\Ga_1,\ldots,\Ga_l}_{\be;k}^{\om,\os}$;
there is then no ambiguity in this notation by Theorem~\ref{countinv_thm}\ref{equivdfn_it}.
If in addition this number does not depend on generic choices of pseudocycle representatives
for the homology classes~$[\Ga_i]$, let
\BE{OGWdfn_e} 
\blr{[\Ga_1],\ldots,[\Ga_l]}_{\be;k}^{\om,\os}
=\blr{\Ga_1,\ldots,\Ga_l}_{\be;k}^{\om,\os}\in R\,.\EE
In this case, we obtain open GW-invariants as in~\eref{JSinvdfn_e} via
the Poincare and Lefschetz Dualities 
\BE{PDisom_e} \PD_X\!: H_p\big(X;R\big)\stackrel{\approx}{\lra} H^{2n-p}\big(X;R\big),
\quad
\PD_{X,Y}\!: H_p\big(X\!-\!Y;R\big)\stackrel{\approx}{\lra} H^{2n-p}\big(X,Y;R\big);\EE
see Theorems~67.1 and~70.2 in~\cite{Mu2}, for example.
By Propositions~\ref{bndch_prp} and~\ref{psisot_prp}, 
all of the above independence assumptions are satisfied in particular 
if $Y$ is an $R$-homology sphere.
For dimensional reasons, the numbers~\eref{OGWdfn_e} vanish unless 
\BE{dimcond_e} \mu_Y^{\om}(\be)\!+\!n\!-\!3=(n\!-\!1)k\!+\!
\sum_{i=1}^l\big(\codim\,\Ga_i\!-\!2).\EE

\begin{rmk}\label{Zdfn_rmk}
If $(X,\om,Y)$ satisfies~\eref{strongpos_e1} and~\eref{strongpos_e2},
there are no nonzero counts of $J$-holomorphic disks without constraints
due to dimensional reasons.
As shown in the proof of~\eref{equivdfn_e0a} in Section~\ref{countequiv_subs}, 
the signed cardinalities of the $k_{\bu}(\eta)$ fiber products in~\eref{JSinvdfn_e2c}
obtained by circularly permuting the components of $(\al_i(\eta)\!)_{i\in[k_{\bu}(\eta)]}$
are the same.
Therefore, the sum in~\eref{JSinvdfn_e2c} can be re-written without 
$s_{\bu}(\eta)\!=\!1/k_{\bu}(\eta)$.
Along with Theorem~\ref{countinv_thm}\ref{equivdfn_it},  
this implies that \eref{JSinvdfn_e2}, \eref{JSinvdfn_e2b}, and~\eref{JSinvdfn_e2c} 
provide counts  of $J$-holomorphic disks in $(X,\om,Y)$ 
with coefficients in any commutative ring~$R$ with unity
under the assumptions~\eref{strongpos_e1} and~\eref{strongpos_e2}.
\end{rmk}

\subsection{Properties of open Gromov-Witten invariants}
\label{OWGprp_subs}

Let $(X,\om)$ be a compact symplectic manifold of real dimension~$2n$,
$Y\!\subset\!X$ be a connected compact Lagrangian submanifold,
and $\os$ be a relative OSpin-structure on~$Y$.
Denote by \hbox{$[Y]_X\!\in\!H_n(X;R)$} the image of the fundamental class
of~$Y$ with respect to the orientation determined by~$\os$. 
The kernel of the homomorphism
$$H_{n-1}(X\!-\!Y;R)\lra H_{n-1}(X;R)$$
is generated by the homology class~$[S(\cN_yY)]$ of
a unit sphere $S(\cN_yY)$ in the fiber of~$\cN Y$ over any $y\!\in\!Y$.
We orient~$S(\cN_yY)$ as in \cite[Sec~2.5]{RealWDVV3} and
denote the image of~$[S(\cN_yY)]$  under the $p\!=\!n\!-\!1$ case of 
the Lefschetz Duality isomorphism~\eref{PDisom_e} by~$\eta_{X,Y}^{\circ}$.
For $B\!\in\!H_2(X;\Z)$, let
$$\lr{\cdot,\ldots,\cdot}_B^{\om}\!: \bigoplus_{l=1}^{\i}H^{2*}(X;R)^{\oplus l}\lra R$$ 
be the standard GW-invariants of~$(X,\om)$.

The properties of the open GW-invariants~\eref{geomJSinvdfn_e}
stated below are as in Theorem~4 of~\cite{JS2} and Corollary~1.5 and Theorem~6 of~\cite{JS3}.
The first four of them are the direct analogues of standard properties
of the closed GW-invariants.
The fifth property, called {\it Wall crossing} in~\cite{JS3}, is 
the direct generalization of Proposition~2.1 in~\cite{RealWDVV3}.
The two remaining properties describe new geometric phenomena 
discovered in~\cite{JS3}.

%Suppose $Y$ is a rational homology sphere and thus 
%\BE{Hpisom_e} H_p(X\!-\!Y;\Q)\approx H_p(X;\Q) \qquad\forall~p\in2\Z\!-\!\{0,n\!-\!1\};\EE 
%the isomorphism above is induced by the inclusion. 

\begin{thm}\label{main_thm}
Let $(X,\om)$ be a compact symplectic $2n$-fold with $n\!\ge\!3$ odd,
$Y$ be a connected compact Lagrangian submanifold, $\os$ be a relative OSpin-structure on~$Y$,
$k,l\!\in\!\Z^{\ge0}$, and \hbox{$\be\!\in\!H_2^{\om}(X,Y)$} with 
\BE{nosphbubb_e2} k\neq\eset \qquad\hbox{or}\qquad 
\be\not\in\Im\big(q_Y\!:H_2(X;\Z)\!\lra\!H_2(X,Y;\Z)\!\big).\EE
Suppose the assumptions~\eref{strongpos_e1} and~\eref{strongpos_e2} are satisfied
and
the numbers~\eref{JSinvdfn_e2b} with $|K|\!=\!k$
do not depend on the choices of the relevant bounding chains, $K\!\subset\!Y$,
or generic choices of pseudocycle representatives $\Ga_i\!\in\!L$
in their homology classes~$[\Ga_i]$.
The open GW-invariants~\eref{OGWdfn_e} then determine 
symmetric multilinear functionals
\BE{geomJSinvdfn_e}
\lr{\cdot,\ldots,\cdot}_{\be,k}^{\om,\os}\!: 
\bigoplus_{l=0}^{\i}\wh{H}^{2*}(X,Y;R)^{\oplus l}\lra R\EE 
with the following properties.
\BEnum{(OGW\arabic*)}

\item\label{dimvan_it} $\lr{\ga_1,\ldots,\ga_l}_{\be,k}^{\om,\os}\!=\!0$ unless 
\eref{dimcond_e} with $\codim\,\Ga_i$ replaced by $\deg\ga_i$ holds.

\item\label{deg0_it} If $\be\!=\!0$, $\displaystyle\lr{\ga_1,\ldots,\ga_l}_{\be,k}^{\om,\os}=
\begin{cases}
\lr{\ga_1,\pt},&\hbox{if}~(k,l)\!=\!(1,1);\\
0,&\hbox{otherwise}.\end{cases}$

\item\label{ins1_it} $\displaystyle\lr{1,\ga_2,\ldots,\ga_l}_{\be,k}^{\om,\os}=
\begin{cases}
1,&\hbox{if}~(\be,k,l)\!=\!(0,1,1);\\
0,&\hbox{otherwise}.\end{cases}$

\item\label{div_it} If $\ga_0\!\in\!H^2(X,Y;R)$, 
$\displaystyle\lr{\ga_0|_X,\ga_1,\ldots,\ga_l}_{\be,k}^{\om,\os}=\lr{\ga_0,\be}
\lr{\ga_1,\ldots,\ga_l}_{\be,k}^{\om,\os}$.

\item\label{sphere_it} $\displaystyle\lr{\ga_{X,Y}^{\circ},\ga_1,\ldots,\ga_l}_{\be,k}^{\om,\os}=
\lr{\ga_1,\ldots,\ga_l}_{\be,k+1}^{\om,\os}$.

\item\label{lagl_it1} 
If $k\!=\!1$ and $\ga_0\!\in\!H^n(X;R)$,
$${}\hspace{-.8in}
\lr{\ga_0,[Y]_X}\blr{\ga_1,\ldots,\ga_l}_{\be,k}^{\om,\os}
=\sum_{B\in q_Y^{-1}(\be)}\!\!\!\!\!\!\!(-1)^{\lr{w_2(\os),B}}
\blr{\PD_X\big([Y]_X\big),\ga_0,\ga_1|_X,\ldots,\ga_l|_X}_B^{\om}\,.$$ 

\item\label{lagl_it2} 
If $[Y]_X\!\neq\!0$ and $k\!\ge\!2$,
then $\lr{\ga_1,\ldots,\ga_l}_{\be,k}^{\om,\os}\!=\!0$.

\EEnum
\end{thm}

\vspace{.1in}

The vanishing property~\ref{dimvan_it}  holds because the dimensions of 
the relevant moduli spaces and the constraints are different unless 
\eref{dimcond_e} with $\codim\,\Ga_i$ replaced by $\deg\ga_i$ holds. 
The symmetry property of the open GW-invariants~\eref{geomJSinvdfn_e} is immediate
from the fiber products in~\eref{JSinvdfn_e2}, \eref{JSinvdfn_e2b}, and~\eref{JSinvdfn_e2b}
being independent of the order of the elements $\Ga_1,\ldots,\Ga_l$ of~$L$.
Both properties apply to the counts~\eref{JSinvdfn_e2}, \eref{JSinvdfn_e2b}, 
and~\eref{JSinvdfn_e2c} without any assumptions on 
the independence of these counts of the choice of the bounding chain.

We establish the remaining properties of the GW-invariants~\eref{geomJSinvdfn_e}
stated in Theorem~\ref{main_thm} in Section~\ref{GWmainpf_sec} by showing that 
the stated properties are satisfied by the numbers~\eref{OGWdfn_e}.
We note which of the many assumptions of Theorem~\ref{main_thm} are actually necessary 
for each given property to be satisfied by the numbers~\eref{JSinvdfn_e2}, 
\eref{JSinvdfn_e2b}, and~\eref{JSinvdfn_e2c}.

\subsection{WDVV-type relations}
\label{OpenWDVV_subs}

We now translate the statements of the WDVV-type equations of Theorem~3 of~\cite{JS3}
to relations for the open GW-invariants~\eref{geomJSinvdfn_e} 
under the assumption that $R$ is a field.
 %and the numbers~\eref{JSinvdfn_e2} and~\eref{JSinvdfn_e2b} 
%with~$|K|$ fixed do not depend on the choices of the relevant bounding chains, $K\!\subset\!Y$,
%or generic choices of pseudocycle representatives $\Ga_i\!\in\!L$
%in their homology classes~$[\Ga_i]$.

Let $(X,\om)$ be a connected compact symplectic manifold and
$Y\!\subset\!X$ be a connected compact Lagrangian submanifold.
Define 
$$\La_Y^{\om}\equiv
\big\{(\Psi\!:H_2^{\om}(X,Y)\!\lra\!R)\!:
\big|\{B\!\in\!H_2^{\om}(X,Y)\!:
\Psi(B)\!\neq\!0,\,\om(B)\!<\!E\}\big|\!<\!\i~\forall\,E\!\in\!\R\big\}.$$
We write an element $\Psi$ of $\La_Y^{\om}$ as 
$$\Psi\equiv\sum_{B\in H_2^{\om}(X,Y)}\!\!\!\!\!\!\!\!\Psi(B)q^B$$
and multiply two such elements as powers series in $q$ with the exponents in $H_2^{\om}(X,Y)$.

Since $\dim\,Y\!=\!n$, the cohomology long exact sequence for the pair~$(X,Y)$
implies that the restriction homomorphism
$$\wh{H}^{2*}(X,Y;R)\lra H^{2*}(X;R)$$
is surjective.
Let 
$$\ga_1^{\st}\!\equiv\!1,\ga_2^{\st},\ldots,\ga_N^{\st}\in\wh{H}^{2*}(X,Y;R)$$
be homogeneous elements such that $\ga_1^{\st},\ga_2^{\st}|_X,\ldots,\ga_N^{\st}|_X$ 
is a basis for $H^{2*}(X;R)$,
$(g_{ij})_{i,j}$ be the $N\!\times\!N$-matrix given~by
\BE{giijdfn_e}g_{ij}=\blr{\ga_i^{\st}\ga_j^{\st},[X]}\EE
and $(g^{ij})_{i,j}$ be its inverse.
Let $\Ga_1^{\st},\Ga_2^{\st},\ldots,\Ga_N^{\st}$ be generic pseudocycles to~$X$
representing the Poincare duals of $\ga_2^{\st},\ldots,\ga_N^{\st}$.
For a tuple $\bt\!\equiv\!(t_1,\ldots,t_N)$ of formal variables, let
$$\ga_{\bt}^{\st}=\ga_1^{\st}t_1\!+\!\ldots\!+\!\ga_N^{\st}t_N\,.$$

For a finite set $L$, $B\!\in\!H_2(X;\Z)$, and an $\om$-tame almost complex 
structure~$J$, we denote~by $\M^\C_L(B;J)$
the moduli space of stable $J$-holomorphic degree~$B$ maps with marked points
indexed by the set~$L$.
It carries a canonical orientation.
For each $i\!\in\!L$, let 
$$\ev_i\!:\M^\C_L(B;J)\lra X$$
be the evaluation morphism at the $i$-th marked point.
If in addition $\Ga_1,\ldots,\Ga_l$ are maps to~$X$, let
$$\M^\C_{0\sqcup[l]}(B;J)\!\fiber\!\big(\!(i,\Ga_i)_{i\in[l]}\big)
\equiv \M^\C_{0\sqcup[l]}(B;J) _{(\ev_1,\ldots,\ev_l)}\!\!\times\!\!
_{\Ga_1\times\ldots\times\Ga_l}\!\big(\!(\dom\,\Ga_1)\!\times\!\ldots\!\times\!(\dom\,\Ga_l)\!\big).$$
If $J$ is generic and $\Ga_1,\ldots,\Ga_l$ are pseudocycles in general position, then
$$f^\C_{B,(\Ga_i)_{i\in[l]}}\equiv
\Big(\ev_0\!:\M^\C_{0\sqcup[l]}(B;J)\!\!\fiber\!\!\big(\!(i,\Ga_i)_{i\in[l]}\big)\lra X\Big)$$
is a pseudocycle of dimension 
$$\dim\,f^\C_{B,(\Ga_i)_{i\in[l]}}
=\mu_{\om}\big(q_Y(B)\!\big)\!+\!2(n\!-\!2)
\!-\!\sum_{i=1}^l\!\big(\codim\,\Ga_i\!-\!2\big)$$
transverse to~$Y$.
With $\ga_i\!=\!\PD_X([\Ga_i])$,
let $(\la_{B,(\ga_i)_{i\in[l]}}^j)_{j\in[N]}\!\in\!R^N$ be such that
$$\big[f^\C_{B,(\Ga_i)_{i\in[l]}}\big]
=\sum_{j=1}^N\la_{B,(\ga_i)_{i\in[l]}}^j
\PD_X\big(\ga_j^{\st}|_X\big)\in H_*(X;R);$$
the tuple $(\la_{B,(\ga_i)_{i\in[l]}}^j)_{j\in[N]}$ depends
only on $B$, $\ga_1,\ldots,\ga_l$, and $\ga_2^{\st},\ldots,\ga_N^{\st}$.

Suppose in addition $\os$ is a relative OSpin-structure.
If $[Y]_X\!=\!0$ and $\ga$ is an $(n\!-\!1)$-dimensional pseudocycle to~$X\!-\!Y$ 
bounding a pseudocycle~$\Ga$ to~$X$ transverse to~$Y$, we define
$$\lk_{\os}(\ga)\equiv\big|\Ga\!\fiber\!\io_Y\big|^{\pm}\,,$$
where $\io_Y\!:Y\!\lra\!X$ is the inclusion;
see Section~\ref{Fp_subs} for the sign conventions for fiber products.
This \sf{linking number} of~$\ga$ and~$Y$ with the orientation determined
by the relative OSpin-structure~$\os$ does not depend on the choice of~$\Ga$.
We set $\lk_{\os}(\ga)\!=\!0$ if $\ga$ is not an $(n\!-\!1)$-dimensional pseudocycle.

For the purpose of WDVV-type equations for the open GW-invariants~\eref{geomJSinvdfn_e},
we extend these signed disk counts to the pairs $(k,\be)$ not 
satisfying~\eref{nosphbubb_e2},
i.e.~$k\!=\!0$ and $\be\!\in\!H_2^{\om}(X,Y)$ is in the image of 
the homomorphism~$q_Y$ in~\eref{qYdfn_e},  as follows.
Let $\ga_1,\ldots,\ga_l\!\in\!\wh{H}^{2*}(X,Y;R)$.
If $[Y]_X\!\neq\!0$, we~define
$$\blr{\ga_1,\ldots,\ga_l}_{\be,0}^{\om,\os}=0.$$
Suppose next that $[Y]_X\!=\!0$.
Let $\Ga_1,\ldots,\Ga_l$ be generic pseudocycles to~$X$ 
representing the Poincare duals of~$\ga_1,\ldots,\ga_l$.
Define
\begin{equation*}\begin{split}
\blr{\ga_1,\ldots,\ga_l}_{\be,0}^{\om,\os}=&
\hbox{RHS of~\eref{JSinvdfn_e2b} with $\al\!=\!\big(\be,\eset,\{\Ga_1,\ldots,\Ga_l\}\big)$}\\
&+\sum_{B\in q_Y^{-1}(\be)}\!\!\!\!\!\!(-1)^{\blr{w_2(\os),B}}
\lk_{\os}\!\Big(f^\C_{B,(\Ga_i)_{i\in[l]}}-
\sum_{j=1}^N\!\la_{B,(\ga_i)_{i\in[l]}}^j\Ga_j^{\st}\Big)
\end{split}\end{equation*}
in this case.
This number depends on the span of the chosen elements $\ga_i^{\st}$ of $H^{n+1}(X,Y;R)$.
By the proof of~\eref{cob_e}, 
pseudo-isotopic bounding chains $(\fb_{0;\al'})_{\al'\in\cC_{\om;\al_0}(Y)}$ 
and $(\fb_{1;\al'})_{\al'\in\cC_{\om;\al_1}(Y)}$
determine the same numbers $\lr{\ga_1,\ldots,\ga_l}_{\be,0}^{\om,\os}$.
%For example,
%$$\blr{\ga_1,\ga_2}_{0,0}^{\om,\os}=
%\lk_{\os}\!\Big(\Ga_1\!\cap\!\Ga_2\!-\!
%\sum_{j=1}^N\!\la_{\ga_1\ga_2}^j\Ga_j^{\st}\Big), 
%\quad\hbox{where}~~
%\ga_1\ga_2\equiv\sum_{j=1}^N\la_{\ga_1\ga_2}^j\ga_j^{\st}|_X\in H^*(X;R).$$

%\vspace{-.1in}

We define $\Phi_{\om}^{\os}\!\in\!\La_Y^{\om}[[t_1,\ldots,t_N]]$
and $\Om_{\om}^{\os}\!\in\!\La_Y^{\om}[[s,t_1,\ldots,t_N]]$ by
\begin{gather*}
\Phi_{\om}^{\os}(t_1,\ldots,t_N)=
\sum_{\begin{subarray}{c}\be\in H_2^{\om}(X,Y)\\ l\in\Z^{\ge0}\end{subarray}}\!\!\!
\Bigg(\sum_{\begin{subarray}{c}B\in H_2(X;\Z)\\
q_Y(B)=\be\end{subarray}}\!\!\!\!\!\!\!(-1)^{\lr{w_2(\os),B}}
\blr{\underset{l}{\underbrace{\ga_{\bt}^{\st}|_X,\ldots,\ga_{\bt}^{\st}|_X}}}_B^{\om}\!\Bigg)
\frac{q^{\be}}{l!}\,,\\
\Om_{\om}^{\os}(s,t_1,\ldots,t_N)=
\sum_{\begin{subarray}{c}\be\in H_2^{\om}(X,Y)\\ k,l\in\Z^{\ge0}\end{subarray}}\!\!\!
\blr{
\underset{l}{\underbrace{\ga_{\bt}^{\st},\ldots,\ga_{\bt}^{\st}}}}_{B,k}^{\om,\os}
\frac{q^{\be}s^k}{k!l!}\,.
\end{gather*}
By Gromov's Compactness Theorem,
the inner sum in the definition of~$\Phi_{\om}^{\os}$ has finitely nonzero terms.
For the same reason, the coefficients of the powers of $t_1,\ldots,t_N,u$ in
$\Phi_{\om}^{\os}$ and $\Om_{\om}^{\os}$ lie in~$\La_Y^{\om}$.

\begin{thm}\label{OpenWDVV_thm}
Suppose $R$ is a field, 
$(X,\om,Y)$ and~$\os$ are as in Theorem~\ref{main_thm} with~$X$ and~$Y$ connected, 
and the independence assumptions of Theorem~\ref{main_thm} are satisfied
by the numbers~\eref{JSinvdfn_e2} even if the condition~\eref{nosphbubb_e2} does not~hold.
For all \hbox{$u\!\in\!\{t_1,\ldots,t_N\}$} and 
\hbox{$v,w\!\in\!\{s,t_1,\ldots,t_N\}$},
\BE{OpenWDVV_e} \begin{split}
&\sum_{1\le i,j\le N}\!\!\!\!\!\!\!\big(\prt_u\prt_v\prt_{t_i}\Phi_{\om}^{\os}\big)g^{ij}
\big(\prt_w\prt_{t_j}\Om_{\om}^{\os}\big)+
\big(\prt_u\prt_v\Om_{\om}^{\os}\big)\!
\big(\prt_s\prt_w\Om_{\om}^{\os}\big)\\
&\hspace{1.4in}=
\sum_{1\le i,j\le N}\!\!\!\!\!\!\!\big(\prt_u\prt_w\prt_{t_i}\Phi_{\om}^{\os}\big)g^{ij}
\big(\prt_v\prt_{t_j}\Om_{\om}^{\os}\big)+
\big(\prt_u\prt_w\Om_{\om}^{\os}\big)\!\big(\prt_s\prt_v\Om_{\om}^{\os}\big).
\end{split}\EE
\end{thm}

By Propositions~\ref{bndch_prp} and~\ref{psisot_prp},  
the above independence assumptions are satisfied if $Y$ is an $R$-homology sphere. 
Theorem~\ref{OpenWDVV_thm} is mostly a translation of Theorem~3 of~\cite{JS3} 
to the geometrically defined invariants of Theorem~\ref{main_thm}.
The framework of lifting bordisms from the Deligne-Mumford moduli spaces of stable curves 
to the moduli spaces of stable maps as in~\cite{RealWDVV,RealWDVV3} 
can be used for 
a self-contained geometric analogue of the proof in~\cite{JS3} establishing 
relations between the disk counts~\eref{JSinvdfn_e2b} which arise from a fixed bounding
chain, without any independence assumptions of Theorem~\ref{OpenWDVV_thm}.
The independence assumptions are used to present these relations 
succinctly as  the partial differential equations~\eref{OpenWDVV_e}.
%ensure that the disk counts determined by 
%the bounding chains $\al'\!\prec\!\al$ do not depend on the choice of 
%the subset $K(\al')\!\subset\!K(\al)$ of a fixed cardinality.
%This is 
This geometric analogue applies over an arbitrary field~$R$ and allows 
taking~$v$ or~$w$ to be~$s$ even if $[Y]_X\!\neq\!0$;
this case is excluded from the statement of Corollary~1.6 of~\cite{JS3}.
%It can also be applied to the disk counts arising from a bounding chain
%$(\fb_{\al'})_{\al'\in\cC_{\om;\al}(Y)}$ with 
%$\al\!\equiv\!(\be,K,L)$ under the assumption
%that the disk counts determined by the bounding chains $\al'\!\prec\!\al$
%do not depend on the choice of the subset $K(\al')\!\subset\!K(\al)$
%of a fixed cardinality.
 
For $\be\!\in\!H_2(X,Y;\Z)$ and $l\!\in\!\Z^{\ge0}$, let
\begin{equation*}\begin{split}
\cP_{\C}(\be)&=\big\{(B_1,B_2)\!\in\!H_2(X;\Z)\!\oplus\!H_2(X;\Z)\!:
q_Y(B_1\!+\!B_2)\!=\!\be\big\},\\
\cP_{12;}(l)&=\big\{(I,J)\!\in\!\cP(l)\!:
\{1,2,\ldots,l\}\!=\!I\!\sqcup\!J,~1,2\!\in\!I\big\},\\
\cP_{1;2}(l)&=\big\{(I,J)\!\in\!\cP(l)\!:
\{1,2,\ldots,l\}\!=\!I\!\sqcup\!J,~1\!\in\!I,~2\!\in\!J\big\}.\\
\end{split}\end{equation*}
For a tuple $\ga\!\equiv\!(\ga_1,\ldots,\ga_l)$ of elements of $H^{2*}(X;R)$ and
$I\!\subset\!\{1,2,\ldots,l\}$, we denote by~$\ga_I$ 
the $|I|$-tuple consisting of the entries of~$\ga$ indexed by~$I$.
Let  $\ga_1^{\st},\ldots,\ga_N^{\st}$ be basis for $H^{2*}(X;R)$,
$(g_{ij})_{i,j}$ be the $N\!\times\!N$-matrix given by~\eref{giijdfn_e},
and $(g^{ij})_{i,j}$ be its inverse.

Corollary~\ref{OpenWDVV_crl} below for the standard (closed) GW-invariants of~$(X,\om)$
follows immediately from the case of Theorem~\ref{OpenWDVV_thm}
with 
$$[Y]_X\neq0, \qquad u,v\in\big\{t_1,\ldots,t_N\big\}, \qquad\hbox{and}\qquad w=s,$$ 
\ref{lagl_it1} and~\ref{lagl_it2} in Theorem~\ref{main_thm}, 
\BEnum{(\alph*)}

\item Propositions~\ref{bndch_prp} and~\ref{psisot_prp} above, and 

\item Theorem~1.1 in~\cite{JakeSaraWel};

\EEnum 
see also Section~1.2 in~\cite{JakeSaraWel}
concerning the second case in Corollary~\ref{OpenWDVV_crl}.
A slightly weaker version of Corollary~\ref{OpenWDVV_crl} follows from
the case of Theorem~\ref{OpenWDVV_thm} with  \hbox{$u,v,w\!\in\!\{t_1,\ldots,t_N\}$},
which is available in~\cite{JS3} for $R\!=\!\R$.

\begin{crl}\label{OpenWDVV_crl}
Suppose $R$ is a field, $(X,\om)$ is a compact symplectic $2n$-fold,
$Y\!\subset\!X$ is an oriented connected compact Lagrangian submanifold,
and $\ga_0\!\in\!H^n(X;R)$ with \hbox{$\lr{\ga_0,[Y]_X}\!=\!1$}.
If either
\BEnum{(\alph*)}

\item $n\!\ge\!3$ is odd and $Y$ is an $R$-homology sphere or

\item $n\!=\!3$, the homomorphism $H_1(Y;R)\!\lra\!H_1(X;R)$ induced by
the inclusion $Y\!\lra\!X$ is injective, and 
the homomorphism $H_2(Y;R)\!\lra\!H_2(X;R)$ is trivial,

\EEnum
then
\begin{equation*}\begin{split}
&\sum_{\begin{subarray}{c}(B_1,B_2)\in\cP_{\C}(\be)\\ 
(I,J)\in\cP_{12;}(l)\end{subarray}}
\sum_{i,j\in[N]}\!\!\!\!
\blr{\ga_I,\ga^\st_i}^{\!\om}_{\!B_1}g^{ij}
\blr{\ga^\st_j,\PD_X\big([Y]_X\big),\ga_0,\ga_J}^{\!\om}_{\!B_2}\\
&\hspace{1in}=\sum_{\begin{subarray}{c}(B_1,B_2)\in\cP_{\C}(\be)\\
(I,J)\in\cP_{1;2}(l)\end{subarray}}\hspace{-.3in}
\blr{\PD_X\big([Y]_X\big),\ga_0,\ga_I}^{\!\om}_{\!B_1}
\blr{\PD_X\big([Y]_X\big),\ga_0,\ga_J}^{\!\om}_{\!B_2}\,
\end{split}\end{equation*}
for all $\be\!\in\!H_2(X,Y;\Z)$ and $\ga_1,\ldots,\ga_l\!\in\!H^{2*}(X;R)$.
\end{crl}

\section{Proofs of Propositions~\ref{bndch_prp} and~\ref{psisot_prp} 
and Theorem~\ref{countinv_thm}}
\label{mainpf}

We establish Propositions~\ref{bndch_prp} and~\ref{psisot_prp}, which guarantee that 
the open invariants~\eref{geomJSinvdfn_e} can actually be constructed 
via~\eref{JSinvdfn_e2}, \eref{JSinvdfn_e2b}, and~\eref{JSinvdfn_e2c} 
at least under some topological assumptions
on the Lagrangian submanifold~$Y$ of~$(X,\om)$, in Section~\ref{bcisot_subs}.
In Section~\ref{countinv_subs}, we show that bounding chains differing 
by a pseudo-isotopy determine the same counts~\eref{JSinvdfn_e2b}.
In Section~\ref{countequiv_subs}, we establish the equivalence of the three
definitions of the disk counts in~$(X,Y)$  as stated in
Theorem~\ref{countinv_thm}\ref{equivdfn_it}.

\subsection{Existence of bounding chains and pseudo-isotopies}
\label{bcisot_subs}

The main steps in the inductive proofs of Propositions~\ref{bndch_prp} and~\ref{psisot_prp}
are Lemmas~\ref{BCpseudo_lmm} and~\ref{psisot_lmm}, respectively, below.
They ensure that the right-hand sides of the identities in
Definitions~\ref{bndch_dfn}\ref{BCprt_it} and~\ref{psisot_dfn}\ref{isoprt_it} 
are closed pseudocycles of the required dimensions 
if the bordered pseudocycles~$\fb_{\al'}$ with $\al'\!\prec\!\al$
and~$\fb_{\wt\al'}$ with $\wt\al'\!\prec\!\wt\al$
satisfy all conditions of Definitions~\ref{bndch_dfn} and~\ref{psisot_dfn}.
Thus, these right-hand sides satisfy at least a necessary 
condition for the existence of bordered pseudocycles~$\fb_{\al'}$ and~$\fb_{\wt\al'}$.

\begin{lmm}\label{BCpseudo_lmm}
Let $\al$ and $J$ be as in Definition~\ref{bndch_dfn}.
If $(\fb_{\al'})_{\al'\in\cC_{\om;\al}(Y)}$ is a bounding chain
on~$(\al,J)$, then
the map~$\fbb_{\al}$ in~\eref{fbbdfn_e} 
is a pseudocycle of dimension~$\dim(\al)\!+\!1$.
\end{lmm}

\begin{proof} For $\eta\!\in\!\cD_{\om}(\al)$, let $\fbb_{\eta}$ be as in~\eref{fbbetadfn_e0}.
By Definition~\ref{bndch_dfn}\ref{BCdim_it} with~$\al'$ replaced by $\al_i(\eta)\!\prec\!\al$,
\BE{BCpseudo_e3}\begin{split}
\dim\,\fbb_{\eta}=\,&\big(\mu_Y^{\om}(\be_{\bu}(\eta)\!)\!+\!(n\!-\!3)\!+\!(k_{\bu}(\eta)\!+\!1)
\!+\!2|L_{\bu}(\eta)|\big)\\
&+\!\!\sum_{i=1}^{k_{\bu}(\eta)}\!\!\!
\big(\dim(\al_i(\eta)\!)\!+\!2\big)
\!-\!nk_{\bu}(\eta)\!-\!\!\!\sum_{\Ga\in L_{\bu}(\eta)}\!\!\!\!\!\!\codim\,\Ga
=\dim(\al)\!+\!1\,.
\end{split}\EE
Thus, the dimension of~$\fbb_{\al}$ is $\dim(\al)\!+\!1$.

We define
\BE{cDonm2dfn_e}\begin{split}
\cD_{\om}^2(\al)=\Big\{\!\big(\eta;\be_{\bu}^2,k_1,k_2,L_{\bu}^2\big)\!\!:\,
&\eta\!\in\!\cD_{\om}(\al),\,\be_{\bu}^2\!\in\!H_2^{\om}(X,Y),\,
k_1,k_2\!\in\![k_\bu(\eta)\!+\!2],\,L_{\bu}^2\!\subset\!L_{\bu}(\eta),\\
&\be_{\bu}(\eta)\!-\!\be_{\bu}^2\!\in\!H_2^{\om}(X,Y),\,k_1\!<\!k_2,\\
&(\be_{\bu}^2,k_2\!-\!1\!-\!k_1,L_{\bu}^2)\!\neq\!(0,0,\eset),(0,1,\eset),
(\be_{\bu}(\eta),k_{\bu}(\eta),L_{\bu}(\eta)\!)\!\Big\}.
\end{split}\EE
For an element $(\eta;\be_{\bu}^2,k_1,k_2,L_{\bu}^2)$ of $\cD_{\om}^2(\al)$, let
\begin{gather*}
K=[k_2\!-\!1]\!-\![k_1], ~ k_{\bu}^1=k_{\bu}(\eta)\!-\!|K|\!+\!1, 
~ k_{\bu}^2=|K|,  ~
K^2=\bigsqcup_{i\in K}\!K_{i-1}(\eta), ~
L^2=L_{\bu}^2\!\sqcup\!\bigsqcup_{i\in K}\!L_{i-1}(\eta),\\
\begin{split}
\be^2&=\be_{\bu}^2\!+\!\sum_{i\in K}\be_{i-1}(\eta),\\
\al^2&=\big(\be^2,K^2,L^2\big),
\end{split}
\quad
\al_i^1=\begin{cases}
\al_i(\eta),&\hbox{if}~i\!\in\![k_1\!-\!1];\\
\al^2,&\hbox{if}~i\!=\!k_1;\\
\al_{i-2+k_2-k_1}(\eta),&\hbox{if}~i\!\in\![k_{\bu}^1]\!-\![k_1].
\end{cases}
\end{gather*}
In particular, $\al^2\!\in\!\cC_{\om;\al}(Y)$,
\begin{equation*}\begin{split}
\eta^1&\equiv\big(\be_{\bu}(\eta)\!-\!\be_{\bu}^2,k_{\bu}^1,
L_{\bu}(\eta)\!-\!L_{\bu}^2,(\al_i^1)_{i\in[k_{\bu}^1]}\big)
\in\cD_{\om}(\al),\\
\eta^2&\equiv\big(\be_{\bu}^2,k_{\bu}^2,L_{\bu}^2,(\al_{i+k_1-1})_{i\in[k_{\bu}^2]}\big)
\in\cD_{\om}(\al^2).
\end{split}\end{equation*}
We note that the resulting map
\BE{BCpseudo_e11}\begin{split}
&\cD_{\om}^2(\al)\lra \ov\cD_{\om}^2(\al)\!\equiv\!\big\{(\eta;i,\eta')\!:\eta\!\in\!\cD_{\om}(\al),\,
i\!\in\![k_{\bu}(\eta)],\, \eta'\!\in\!\cD_{\om}(\al_i(\eta)\!)\big\},\\
&\hspace{1in} (\eta;\be_{\bu}^2,k_1,k_2,L_{\bu}^2)\lra(\eta^1;k_1,\eta^2),
\end{split}\EE
is bijective.
%Let $\cD_{\om}^{2*}(\al)\!\subset\!\cD_{\om}^2(\al)$ be the preimage of 
%the triples $(\eta;i,\eta')\!\in\!\ov\cD_{\om}^2(\al)$ such that 
%$\eta\!\in\!\cD_{\om}^*(\al)$,
%and \hbox{$\eta'\!\in\!\cD_{\om}^*(\al_i(\eta)\!)$}.
%If $(\eta;\vt)\!\in\!\cD_{\om}^{2*}(\al)$, then $\eta\!\in\!\cD_{\om}^*(\al)$.

For $(\eta;i,\eta')\!\in\!\cD_{\om}^2(\al)$, define 
$$\fbb_{\eta}(i,\eta')=
\bigg(\!\evb_1\!\!:\M^+_{\eta;J}\!\fiber\!
\big(\!(j\!+\!1,\fb_{\al_{j}(\eta)})_{j\in[k_{\bu}(\eta)]-\{i\}},
(i\!+\!1,\fbb_{\eta'});(j,\Ga_j)_{\Ga_j\in L_{\bu}(\eta)}\big)\lra Y\!\bigg).$$
For $(\eta;\vt)\!\in\!\cD_{\om}^2(\al)$ with 
$\vt\!\equiv\!(\be_{\bu}^2,k_1,k_2,L_{\bu}^2)$, let
\BE{BCpseudo_e12}\cS_{\eta}^+(\vt)\subset \prt\M^+_{\eta;J}\EE
be the subspace consisting of $J$-holomorphic maps from $(\D^2\!\v\!\D^2,S^1\!\v\!S^1)$
to $(X,Y)$ with the second component of degree~$\be_{\bu}^2$
and carrying the boundary marked points indexed by $[k_2\!-\!1]\!-\![k_1]$ 
and the interior marked points indexed by~$L_{\bu}^2$.
These subspaces are the topological components of $\prt\M^+_{\eta;J}$
and thus inherit orientations from the orientation of~$\M^+_{\eta;J}$.
Let 
$$\fbb_\eta(\vt)\equiv 
\Big(\evb_1\!\!: \cS_{\eta}^+(\vt)
\!\fiber\!\!\big(\!(i\!+\!1,\fb_{\al_{i}(\eta)})_{i\in[k_\bu(\eta)]};
(i,\Ga_i)_{\Ga_i\in L_\bu(\eta)}\big)\lra Y\Big);$$
see the second diagram in Figure~\ref{cSetavt_fig}.
If $(\eta^1;k_1,\eta^2)$ is the image of $(\eta;\vt)$
under~\eref{BCpseudo_e11},
Corollary~\ref{Scompare_crl} with $B\!=\!\{\pt\}$ gives
\BE{BCpseudo_e10}\fbb_{\eta}(\vt)
=(-1)^{k_1+k_2}\cdot(-1)^{k_{\bu}(\eta^2)}\fbb_{\eta^1}\big(k_1,\eta^2\big)
=-\fbb_{\eta^1}\big(k_1,\eta^2\big).\EE

\begin{figure}
\begin{pspicture}(-2.5,-4.6)(10,.8)
\psset{unit=.5cm}
% top left 
\pscircle[linewidth=.05](2,-3){2}
\pscircle*(0,-3){.1}\pscircle*(4,-3){.1}\pscircle*(2,-1){.1}\pscircle*(2,-5){.1}
\pscircle*(.59,-1.59){.1}\pscircle*(.59,-4.41){.1}
\pscircle*(3.41,-1.59){.1}\pscircle*(3.41,-4.41){.1}
\rput(-.4,-2.9){\sm{$1$}}\rput(.7,-3.9){\sm{$2$}}\rput(2,-4.5){\sm{$3$}}
\rput(3.1,-4){\sm{$4$}}\rput(3.5,-2.9){\sm{$5$}}\rput(3.1,-2){\sm{$6$}}
\rput(2,-1.5){\sm{$7$}}\rput(.7,-2.1){\sm{$8$}}
\rput(.2,-4.8){\sm{$\al_1$}}\rput(2,-5.5){\sm{$\al_2$}}\rput(3.8,-4.8){\sm{$\al_3$}}
\rput(4.6,-3.1){\sm{$\al_4$}}\rput(3.8,-1.2){\sm{$\al_5$}}\rput(2.1,-.5){\sm{$\al_6$}}
\pscircle*(1.5,-3.3){.1}\pscircle*(2.5,-3.3){.1}
\rput(.2,-1.1){\sm{$\al_7$}}
\rput(1.5,-2.75){\sm{$\Ga_1$}}\rput(2.6,-2.75){\sm{$\Ga_2$}}
\rput(2.2,.7){$\fbb_{\eta}$}
\rput(2.2,-6.7){\sm{$\eta\!=\!\big(\be_{\bu},7,\{\Ga_1,\Ga_2\},(\al_i)_{i\in[7]}\big)$}}
% right
\pscircle[linewidth=.05](15,-3){1.5}\pscircle[linewidth=.05](18,-3){1.5}
\pscircle*(16.5,-3){.1}
\pscircle*(13.5,-3){.1}\pscircle*(15,-4.5){.1}\pscircle*(15,-1.5){.1}
\pscircle*(13.94,-4.06){.1}\pscircle*(13.94,-1.94){.1}
\pscircle*(19.5,-3){.1}\pscircle*(18,-4.5){.1}\pscircle*(18,-1.5){.1}
\rput(13.1,-2.9){\sm{$1$}}\rput(14.1,-3.5){\sm{$2$}}\rput(15,-4){\sm{$3$}}
\rput(18,-4){\sm{$4$}}\rput(19.1,-2.9){\sm{$5$}}\rput(18,-2){\sm{$6$}}
\rput(15,-2){\sm{$7$}}\rput(14.1,-2.5){\sm{$8$}}
\rput(13.4,-4.4){\sm{$\al_1$}}\rput(15,-5){\sm{$\al_2$}}\rput(18,-5){\sm{$\al_3$}}
\rput(20.1,-3.1){\sm{$\al_4$}}\rput(18,-1){\sm{$\al_5$}}\rput(15.1,-1){\sm{$\al_6$}}
\pscircle*(15.5,-3.3){.1}\pscircle*(17.5,-3.3){.1}
\rput(13.6,-1.5){\sm{$\al_7$}}
\rput(15.6,-2.8){\sm{$\Ga_2$}}\rput(17.7,-2.85){\sm{$\Ga_1$}}
\rput(16.5,.7){$\fbb_{\eta}(\vt)\!\approx\!-\fbb_{\eta^1}(k_1,\eta^2)$}
\rput(16.5,-6.7){\sm{$\vt\!=\!\big(\be_{\bu}^2,k_1\!=\!3,k_2\!=\!7,L_{\bu}^2\!=\!\{\Ga_1\}\big)$}}
\rput(16.5,-7.7){\sm{$\eta^1\!=\!\big(\be_{\bu}\!-\!\be_{\bu}^2,4,\{\Ga_2\},
(\al_1,\al_2,\al_6,\al_7)\!\big)$}}
\rput(16.5,-8.7){\sm{$\eta^2\!=\!\big(\be_{\bu}^2,3,\{\Ga_1\},(\al_3,\al_4,\al_5)\!\big)$}}
\end{pspicture}
\caption{Elements of the domains of $\fbb_{\eta}$ and $\fbb_{\eta}(\vt)$.}
\label{cSetavt_fig}
\end{figure}

For each $\eta\!\in\!\cD_{\om}(\al)$, 
Lemma~\ref{fibersign_lmm} and the first statement in~\eref{bdimprp_e} give
\begin{equation*}\begin{split}
\prt\,\dom(\fbb_\eta)= &
\bigsqcup_{(\eta;\vt)\in\cD_{\om}^2(\al)}\hspace{-.25in}\dom\big(\fbb_\eta(\vt)\!\big)\\
&~~~\sqcup
\bigsqcup_{i=1}^{k_\bu(\eta)}\!\!
\M^+_{\eta;J}\!\fiber\!\!\big(\!(j\!+\!1,\fb_{\al_{j}(\eta)})_{j\in[k_\bu(\eta)]-\{i\}},
(i\!+\!1,\prt\fb_{\al_i(\eta)});(j,\Ga_j)_{\Ga_j\in L_\bu(\eta)}\big).
\end{split}\end{equation*}
Combining this with~\ref{BC0_it} and~\ref{BCprt_it} in Definition~\ref{bndch_dfn} with $\al'$ 
replaced by $\al_i(\eta)\!\prec\!\al$, we obtain
\BE{bdrybb_e}
\prt\fbb_{\al}= \!\!\!
\bigsqcup_{(\eta;\vt)\in\cD_{\om}^2(\al)}\hspace{-.25in}\fbb_\eta(\vt)\sqcup
\bigsqcup_{\begin{subarray}{c}(\eta^1;i,\eta^2)\in\ov\cD_{\om}^2(\al)\\
\dim(\al_i(\eta^1))\le n-2\end{subarray}}\hspace{-.4in}
\fbb_{\eta^1}\big(i,\eta^2\big)\,.\EE
By the bijectivity of~\eref{BCpseudo_e11} and~\eref{BCpseudo_e10}, 
each term $\fbb_{\eta^1}(i,\eta^2)$ in~\eref{bdrybb_e} 
cancels with the corresponding term~$\fbb_\eta(\vt)$.
Below we show that the remaining terms~$\fbb_\eta(\vt)$ either do not contribute 
to~$\prt\fbb_{\al}$ for dimensional reasons or cancel in pairs.

Suppose $(\eta;\vt)\!\in\!\cD_{\om}^2(\al)$, 
$\vt\!=\!(\be_{\bu}^2,k_1,k_2,L_{\bu}^2)$, and
the image $(\eta^1;k_1,\eta^2)\!\in\!\ov\cD_{\om}^2(\al)$ of $(\eta;\vt)$
under~\eref{BCpseudo_e11} does not satisfy 
the inequality in~\eref{bdrybb_e} with $i\!=\!k_1$.
Thus, $\dim\,\fbb_{\al_{k_1}(\eta^1)}\!\ge\!n$.
Let $\be_{\bu}^1\!=\!\be_{\bu}(\eta^1)$, $k_{\bu}^1\!=\!k_{\bu}(\eta^1)$,
and $L_{\bu}^1\!=\!L_{\bu}(\eta^1)$.
Since
\begin{equation*}\begin{split}
\dim\,\M^+_{\eta^1;J}\!\fiber\!\!
\big(\!(i\!+\!1,\fb_{\al_i(\eta^1)})_{i\in[k_\bu^1]-\{k_1\}};
(i,\Ga_i)_{\Ga_i\in L_\bu^1}\big)
\!-\!\big(n\!-\!\dim\,\fbb_{\al_{k_1}}(\eta^1)\big)
&=\dim\,\fbb_{\eta}\!-\!1\\
&=\dim(\al) 
\end{split}\end{equation*}
by~\eref{BCpseudo_e3}, it follows~that
\BE{BCpseudo_e16}\begin{split}
&\dim\,\M_{k^1_\bu,L^1_\bu}(\be_\bu^1;J)
\!\fiber\!\!
\big(\!(i\!+\!1,\fb_{\al_i(\eta^1)})_{i\in[k_1-1]},
(i,\fb_{\al_i(\eta^1)})_{i\in[k_\bu^1]-[k_1]};
(i,\Ga_i)_{\Ga_i\in L_\bu^1}\big) %\\
%&\hspace{4.2in}
<\dim(\al).
\end{split}\EE
If $\be_{\bu}^1\!\neq\!0$, or $k_{\bu}^1\!\ge\!3$, or $L_{\bu}^1\!\neq\!\eset$, 
the map $\fbb_{\eta}(\vt)$ thus factors through a manifold
of dimension less than~$\dim(\al)$.
Thus, $\fbb_{\eta}(\vt)$ does not contribute to~$\prt\fbb_{\al}$ in this case.

The remaining case is $\be_{\bu}^1\!=\!0$, $k_{\bu}^1\!=\!2$, and $L_{\bu}^1\!=\!\eset$.
The associated boundary terms come in pairs arising from two elements $\eta\!\in\!\cD_{\om}(\al)$
with the same $\be_{\bu}(\eta)$, $k_{\bu}(\eta)$, and $L_{\bu}(\eta)$ and with the tuples 
$(\al_i(\eta)\!)_{i\in k_{\bu}(\eta)}$ differing by the circular permutation moving
the first component to the last position.
The pair $(k_1,k_2)$ is $(2,k_{\bu}(\eta)\!+\!2)$ in one case and 
$(1,k_{\bu}(\eta)\!+\!1)$ in the other.
Since dimension of~$Y$ is odd and the dimension of every $\fb_{\al_i(\eta)}$ is even,
\eref{BCpseudo_e10} and Lemma~\ref{fibprodflip_lmm} imply that the boundary terms in
each such pair come with opposite orientations and thus cancel.
\end{proof}

\begin{proof}[{\bf{\emph{Proof of Proposition~\ref{bndch_prp}}}}]
We use induction with respect to the partial order $\prec$ on 
$\cC_{\om}(Y)$ defined in Section~\ref{Notation_subs}.
We assume that~$(\al,J)$ and $(\fb_{\al'})_{\al'\in\cC_{\om;\al}(Y)}$ 
are as in the statement of Lemma~\ref{BCpseudo_lmm} with 
$-2\!<\!\dim(\al)\!\le\!n\!-\!2$.
By this lemma, $\fbb_{\al}$ is then a pseudocycle~with
\BE{bndch_e3}\dim\,\fbb_{\al}=\dim(\al)\!+\!1\le n\!-\!1.\EE
By~\eref{dimeven_e}, this dimension is odd.
Since $Y$ is a rational homology sphere, there exists a bordered pseudocycle~$\fb_{\al}$
into~$Y$ satisfying~\ref{BCdim_it} and~\ref{BCprt_it} in Definition~\ref{bndch_dfn}
with~$\al'$ replaced by~$\al$.
\end{proof}

\vspace{.1in}

\begin{rmk}\label{bndch_rmk}
Without the condition~\eref{dimeven_e}, 
the pseudocycle~$\fbb_{\al}$ in~\eref{bndch_e3} could be of dimension~0. 
If $\eta\!\in\!\cD_{\om}(\al)$ and $\fbb_{\eta}\!\neq\!\eset$ in such a case,
then either
\begin{gather*}
\eta=\big(0,0,\{\Ga\},()\!\big)~~\hbox{with}~\Ga\in\PC(X),~\dim\,\Ga=n,
\qquad\hbox{or}\\
\eta=(0,2,\eset,(\al_1,\al_2)\!\big)~~\hbox{with}~\al_1,\al_2\in\cC_{\om;\al}(Y),~
\dim\,\fb_{\al_1}\!+\!\dim\,\fb_{\al_2}=\dim\,Y.
\end{gather*}
The first possibility could be excluded by requiring that $\Ga\!\cap\!Y\!=\eset$
whenever $\dim\,\Ga\!=\!n$.
Since \hbox{$\dim\,Y\!\not\in\!2\Z$}, 
Lemma~\ref{fibprodflip_lmm} implies that 
$$\fbb_{(0,2,\eset,(\al_1,\al_2))}=-\fbb_{(0,2,\eset,(\al_2,\al_1))}$$
in the second case.
Thus, the pseudocycles $\fbb_{\eta}$ in this case cancel in pairs.
In either case, we could thus take $\fb_{\al}\!=\!\eset$.
We alternatively could restrict the condition in Definition~\ref{bndch_dfn}\ref{BCprt_it} 
to $\al'\!\in\!\cC_{\om;\al}(Y)$ with $-2\!<\!\dim(\al')\!\le\!n\!-\!2$
and treat the additional special case~$\al_{k_1}\!(\eta^1)$ in the proof of Lemma~\ref{BCpseudo_lmm} 
just as in the proof of Lemma~\ref{psisot_lmm} below.
\end{rmk}

\begin{lmm}\label{psisot_lmm}
Let $\al_0,\al_1,\wt\al$, $J_0,J_1,\wt{J}$, and 
$(\fb_{0;\al'})_{\al'\in\cC_{\om;\al_0}(Y)}$ and $(\fb_{1;\al'})_{\al'\in\cC_{\om;\al_1}(Y)}$
be as in Definition~\ref{psisot_dfn} so that $\wt\Ga\!\cap\!Y\!=\!\eset$ for every 
$\wt\Ga\!\in\!\wt{L}(\wt\al)$ with $\dim\,\wt\Ga\!=\!n$.
If $(\fb_{\wt\al'})_{\wt\al'\in\wt\cC_{\om;\wt\al}(Y)}$ is a pseudo-isotopy 
on $(\wt\al,\wt{J})$ between $(\fb_{0;\al'})_{\al'\in\cC_{\om;\al_0}(Y)}$
and $(\fb_{1;\al'})_{\al'\in\cC_{\om;\al_1}(Y)}$,
then the map~$\fbb_{\wt\al}$ in~\eref{wtfbbdfn_e} is a bordered pseudocycle~with
\BE{psisot_e}\dim\,\fbb_{\wt\al}=\dim(\wt\al)\!+\!2 
\quad\hbox{and}\quad
\prt\fbb_{\wt\al}=\{0\}\!\times\!\fbb_{0;\al_0}\!-\!\{1\}\!\times\!\fbb_{1;\al_1}.\EE
\end{lmm}

\begin{proof} The proof is similar to that of Lemma~\ref{BCpseudo_lmm}.
For each $\wt\eta\!\in\!\cD_{\om}(\wt\al)$, let $\fbb_{\wt\eta}$
be as in~\eref{wtfbbetadfn_e0}.
By Definition~\ref{psisot_dfn}\ref{isodim_it} with~$\wt\al'$ 
replaced by $\wt\al_i(\wt\eta)\!\prec\!\wt\al$,
\BE{BCpseudo_e3b}\begin{split}
\dim\,\fbb_{\wt\eta}=\,&
\big(\mu_Y^{\om}(\be_{\bu}(\wt\eta)\!)\!+\!(n\!-\!3)\!+\!(k_{\bu}(\wt\eta)\!+\!1)
\!+\!2|\wt{L}_{\bu}(\wt\eta)|\big)\!+\!1\\
&+\!\!\sum_{i=1}^{k_{\bu}(\wt\eta)}\!\!\!
\big(\dim(\wt\al_i(\wt\eta)\!)\!+\!3\big)
\!-\!(n\!+\!1)k_{\bu}(\wt\eta)
\!-\!\!\!\sum_{\wt\Ga\in\wt{L}_{\bu}(\wt\eta)}\!\!\!\!\!\codim\,\wt\Ga
=\dim(\wt\al)\!+\!2\,.
\end{split}\EE
Thus, the dimension of~$\fbb_{\wt\al}$ is $\dim(\wt\al)\!+\!2$.

We define $\cD_{\om}^2(\wt\al)$ and $\ov\cD_{\om}^2(\wt\al)$ as in~\eref{cDonm2dfn_e}
and~\eref{BCpseudo_e11} with~$\al$, $L_{\bu}(\eta)$, and $\al_i(\eta)$ replaced 
by~$\wt\al$, $\wt{L}_{\bu}(\eta)$, and~$\wt\al_i(\eta)$, respectively, and
a bijection
\BE{BCpseudo_e11b}\cD_{\om}^2(\wt\al) \lra \ov\cD_{\om}^2(\wt\al)\EE
as above~\eref{BCpseudo_e11}.
For $(\wt\eta;i,\wt\eta')\!\in\!\ov\cD_{\om}^2(\wt\al)$, let
$$\fbb_{\wt\eta}(i,\wt\eta')=
\bigg(\!\wt\evb_1\!\!:\M^+_{\wt\eta;\wt{J}}\!\fiber\!\!
\big(\!(j\!+\!1,\fb_{\wt\al_{j}(\wt\eta)})_{j\in[k_{\bu}(\wt\eta)]-\{i\}},
(i\!+\!1,\fbb_{\wt\eta'});(j,\wt\Ga_j)_{\wt\Ga_j\in\wt{L}_{\bu}(\wt\eta)}\big)
\lra [0,1]\!\times\!Y\!\!\bigg).$$
For $(\wt\eta;\vt)\!\in\!\cD_{\om}^2(\wt\al)$ with 
$\vt\!=\!(\be_{\bu}^2,k_1,k_2,\wt{L}_{\bu}^2)$, let
\BE{BCpseudo_e12b}\cS_{\wt\eta}^+(\vt)\subset \prt\M^+_{\wt\eta;\wt{J}}\EE 
be the subspace consisting of $\wt{J}$-holomorphic maps from $(\D^2\!\v\!\D^2,S^1\!\v\!S^1)$
to $[0,1]\!\times\!(X,Y)$ with the second component of degree~$\be_{\bu}^2$
and carrying the boundary marked points indexed by $[k_2\!-\!1]\!-\![k_1]$ 
and the interior marked points indexed by~$\wt{L}_{\bu}^2$.
This topological component of $\prt\M^+_{\wt\eta;\wt{J}}$
inherits an orientation from the orientation of~$\M^+_{\wt\eta;\wt{J}}$.
Define
$$\fbb_{\wt\eta}(\vt)\equiv 
\Big(\evb_1\!: \cS_{\wt\eta}^+(\vt)
\!\fiber\!\big(\!(i\!+\!1,\fb_{\wt\al_{i}(\wt\eta)})_{i\in[k_\bu(\wt\eta)]};
(i,\wt\Ga_i)_{\wt\Ga_i\in\wt{L}_\bu(\wt\eta)}\big)\lra [0,1]\!\times\!Y\!\Big).$$
If $(\wt\eta^1;k_1,\wt\eta^2)$ is the image of $(\wt\eta;\vt)$
under~\eref{BCpseudo_e11b},
Corollary~\ref{Scompare_crl} with $B\!=\![0,1]$, $k\!=\!k_{\bu}(\wt\eta)\!+\!1$, 
and $|I|\!=\!k_{\bu}(\wt\eta)$ give
\BE{BCpseudo_e10b}\begin{split}
(-1)^{\binom{k_{\bu}(\wt\eta)+1}{2}}\fbb_{\wt\eta}(\vt)
&=-(-1)^{k_2+k_{\bu}(\wt\eta)(k_1+k_2)}\cdot(-1)^{\binom{k_{\bu}(\wt\eta)+1}{2}}
\cdot(-1)^{\binom{k_{\bu}(\wt\eta^2)}{2}}\fbb_{\wt\eta^1}\big(k_1,\wt\eta^2\big)\\
&=-(-1)^{\binom{k_{\bu}(\wt\eta^1)+1}{2}-k_1}\fbb_{\wt\eta^1}\big(k_1,\wt\eta^2\big).
\end{split}\EE

\vspace{-.1in}

For $\wt\eta\!\in\!\cD_{\om}(\wt\al)$, 
Lemma~\ref{fibersign_lmm} and the last statement in~\eref{bdimprp_e} give
\begin{equation*}\begin{split}
&(-1)^{\binom{k_{\bu}(\wt\eta)}{2}}\!\cdot\!(-1)^{k_{\bu}(\wt\eta)}
\prt\,\dom\big(\fbb_{\wt\eta}\big|_{(0,1)}\big)
=\!\!
\bigsqcup_{(\wt\eta;\vt)\in\cD_{\om}^2(\wt\al)}\!\!\!\!\!\!\!\!\!\!
\fbb_{\wt\eta}(\vt)\\ 
&\hspace{1in}
\sqcup\!\bigsqcup_{i=1}^{k_\bu(\wt\eta)}\!\!(-1)^i
\M^+_{\wt\eta;\wt{J}}\!\fiber\!\!\big(\!(j\!+\!1,\fb_{\wt\al_j}(\wt\eta))_{j\in[k_\bu(\wt\eta)]-\{i\}},
(i\!+\!1,\prt\fb_{\wt\al_i(\wt\eta)});(j,\wt\Ga_j)_{\wt\Ga_j\in\wt{L}_\bu(\wt\eta)}\big).
\end{split}\end{equation*}
Combining this with~\ref{isozero_it} and~\ref{isoprt_it} 
in Definition~\ref{psisot_dfn} with $\wt\al'$ replaced by $\wt\al_i(\wt\eta)\!\prec\!\wt\al$, 
we obtain
\BE{wtcancel_e}
\fbb_{\wt\al}\big|_{(0,1)}=\!\!\!
\bigsqcup_{(\wt\eta;\vt)\in\cD_{\om}^2(\wt\al)}\hspace{-.22in}
(-1)^{\binom{k_{\bu}(\wt\eta)+1}{2}}\fbb_{\wt\eta}(\vt)
\sqcup\!\!\!
\bigsqcup_{\begin{subarray}{c}(\wt\eta^1;i,\wt\eta^2)\in\ov\cD_{\om}^2(\wt\al)\\
-2<\dim(\wt\al_i(\wt\eta^1))\le n-2\end{subarray}}\hspace{-.48in}
(-1)^{\binom{k_{\bu}(\wt\eta^1)+1}{2}+i}
\fbb_{\wt\eta^1}\big(i,\wt\eta^2\big)\,.\EE
By the bijectivity of~\eref{BCpseudo_e11b} and~\eref{BCpseudo_e10b}, 
each term $\fbb_{\wt\eta^1}(i,\wt\eta^2)$ in~\eref{wtcancel_e} 
cancels with the corresponding term~$\fbb_{\wt\eta}(\vt)$.
Below we show that the remaining terms~$\fbb_{\wt\eta}(\vt)$ either do not contribute 
to~$\prt\fbb_{\wt\al}|_{(0,1)}$ for dimensional reasons or cancel in pairs.

Let $(\wt\eta;\vt)\!\in\!\cD_{\om}^2(\wt\al)$, 
$\vt\!=\!(\be_{\bu}^2,k_1,k_2,\wt{L}_{\bu}^2)$, and
$(\wt\eta^1;k_1,\wt\eta^2)\!\in\!\ov\cD_{\om}^2(\wt\al)$ 
be the image of $(\wt\eta;\vt)$ under~\eref{BCpseudo_e11b}.
If the second inequality in~\eref{wtcancel_e} with $i\!=\!k_1$ fails, 
similar reasoning to that in the last two paragraphs
of the proof of Lemma~\ref{BCpseudo_lmm} and~\eref{BCpseudo_e3b}
imply that the term~$\fbb_{\wt\eta}(\vt)$ either does not contribute
to~$\prt\fbb_{\wt\eta}|_{(0,1)}$ for dimensional reasons or cancels with 
another term~$\fbb_{\wt\eta'}(\vt')$.

Suppose the first inequality in~\eref{wtcancel_e} with $i\!=\!k_1$ fails and 
$\wt\fbb_{\wt\eta^1}(k_1,\wt\eta^2)\!\neq\!\eset$.
Let $k_{\bu}^2\!=\!k_{\bu}(\wt\eta^2)$.
Since 
\BE{etadimsplit_e}\dim\,\M^+_{\wt\eta^2;\wt J}\!\fiber\!
\big(\!(i\!+\!1,\fb_{\wt\al_i(\wt\eta^2)})_{i\in[k_\bu^2]};
(i,\wt\Ga_i)_{\wt\Ga_i\in\wt{L}_{\bu}^2}\big)
=\dim\,\fbb_{\wt\eta^2}=\dim\big(\wt\al_{k_1}\!(\wt\eta^1)\!\big)\!+\!2,\EE
it follows that this dimension is~0.
Thus, 
$$\dim\,\M_{\wt\eta^2;\wt J}\!\fiber\!
\big(\!(i,\fb_{\wt\al_i(\wt\eta^2)})_{i\in[k_\bu^2]};
(i,\wt\Ga_i)_{\wt\Ga_i\in\wt{L}_{\bu}^2}\big)
=-1.$$
If $\be_{\bu}^2\!\neq\!0$ or $k_{\bu}^2\!+\!2|\wt{L}_{\bu}^2|\!\ge\!3$, 
this implies that $\wt\fbb_{\wt\eta^1}(k_1,\wt\eta^2)\!=\!\eset$.
If $\be_{\bu}^2\!=\!0$, $k_{\bu}^2\!=\!0$, and $\wt{L}_{\bu}^2\!=\!\{\wt\Ga\}$ 
is a single-element set,
then the dimension of~$\wt\Ga$ is~$n$.
Since $\wt\Ga$ is then disjoint from~$Y$, it follows that
$\wt\fbb_{\wt\eta^1}(k_1,\wt\eta^2)\!=\!\eset$ in this case as well.
The remaining case is $\be_{\bu}^2\!=\!0$, $k_{\bu}^2\!=\!2$, and $\wt{L}_{\bu}^2\!=\!\eset$.
The associated boundary terms come in pairs arising from the same $k_1$ and $k_2\!=\!k_1\!+\!3$
and from two elements $\wt\eta\!\in\!\cD_{\om}(\wt\al)$ with the same 
$\be_{\bu}(\wt\eta)$, $k_{\bu}(\wt\eta)$, and $\wt{L}_{\bu}(\wt\eta)$ 
and with the tuples $(\wt\al_i(\wt\eta))_{i\in k_{\bu}(\wt\eta)}$ 
differing by the transposition interchanging the $k_1\!+\!1$ and $k_1\!+\!2$ entries.
By Lemma~\ref{fibprodflip_lmm} and the last statement in~\eref{bdimprp_e},
the associated cycles~$\fbb_{\wt\eta^2}$ have opposite orientations.
Along with~\eref{BCpseudo_e10b}, this implies that
the paired up boundary terms~$\fbb_{\wt\eta}(\vt)$
come with opposite orientations as well and thus cancel.

Let $\wt\eta\!\in\!\cD_{\om}(\wt\al)$, $\eta_0\!\in\!\cD_{\om}(\al_0)$, 
and $\eta_1\!\in\!\cD_{\om}(\al_1)$ be so~that
\BE{eta01dfn_e}\prt\wt\eta=\{1\}\!\times\!\eta_1\!-\!\{0\}\!\times\!\eta_0 \,.\EE
In order to compute the signs of the boundary terms of $\fbb_{\wt\eta}$ over~0 and~1,
we extend $\M^+_{\wt\eta;\wt J}$, $\fb_{\wt\al'}$ with $\wt\al'\!\prec\!\wt\al$,
and $\wt\Ga_i\!\in\!\wt{L}_{\bu}(\wt\eta)$ past their boundaries over $0,1\!\in\!\R$.
In other words, let 
$$\M^{+'}_{\wt\eta;\wt J}=\Big(\!\!
\big(\!(-1,0]\!\times\!\M^+_{\eta_0;J_0}\big)\!\sqcup\!\M^+_{\wt\eta;\wt J}\!\sqcup\!
\big([1,2)\!\times\!\M^+_{\eta_1;J_1}\big)\!\!\Big)\big/\!\!\sim$$
with the identifications~$\sim$ of the elements $(0,\u)$ and $(1,\u)$ of $\M^+_{\wt\eta;\wt J}$
with same elements in the added collars.
Let
$$\wt\evb'_i\!:
\M^{+'}_{\wt\eta;\wt J}\lra \R\!\times\!Y\ (\tn{resp. }\R\!\times\!X)$$
be extensions of $\wt\evb_i$ with $i\!\in\![k_{\bu}(\wt\eta)]$ 
(resp.~$i\!\in\!\wt{L}_{\bu}(\wt\eta)$)
so that their compositions
with the projections to~$\R$ restrict over the two collars to the projections to the first factor.
We similarly extend the domains of $\fb_{\wt\al'}$ and $\wt\Ga_i$ by collars over
$(-1,0)$ and $(1,2)$ and then extend the maps $\fb_{\wt\al'}$ and $\wt\Ga_i$ to smooth maps
$\fb_{\wt\al'}'$ and~$\wt\Ga_i'$ to $\R\!\times\!Y$ and $\R\!\times\!X$, respectively.
We extend the use of the notation~$\fiber$ defined at the end of Section~\ref{Notation_subs}
to~$\M^{+'}_{\wt\eta;\wt J}$.
We assume that the map extensions above are chosen generically so that all 
relevant fiber products are smooth.
Let
$$e'\!: \M^{+'}_{\wt\eta;\wt J}\!\fiber\!
\big(\!(i\!+\!1,\wt\fb'_{\wt\al_i(\wt\eta)})_{i\in[k_\bu(\wt\eta)]};
(i,\wt\Ga'_i)_{\wt\Ga_i\in\wt{L}_\bu(\wt\eta)}\big)
\lra\R$$
be the projection map.

Let $\io\!:[0,1]\!\lra\!\R$ be the inclusion.
By Lemma~\ref{fibprodemd_lmm}, \eref{BCpseudo_e3b}, and~\eref{dimeven_e}, 
\begin{equation*}\begin{split}
&\M^+_{\wt\eta;\wt J}\!\fiber\!
\big(\!(i\!+\!1,\fb_{\wt\al_i(\wt\eta)})_{i\in[k_\bu(\wt\eta)]};
(i,\wt\Ga_i)_{\wt\Ga_i\in\wt{L}_\bu(\wt\eta)}\big)\\
&\hspace{1in}
=-\Big([0,1]_{\io}\!\!\times_{e'}\!\!
\Big(\M^{+'}_{\wt\eta;\wt J}\!\fiber\!\!
\big(\!(i\!+\!1,\fb'_{\wt\al_i(\wt\eta)})_{i\in[k_\bu(\wt\eta)]};
(i,\wt\Ga'_i)_{\wt\Ga_i\in\wt{L}_\bu(\wt\eta)}\big)\!\Big)\!\Big). 
\end{split}\end{equation*}
Along with Lemma~\ref{fibersign_lmm}, this implies that  
\begin{equation*}\begin{split}
&\prt\big(\dom\,\fbb_{\wt\eta}\big)
=(-1)^{\binom{k_{\bu}(\wt\eta)}{2}}\bigg(\!\!
\big(\{1\}\!-\!\{0\}\big)\!\,_{\io}\!\!\times_{e'}\!\!
\Big(\M^{+'}_{\wt\eta;\wt J}\!\fiber\!\!
\big(\!(i\!+\!1,\fb'_{\wt\al_i(\wt\eta)})_{i\in[k_\bu(\wt\eta)]};
(i,\wt\Ga'_i)_{\wt\Ga_i\in\wt{L}_\bu(\wt\eta)}\big)\!\Big)\\
&\hspace{2.2in}\sqcup[0,1]_{\io}\!\!\times_{e'}\!
\prt\Big(\M^{+'}_{\wt\eta;\wt J}\!\fiber\!\!
\big(\!(i\!+\!1,\fb'_{\wt\al_i(\wt\eta)})_{i\in[k_\bu(\wt\eta)]};
(i,\wt\Ga'_i)_{\wt\Ga_i\in\wt{L}_\bu(\wt\eta)}\big)\!\Big)\!\!\bigg).
\end{split}\end{equation*}
Applying Lemma~\ref{Bdrop_lmm}, we then obtain 
$$\prt\wt\fbb_\eta
=(-1)^{\binom{k_{\bu}(\wt\eta)}{2}}\!\cdot\!(-1)^{k_{\bu}(\wt\eta)+\binom{k_{\bu}(\wt\eta)+2}{2}}
\!\cdot\!(-1)^{k_{\bu}(\wt\eta)}\!
\big(\{1\}\!\times\!\fbb_{1;\eta_1}\!-\!\{0\}\!\times\!\fbb_{0;\eta_0}\big)
\!+\!\prt\fbb_{\wt\eta}\big|_{(0,1)}\,.$$
Since the last term above vanishes after summing over $\wt\eta\!\in\!\cD_{\om}(\wt\al)$, 
this establishes the claim.
\end{proof}

\begin{proof}[{\bf{\emph{Proof of Proposition~\ref{psisot_prp}}}}]
We use induction with respect to the partial order $\prec$ on 
$\cC_{\om}(Y)$ defined in Section~\ref{Notation_subs}.
We assume $\al_0,\al_1,\wt\al$, $J_0,J_1,\wt{J}$, 
$(\fb_{0;\al'})_{\al'\in\cC_{\om;\al_0\!}(Y)}$, $(\fb_{1;\al'})_{\al'\in\cC_{\om;\al_1\!}(Y)}$,
and $(\fb_{\wt\al'})_{\wt\al'\in\wt\cC_{\om;\wt\al}(Y)}$ 
are as in the statement of Lemma~\ref{psisot_lmm} with 
$$-2<\dim(\wt\al)\le n\!-\!2.$$
By this lemma and~\eref{BCprt_e}, 
$\fbb_{\wt\al}\!+\!\{1\}\!\times\!\fb_{1;\al_1}\!-\!\{0\}\!\times\!\fb_{0;\al_0}$
is then a pseudocycle~with
$$0<\dim\big(\fbb_{\wt\al}\!+\!\{1\}\!\times\!\fb_{1;\al_1}\!-\!\{0\}\!\times\!\fb_{0;\al_0}\big)
=\dim(\wt\al)\!+\!2\le n.$$ 
By~\eref{dimeven_e}, this dimension is even.
Since $Y$ is a rational homology sphere, there exists a bordered pseudocycle~$\fb_{\wt\al}$
to $[0,1]\!\times\!Y$ satisfying~\ref{isodim_it} and~\ref{isoprt_it} in Definition~\ref{psisot_dfn}
with $\wt\al'$ replaced by~$\wt\al$.
\end{proof}

\subsection{Pseudo-isotopies and invariance of disk counts}
\label{countinv_subs}

We now complete the proof of Theorem~\ref{countinv_thm}\ref{countinv_it} 
by establishing~\eref{cob_e}.
Its proof is similar to that of Lemma~\ref{psisot_lmm}, but is more combinatorially involved.
It uses Lemma~\ref{rhoRsgn_lmm} below.

For $\wt\al\!\in\!\wt\cC_{\om}(Y)$, 
let $\ov\cD_{\om}^2(\wt\al)$ be as in the proof of Lemma~\ref{psisot_lmm} and
\BE{ovCdom2dfn_e}\ov\cD_{\om}^{2*}(\wt\al)=
\big\{\!(\wt\eta;i,\wt\eta')\!\in\!\ov\cD_{\om}^2(\wt\al)\!:
(\be_{\bu}(\wt\eta),k_{\bu}(\wt\eta),\wt{L}_{\bu}(\wt\eta)\!)\!\neq\!(0,2,\eset)\big\}.\EE
We define a ``rotation" on the elements of $\cD_{\om}(\wt\al)$ by
\BE{rhodfn_e}
\rho\!:\cD_{\om}(\wt\al)\lra\cD_{\om}(\wt\al),\quad
\rho\big(\be_{\bu},k_{\bu},\wt{L}_{\bu},(\wt\al_i)_{i\in[k_{\bu}(\wt\eta)]}\big)
=\big(\be_{\bu},k_{\bu},\wt{L}_{\bu},
(\wt\al_2,\wt\al_3,\ldots,\wt\al_{k_{\bu}(\wt\eta)},\wt\al_1)\!\big).\EE
This bijection induces a bijection
\BE{rhoovcDdfn_e}\rho\!:\ov\cD_{\om}^2(\wt\al)\lra\ov\cD_{\om}^2(\wt\al), \quad
\quad \rho(\wt\eta;i,\wt\eta')=\begin{cases}(\rho(\wt\eta);i\!-\!1,\wt\eta'),&\hbox{if}~i\!>\!1;\\
(\rho(\wt\eta);k_{\bu}(\wt\eta),\wt\eta'),&\hbox{if}~i\!=\!1;
\end{cases}\EE
it restricts to a bijection on $\ov\cD_{\om}^{2*}(\wt\al)$.

For $i\!\in\![k_{\bu}(\wt\eta)]$, 
we define $\wt\eta\bsl i\!\in\!\cD_\om(\wt\al\!-\!\wt\al_i(\wt\eta)\!)$ by 
\BE{etabslidfn_e}
\begin{split}
\be_\bu(\wt\eta\bsl i)&=\be_\bu(\wt\eta),\\
k_\bu(\wt\eta\bsl i)&=k_\bu(\wt\eta)\!-\!1,\\
\wt{L}_\bu(\wt\eta\bsl i)&=\wt{L}_\bu(\wt\eta),
\end{split}\quad
\wt\al_j(\wt\eta\bsl i)=\begin{cases}
\wt\al_{j+i}(\wt\eta),&\hbox{if}~j\!\in\![k_\bu(\wt\eta)\!-\!i];\\
\wt\al_{j+i-k_\bu(\wt\eta)},&\hbox{if}~j\!\in\![k_\bu(\wt\eta)\!-\!1]\!-\![k_\bu(\wt\eta)\!-\!i].
\end{cases}\EE
Thus, $\wt\eta\bsl i$ is obtained from $\wt\eta$ by dropping the component $\wt\al_i(\wt\eta)$
and ordering the remaining components~$\wt\al_j(\wt\eta)$ starting from the next one
in the circular order. 

For $(\wt\eta;i,\wt\eta')\!\in\!\ov\cD_{\om}^{2*}(\wt\al)$, let
$$\wh\al_j^1=\begin{cases}\wt\al\!-\!\wt\al_i(\wt\eta),&\hbox{if}~j\!=\!1;\\
\wt\al_{j-1}(\wt\eta'),&\hbox{if}~j\!\in\![k_{\bu}(\wt\eta')\!+\!1]\!-\!\{1\};
\end{cases}\quad
\wh\eta=\big(\be_{\bu}(\wt\eta'),k_{\bu}(\wt\eta')\!+\!1,\wt{L}_{\bu}(\wt\eta'),
(\wh\al_j^1)_{j\in[k_{\bu}(\wt\eta')+1]}\big).$$
This construction induces a ``reflection"
\BE{Rdfn_e}R\!:\ov\cD_{\om}^{2*}(\wt\al)\lra \ov\cD_{\om}^{2*}(\wt\al), \quad
R(\wt\eta;i,\wt\eta')=(\wh\eta;1,\wt\eta\bsl i),\EE
such that $R^3\!=\!R$.
Furthermore, $R$ is invariant under the rotation~$\rho$ in~\eref{rhoovcDdfn_e} and
\BE{Rdimprp_e} \dim\big(\wt\al_i(\wt\eta)\!\big)+\dim\big(\wt\al_1(\wh\eta)\!\big)
=\dim(\wt\al)\!+\!n\!-\!3\EE
with the notation as in~\eref{Rdfn_e}.

For $\wt\eta\!\in\!\cD_{\om}(\wt\al)$, let $\fbb_{\wt\eta}$ be as in~\eref{wtfbbetadfn_e0}. 
If in addition $i\!\in\![k_{\bu}(\wt\eta)]$, define 
\BE{rhoMdfn_e}\begin{split}
\rho_{\M}\!:\M_{\wt\eta;\wt J}&\lra\M^+_{\wt\eta\bsl i;\wt J},\\
\big[u,(x_j)_{j\in[k_{\bu}(\wt\eta)]},(z_j)_{j\in\wt{L}(\wt\eta)}\big]
&\lra \big[u,(x_i,x_{i+1},\ldots,x_{k_{\bu}(\wt\eta)},x_1,\ldots,x_{i-1}),
(z_j)_{j\in\wt{L}(\wt\eta)}\big].
\end{split}\EE
For $(\wt\eta;i,\wt\eta')\!\in\!\ov\cD_{\om}^{2*}(\wt\al)$, let
$$\fbb^*_{\wt\eta}(i,\wt\eta')\equiv\M_{\wt\eta;\wt J}\!\fiber\!\!
\big(\!(j,\fb_{\wt\al_j(\wt\eta)})_{j\in[k_{\bu}(\wt\eta)]-\{i\}},
(i,\fbb_{\wt\eta'});(j,\wt\Ga_j)_{\wt\Ga_j\in\wt{L}_{\bu}(\wt\eta)}\big).$$ 
The ``backwards" cyclic
permutations of the boundary marked points of the moduli space in~\eref{rhoMdfn_e}
and of the pseudocycles to~$[0,1]\!\times\!Y$, 
\BE{rotdiffdfn_e}
\rho_{\fb}\!:\prod_{j\in[k_{\bu}(\wt\eta)]}\hspace{-.15in}(\dom\,\fb_{\wt\al_j(\wt\eta)})
\lra\prod_{j\in[k_{\bu}(\wt\eta)]}\hspace{-.15in}(\dom\,\fb_{\wt\al_j(\wt\eta)}),\EE
induce diffeomorphisms
$$\rho_{\wt\eta}\!:\fbb^*_{\wt\eta}\lra\fbb^*_{\rho(\wt\eta)}, \qquad
\rho_{\wt\eta;i,\wt\eta'}\!:\fbb^*_{\wt\eta}(i,\wt\eta')\lra
\begin{cases}\fbb^*_{\rho(\wt\eta)}(i\!-\!1,\wt\eta'),&\hbox{if}~i\!>\!1;\\
\fbb^*_{\rho(\wt\eta)}(k_{\bu}(\wt\eta),\wt\eta'),&\hbox{if}~i\!=\!1.
\end{cases}$$
The interchange of the moduli space components induces a diffeomorphism
$$R_{\wt\eta;i,\wt\eta'}\!: \fbb^*_{\wt\eta}(i,\wt\eta')\lra\fbb^*_{\wh\eta}(1,\wt\eta\bsl i),
\quad\hbox{where}~~(\wh\eta;1,\wt\eta\bsl i)\equiv R(\wt\eta;i,\wt\eta').$$
Let 
$$\ep(\wt\eta,i)=\binom{k_{\bu}(\wt\eta)\!+\!1}{2}\!+\!i\,.$$

\begin{lmm}\label{rhoRsgn_lmm}
Let $\wt\al\!\in\!\wt\cC_{\om}(Y)$.
The diffeomorphism $\rho_{\wt\eta}$ is orientation-preserving for every 
\hbox{$\wt\eta\!\in\!\cD_{\om}(\wt\al)$}.
The sign of the diffeomorphism
\begin{equation*}\begin{split}
&\M_{\wt\eta;\wt J}\!\fiber\!\!
\big(\!(j,\fb_{\wt\al_j(\wt\eta)})_{j\in[k_\bu(\wt\eta)]-\{i\}};
(j,\wt\Ga_j)_{\wt\Ga_j\in\wt{L}_\bu(\wt\eta)}\big)\\
&\hspace{1.5in}\approx
\M_{\wt\eta\bsl i;\wt J}^+\!\fiber\!\!
\big(\!(j\!+\!1,\fb_{\wt\al_j(\wt\eta\bsl i)})_{j\in[k_\bu(\wt\eta\bsl i)]};
(j,\wt\Ga_j)_{\wt\Ga_j\in\wt{L}_\bu(\wt\eta\bsl i)}\big)
\end{split}\end{equation*}
induced by $i\!-\!1$ ``backwards" cyclic permutations of the boundary marked points 
of the moduli space and of the pseudocycles to~$[0,1]\!\times\!Y$
is $(-1)^{i-1}$ for all \hbox{$\wt\eta\!\in\!\cD_{\om}(\wt\al)$} 
and $i\!\in\![k_{\bu}(\eta)]$.
The signs of the diffeomorphisms~$\rho_{\wt\eta;i,\wt\eta'}$ with $i\!\neq\!1$
and $R_{\wt\eta;i,\wt\eta'}$
are $-1$ and $(-1)^{\ep(\wt\eta,i)-\ep(\wh\eta,1)}$, respectively, 
for every \hbox{$(\wt\eta;i,\wt\eta')\!\in\!\ov\cD_{\om}^{2*}(\wt\al)$}.
\end{lmm}

\begin{proof}
The cyclic permutations of the boundary marked points of the elements of $\M_{\wt\eta;\wt J}$
and of the pseudocycles to~$[0,1]\!\times\!Y$ induce a commutative diagram
$$\xymatrix{\M_{\wt\eta;\wt J}\ar[r]^>>>>>{\ev}\ar[d]_{\rho_{\M}}&
\big([0,1]\!\times\!Y\big)^{k_{\bu}(\wt\eta)}\!\!\times\!\!
\big([0,1]\!\times\!X\big)^{\wt{L}_{\bu}(\wt\eta)}\ar[d]_{\rho_Y}&
\prod\limits_{j\in[k_{\bu}(\wt\eta)]}\hspace{-.15in}(\dom\,\fb_{\wt\al_j(\wt\eta)})
\!\times\!\!
\prod\limits_{\wt\Ga_j\in\wt{L}_{\bu}(\wt\eta)}
\hspace{-.12in}(\dom\,\wt\Ga_j) \ar[l]\ar[d]^{\rho_{\fb}}\\ 
\M_{\wt\eta;\wt J}\ar[r]^>>>>>{\ev}&
\big([0,1]\!\times\!Y\big)^{k_{\bu}(\wt\eta)}\!\!\times\!\!
\big([0,1]\!\times\!X\big)^{\wt{L}_{\bu}(\wt\eta)} &
\prod\limits_{j\in[k_{\bu}(\wt\eta)]}\hspace{-.15in}(\dom\,\fb_{\wt\al_j(\wt\eta)})
\!\times\!\!
\prod\limits_{\wt\Ga_j\in\wt{L}_{\bu}(\wt\eta)}\hspace{-.12in}(\dom\,\wt\Ga_j) \ar[l]}$$
so that the vertical arrows are diffeomorphisms.
Since the dimension of~$Y$ is odd, the diffeomorphism~$\rho_Y$ is orientation-preserving.
By the construction of the orientation on~$\M_{\wt\eta;\wt J}$ in Section~\ref{Ms_subs},
the sign of the diffeomorphism~$\rho_{\M}$ is~$(-1)^{k_{\bu}(\wt\eta)-1}$.
By the last statement in~\eref{bdimprp_e}, 
this is also the sign of the diffeomorphism~$\rho_{\fb}$.
The first claim of the lemma now follows from Lemma~\ref{fibprodflip_lmm}.
The claim concerning~$\rho_{\wt\eta;i,\wt\eta'}$ is obtained in the same way
by replacing the odd-dimensional insertion~$\fb_{\wt\al_i(\wt\eta)}$
by the even-dimensional insertion~$\fbb_{\wt\al_i(\wt\eta)}$.
Dropping the insertion~$\fb_{\wt\al_i(\wt\eta)}$ entirely,
we find that the diffeomorphism induced by the ``backwards" rotations is then
orientation-reversing; this establishes the second claim of the lemma.

By the second statement in~\eref{bdimprp_e} and 
Lemmas~\ref{fibprodflip_lmm} and~\ref{fibprodisom_lmm1}, 
$$\fbb^*_{\wt\eta}(i,\wt\eta')
\approx\Big(\M_{\wt\eta;\wt J}\!\fiber\!
\big(\!(j,\fb_{\wt\al_j(\wt\eta)})_{j\in[k_\bu(\wt\eta)]-\{i\}};
(j,\wt\Ga_j)_{\wt\Ga_j\in\wt{L}_\bu(\wt\eta)}\big)\!\!\Big)
\!\,_{\wt\evb_i}\!\!\!\times_{\fbb_{\wt\eta'}}\!\!\big(\dom\,\fbb_{\wt\eta'}\big).$$
Along with the second claim of the lemma, this gives
$$\fbb^*_{\wt\eta}(i,\wt\eta')
\approx (-1)^{i-1}(-1)^{\binom{k_\bu(\wt\eta\bsl i)}{2}}
\big(\dom\,\fbb_{\wt\eta\bsl i}\big)\!\,_{\fbb_{\wt\eta\bsl i}}
\!\!\times_{\fbb_{\wt\eta'}}\!\!\big(\dom\,\fbb_{\wt\eta'}\big).$$
Since the dimension of~$\fbb_{\wt\eta'}$ is even,
it follows that  
$$\fbb^*_{\wt\eta}(i,\wt\eta')
\approx (-1)^{\ep(\wt\eta,i)}\!\big(\fbb_{\wt\eta\bsl i}\!\fiber\!\fbb_{\wt\eta'}\big)
\approx (-1)^{\ep(\wt\eta,i)}\!\big(\fbb_{\wt\eta'}\!\fiber\!\fbb_{\wt\eta\bsl i}\big)
\approx (-1)^{\ep(\wt\eta,i)}(-1)^{\ep(\wh\eta,1)}
\fbb^*_{\wh\eta}(1,\wt\eta\bsl i).$$
This establishes the last claim of the lemma.
 \end{proof}

\begin{proof}[{\bf{\emph{Proof of~\eref{cob_e}}}}]
Let $\al_0,\al_1.\wt\al$, $J_0,J_1,\wt{J}$, 
$(\fb_{0;\al'})_{\al'\in\cC_{\om;\al_0}(Y)}$, $(\fb_{1;\al'})_{\al'\in\cC_{\om;\al_1}(Y)}$,
and $(\fb_{\wt\al'})_{\wt\al'\in\wt\cC_{\om;\wt\al}(Y)}$ 
be as in Definition~\ref{psisot_dfn2}.
With $\cD_{\om}^2(\wt\al)$ as in the proof of Lemma~\ref{psisot_lmm}, define
\BE{cDom2stdfn_e}\begin{split}
\cD_{\om}^{2*}(\wt\al)=&\big\{\!\big(\wt\eta;\be_{\bu}^2,k_1,k_2,\wt{L}_{\bu}^2\big)
\!\in\!\cD_{\om}^2(\wt\al)\!:k_2\!\le\!k_\bu(\wt\eta)\!+\!1,\,
(\be_{\bu}^2,k_2\!-\!k_1,\wt{L}_{\bu}^2)
\!\neq\!(\be_{\bu}(\wt\eta),k_{\bu}(\wt\eta),\wt{L}_{\bu}(\wt\eta)\!)\!\big\}\\
&\sqcup\big\{\!\big(\wt\eta;\be_{\bu}^2,0,1,\wt{L}_{\bu}^2\big)\!:k_\bu(\wt\eta)\!=\!0,\,
\big(\wt\eta;\be_{\bu}^2,1,2,\wt{L}_{\bu}^2\big)\!\in\!\cD_{\om}^2(\wt\al)\big\}.
\end{split}\EE
The rotation~$\rho$ on~$\cD_{\om}(\wt\al)$ defined in~\eref{rhodfn_e}
lifts to a bijection
\begin{gather}\label{rhocD2_e}
\rho\!:\cD_{\om}^{2*}(\wt\al)\lra \cD_{\om}^{2*}(\wt\al),\\
\notag
\rho\big(\wt\eta;\be_{\bu}^2,k_1,k_2,\wt{L}_{\bu}^2 \big)
=\begin{cases}
(\rho(\wt\eta);\be_{\bu}^2,k_1\!-\!1,k_2\!-\!1,\wt{L}_{\bu}^2),&\hbox{if}~k_1\!>\!1;\\
(\rho(\wt\eta);\be_{\bu}^2,k_{\bu}(\wt\eta),k_{\bu}(\wt\eta)\!+\!1,
\wt{L}_{\bu}^2),&\hbox{if}~k_2\!=\!2;\\
(\rho(\wt\eta);\be_{\bu}(\wt\eta)\!-\!\be_{\bu}^2,k_2\!-\!2,k_{\bu}(\wt\eta)\!+\!1,
\wt{L}_{\bu}(\wt\eta)\!-\!\wt{L}_{\bu}^2),&\hbox{if}~k_1\!=\!1,\,k_2\!>\!2;\\
(\rho(\wt\eta);\be_{\bu}(\wt\eta)\!-\!\be_{\bu}^2,0,1,
\wt{L}_{\bu}(\wt\eta)\!-\!\wt{L}_{\bu}^2),&\hbox{if}~k_1\!=\!0;
\end{cases}
\end{gather}
see Figure~\ref{rhocD2_fig}.

\begin{figure}
\begin{pspicture}(-2.5,-6)(10,-.2)
\psset{unit=.4cm}
% top left 
\pscircle[linewidth=.05](2,-3){1.5}\pscircle[linewidth=.05](5,-3){1.5}
\pscircle*(3.5,-3){.1}
\pscircle*(2,-1.5){.1}\pscircle*(2,-4.5){.1}
\pscircle*(.94,-1.94){.1}\pscircle*(.94,-4.06){.1}
\pscircle*(5,-1.5){.1}\pscircle*(5,-4.5){.1}\pscircle*(6.06,-1.94){.1}
\rput(2.2,-5.2){\sm{$\wt\al_2(\wt\eta)$}}\rput(-.4,-4.3){\sm{$\wt\al_1(\wt\eta)$}}
\rput(7.3,-1.6){\sm{$\wt\al_j(\wt\eta)$}}
\psline[linewidth=.08]{->}(8,-3)(15.5,-3)\rput(12,-2.4){\sm{$\rho$}}
% top right
\pscircle[linewidth=.05](18,-3){1.5}\pscircle[linewidth=.05](21,-3){1.5}
\pscircle*(19.5,-3){.1}
\pscircle*(18,-1.5){.1}\pscircle*(18,-4.5){.1}
\pscircle*(16.94,-1.94){.1}\pscircle*(16.94,-4.06){.1}
\pscircle*(21,-1.5){.1}\pscircle*(21,-4.5){.1}\pscircle*(22.06,-1.94){.1}
\rput(21.2,-6.4){\sm{$\wt\al_1(\wt\eta')\!=\!\wt\al_2(\wt\eta)$}}
\rput(14,-6.4){\sm{$\wt\al_{k_{\bu}}(\wt\eta')\!=\!\wt\al_1(\wt\eta)$}}
\rput(25.4,-1.6){\sm{$\wt\al_{j-1}(\wt\eta')\!=\!\wt\al_j(\wt\eta)$}}
\psline[linewidth=.03]{->}(12.5,-5.8)(16.7,-4.2)
\psline[linewidth=.03]{->}(19,-6)(18.2,-4.8)
% bottom left 
\pscircle[linewidth=.05](2,-11){1.5}\pscircle[linewidth=.05](5,-11){1.5}
\pscircle*(3.5,-11){.1}
\pscircle*(2,-9.5){.1}\pscircle*(2,-12.5){.1}
\pscircle*(.94,-9.94){.1}\pscircle*(.94,-12.06){.1}
\pscircle*(5,-9.5){.1}\pscircle*(5,-12.5){.1}\pscircle*(6.06,-9.94){.1}
\rput(2.2,-13.2){\sm{$\wt\al_1(\wt\eta)$}}\rput(5.2,-13.2){\sm{$\wt\al_2(\wt\eta)$}}
\rput(7.3,-9.6){\sm{$\wt\al_j(\wt\eta)$}}
\psline[linewidth=.08]{->}(8,-11)(15.5,-11)\rput(12,-10.4){\sm{$\rho$}}
% bottom right
\pscircle[linewidth=.05](18,-11){1.5}\pscircle[linewidth=.05](21,-11){1.5}
\pscircle*(19.5,-11){.1}
\pscircle*(18,-9.5){.1}\pscircle*(18,-12.5){.1}
\pscircle*(16.94,-12.06){.1}\pscircle*(22.06,-12.06){.1}
\pscircle*(21,-9.5){.1}\pscircle*(21,-12.5){.1}\pscircle*(22.06,-9.94){.1}
\rput(16,-8){\sm{$\wt\al_1(\wt\eta')\!=\!\wt\al_2(\wt\eta)$}}       
\rput(24,-9){\sm{$\wt\al_{k_{\bu}}(\wt\eta')\!=\!\wt\al_1(\wt\eta)$}} 
\rput(14,-14.4){\sm{$\wt\al_{j-1}(\wt\eta')\!=\!\wt\al_j(\wt\eta)$}}   
\psline[linewidth=.03]{->}(12.5,-13.8)(16.7,-12.2)
\psline[linewidth=.03]{->}(14.5,-8.7)(17.8,-9.3)
\end{pspicture}
\caption{The images $(\wt\eta'\!\equiv\!\rho(\wt\eta);\vt')$ of
two elements of $\cD_{\om}^{2*}(\wt\al)$ with $k_\bu\!\equiv\!k_\bu(\wt\eta)$ nonzero
under the map~$\rho$ in~\eref{rhocD2_e}.}
\label{rhocD2_fig}
\end{figure}

For an element $(\wt\eta;\be_{\bu}^2,k_1,k_2,\wt{L}_{\bu}^2)$ of $\cD_{\om}^2(\wt\al)$, let
\begin{gather*}
K=[k_2\!-\!1]\!-\![k_1], ~ k_{\bu}^1=k_{\bu}(\wt\eta)\!-\!|K|\!+\!1, 
~ k_{\bu}^2=|K|,  ~
\wt{K}^2=\bigsqcup_{i\in K}\!\wt{K}_i(\wt\eta), ~
\wt{L}^2=\wt{L}_{\bu}^2\!\sqcup\!\bigsqcup_{i\in K}\!\wt{L}_{i}(\wt\eta),\\
\begin{split}
\be^2&=\be_{\bu}^2\!+\!\sum_{i\in K}\be_i(\wt\eta),\\
\wt\al^2&=\big(\be^2,\wt{K}^2,\wt{L}^2\big),
\end{split}
\quad
\wt\al_i^1=\begin{cases}
\wt\al_i(\wt\eta),&\hbox{if}~i\!\in\![k_1];\\
\wt\al^2,&\hbox{if}~i\!=\!k_1\!+\!1;\\
\wt\al_{i-2+k_2-k_1}(\wt\eta),&\hbox{if}~i\!\in\![k_{\bu}^1]\!-\![k_1\!+\!1].
\end{cases}
\end{gather*}
In particular, $\wt\al^2\!\in\!\wt\cC_{\om;\wt\al}(Y)$,
\BE{corrdfn_e}\begin{split}
\wt\eta^1&\equiv\big(\be_{\bu}(\wt\eta)\!-\!\be_{\bu}^2,k_{\bu}^1,
\wt{L}_{\bu}(\wt\eta)\!-\!\wt{L}_{\bu}^2,(\wt\al_i^1)_{i\in[k_{\bu}^1]}\big)
\in\cD_{\om}(\wt\al),\\
\wt\eta^2&\equiv\big(\be_{\bu}^2,k_{\bu}^2,\wt{L}_{\bu}^2,
(\wt\al_{i+k_1})_{i\in[k_{\bu}^2]}\big)\in\cD_{\om}(\wt\al^2).
\end{split}\EE
The resulting map
\BE{BCpseudo_e11c}
\cD_{\om}^{2*}(\wt\al)\lra\ov\cD_\om^{2*}(\wt\al),\quad
\big(\wt\eta;\be_{\bu}^2,k_1,k_2,\wt{L}_{\bu}^2\big)\lra
\big(\wt\eta^1;k_1\!+\!1,\wt\eta^2\big),\EE
is well-defined and injective.
Its image consists of the elements $(\wt\eta;i,\wt\eta')$ of $\ov\cD_\om^{2*}(\wt\al)$
such~that either $i\!>\!1$ or $k_{\bu}(\wt\eta)\!+\!k_{\bu}(\wt\eta')\!=\!1$.
The map~\eref{BCpseudo_e11c} descends to a bijection from the quotient of 
the left-hand side by the equivalence relation generated 
by the map~$\rho$ in~\eref{rhocD2_e}
to the quotient of the right-hand side by the equivalence relation generated 
by~$\rho$ in~\eref{rhoovcDdfn_e} and~$R$ in~\eref{Rdfn_e}.

For $(\wt\eta;\vt)\!\in\!\cD_{\om}^{2*}(\wt\al)$ with 
$\vt\!=\!(\be_{\bu}^2,k_1,k_2,\wt{L}_{\bu}^2)$, we define
$\cS_{\wt\eta}(\vt)\!\subset\!\prt\M_{\wt\eta;\wt J}$ as in~\eref{BCpseudo_e12b}.
These subspaces are distinct if $k_{\bu}(\wt\eta)\!>\!0$; otherwise, the~tuples 
\BE{Sbndtuples_e} \big(\wt\eta;\be_{\bu}^2,0,1,\wt{L}_{\bu}^2\big) \quad\hbox{and}\quad 
\big(\wt\eta;\be_{\bu}(\wt\eta)\!-\!\be_{\bu}^2,0,1,\wt{L}_{\bu}(\wt\eta)\!-\!\wt{L}_{\bu}^2\big)\EE
describe the same subspace. 
Let
$$\fbb^*_{\wt\eta}(\vt)\equiv\cS_{\wt\eta}\big(\vt\big)
\!\!\fiber\!\!\big(\!(i,\fb_{\wt\al_i(\wt\eta)})_{i\in[k_\bu(\wt\eta)]};
(i,\wt\Ga_i)_{\wt\Ga_i\in\wt{L}_\bu(\wt\eta)}\big).$$
The bijection~\eref{rhocD2_e} induces a diffeomorphism 
$$\rho_{\wt\eta;\vt}\!:\fbb^*_{\wt\eta}(\vt)\lra \fbb^*_{\rho(\wt\eta)}(\vt'),
\quad\hbox{where}~\big(\rho(\wt\eta);\vt')\equiv\rho(\wt\eta;\vt).$$
By the first statement of Lemma~\ref{rhoRsgn_lmm}, this diffeomorphism is orientation-preserving.
If $(\wt\eta^1;k_1\!+\!1,\wt\eta^2)$ is the image of $(\wt\eta;\vt)$
under~\eref{BCpseudo_e11c}, 
Corollary~\ref{Scompare_crl} with $B\!=\![0,1]$ and $k,|I|\!=\!k_{\bu}(\wt\eta)$ gives
\BE{BCpseudo_e10c}\begin{split}
(-1)^{\binom{k_{\bu}(\wt\eta)+1}{2}}\fbb_{\wt\eta}^*(\vt)
&=(-1)^{\binom{k_{\bu}(\wt\eta)+1}{2}}\cdot(-1)^{k_2+k_{\bu}(\wt\eta)(k_1+k_2)}
\cdot(-1)^{\binom{k_{\bu}(\wt\eta^2)}{2}}\fbb_{\wt\eta^1}^*\big(k_1\!+\!1,\wt\eta^2\big)\\
&=-(-1)^{\ep(\wt\eta^1,k_1+1)}\fbb_{\wt\eta^1}^*\big(k_1\!+\!1,\wt\eta^2\big).
\end{split}\EE

For $\wt\eta\!\in\!\cD_{\om}(\wt\al)$, let 
$$s'(\wt\eta)=\begin{cases}s^*(\wt\eta),&\hbox{if}~k_\bu(\wt\eta)\!\neq\!0,\\
\frac12s^*(\wt\eta),&\hbox{if}~k_\bu(\wt\eta)\!=\!0.
\end{cases}$$
Similarly to~\eref{wtcancel_e}, 
\BE{wtcancel_ec}\begin{split}
\prt\bigg(\bigsqcup_{\wt\eta\in\cD_\om(\wt\al)}\hspace{-.15in}
s^*(\wt\eta)\fbb_{\wt\eta}^*\bigg)\bigg|_{(0,1)}
=&\!\!
\bigsqcup_{(\wt\eta;\vt)\in\cD_{\om}^2(\wt\al)}\hspace{-.22in}
(-1)^{\binom{k_{\bu}(\wt\eta)+1}{2}}s'(\wt\eta)\fbb_{\wt\eta}(\vt)\\
&\qquad
\sqcup\!\!\!
\bigsqcup_{\begin{subarray}{c}(\wt\eta^1;i,\wt\eta^2)\in\ov\cD_{\om}^{2*}(\wt\al)\\
-2<\dim(\wt\al_i(\wt\eta^1))\le n-2\end{subarray}}\hspace{-.48in}
(-1)^{\ep(\wt\eta^1,i)}s^*(\wt\eta^1)
\fbb_{\wt\eta^1}\big(i,\wt\eta^2\big)\,;
\end{split}\EE
if $(\wt\eta^1;i,\wt\eta^2)\!\in\!\ov\cD_{\om}^2(\al)\!-\!\ov\cD_{\om}^{2*}(\al)$, 
$\fbb_{\eta^1}^*(i,\wt\eta^2)\!=\!\eset$.

Let $(\eta;\vt)\!\in\!\cD_{\om}^{2*}(\wt\al)$, 
$(\wt\eta^1;i,\wt\eta^2)\!\in\!\ov\cD_{\om}^{2*}\!(\wt\al)$
be its image under~\eref{BCpseudo_e11c}, and
$$\big(\wh\eta^1;1,\wt\eta^1\bsl i\big)\equiv R\big(\wt\eta^1;i,\wt\eta^2\big)
\in\ov\cD_{\om}^{2*}(\wt\al).$$
Suppose $(\wt\eta^1;i,\wt\eta^2)$ satisfies both inequalities in~\eref{wtcancel_ec}
and $(\wh\eta^1;1,\wt\eta^1\bsl i)$ satisfies the corresponding inequalities.
By~\eref{BCpseudo_e10c}, the term $\fbb_{\wt\eta}(\vt)$ on the first line in~\eref{wtcancel_ec}
then appears with the opposite orientation of the corresponding term 
on the second line.
The $\rho$-orbits of $(\wt\eta^1;i,\wt\eta^2)$ and $(\wh\eta^1;1,\wt\eta^1\bsl i)$
in~$\ov\cD_{\om}^{2*}(\wt\al)$
contain $k_{\bu}(\wt\eta^1)$ and $k_{\bu}(\wh\eta^1)$ elements, respectively.
By the last two statements of Lemma~\ref{rhoRsgn_lmm}, all 
\hbox{$k_{\bu}(\wt\eta^1)s^*(\wt\eta^1)\!+\!k_{\bu}(\wh\eta^1)s^*(\wh\eta^1)$} 
associated copies of $\fbb_{\wt\eta^1}(i,\wt\eta^2)$ appear 
on the second line of~\eref{wtcancel_ec} with the same orientation.
If $k_{\bu}(\wt\eta)\!>\!0$, the $\rho$-orbit of~$(\wt\eta;\vt)$ 
in~$\cD_{\om}^{2*}(\wt\al)$ contains $k_{\bu}(\wt\eta)$~elements.
By the first statement of Lemma~\ref{rhoRsgn_lmm}, all $k_{\bu}(\wt\eta)s'(\wt\eta)$
associated copies of~$\fbb_{\wt\eta}(\vt)$ appear
on the first line of~\eref{wtcancel_ec} with the same orientation.
In this case,
$$k_{\bu}(\wt\eta^1)s^*(\wt\eta^1)\!+\!k_{\bu}(\wh\eta^1)s^*(\wh\eta^1)
=k_{\bu}(\wt\eta)s'(\wt\eta).$$
If $k_{\bu}(\wt\eta)\!=\!0$, the $\rho$-equivalence class of~$(\wt\eta;\vt)$ consists of
two elements as in~\eref{Sbndtuples_e}, 
which describe the same space on the first line in~\eref{wtcancel_ec}.
In this case,
$$k_{\bu}(\wt\eta^1)s^*(\wt\eta^1)\!+\!k_{\bu}(\wh\eta^1)s^*(\wh\eta^1)= 2s'(\wt\eta).$$ 
In either case, we conclude that the boundary terms corresponding to 
the elements of the $\rho$-equivalence class of~$(\wt\eta;\vt)$ 
in~$\cD_{\om}^{2*}(\wt\al)$ on the first line in~\eref{wtcancel_ec}
cancel with the boundary terms 
corresponding to the elements of the $(\rho,R)$-equivalence class 
of~$(\wt\eta^1;i,\wt\eta^2)$ in~$\ov\cD_{\om}^{2*}(\wt\al)$ on the second~line.

Suppose $(\wt\eta^1;i,\wt\eta^2)$ does not satisfy the first inequality in~\eref{wtcancel_ec}.
Similar reasoning to that in the proof of Lemma~\ref{psisot_lmm} then implies that 
the boundary term $\fbb_{\wt\eta}(\vt)$ on the first line in~\eref{wtcancel_ec}
is either empty or cancels with another boundary term~$\fbb_{\wt\eta'}(\vt')$;
the same happens with the term $\fbb_{\wh\eta^1}(1,\wt\eta^1\bsl i)$ on the second line.
The same reasoning with the two disk components interchanged applies if 
$(\wh\eta^1;1,\wt\eta^1\bsl i)$ does not satisfy the analogue of 
the second inequality in~\eref{wtcancel_ec}.

By the last two paragraphs and~\eref{Rdimprp_e},
\BE{wtcancel_ec2}
\prt\bigg(\bigsqcup_{\wt\eta\in\cD_\om(\wt\al)}\hspace{-.15in}
s^*(\wt\eta)\fbb_{\wt\eta}^*\bigg)\bigg|_{(0,1)}=\eset\EE
if $\dim(\wt\al)\!=\!0$. 
By the same reasoning as at the end of the proof of Lemma~\ref{psisot_lmm},
$$\prt\fbb_{\wt\eta}^*
=(-1)^{\binom{k_{\bu}(\wt\eta)}{2}}\!\cdot\!(-1)^{k_{\bu}(\wt\eta)
+\binom{k_{\bu}(\wt\eta)+2}{2}-1}
\cdot(-1)^{k_{\bu}(\wt\eta)}\big(\{1\}\!\times\!\fbb_{1;\eta_1}^*\!-\!
\{0\}\!\times\!\fbb_{0;\eta_0}^*\big)+\prt\fbb_{\wt\eta}^*\big|_{(0,1)}$$
with $\eta_0,\eta_1$ as in~\eref{eta01dfn_e}.
This establishes~\eref{cob_e}. 
\end{proof}

\subsection{Equivalence of definitions of disk counts}
\label{countequiv_subs}

We next complete the proof of Theorem~\ref{countinv_thm}\ref{equivdfn_it} 
by establishing~\eref{equivdfn_e0a} and~\eref{equivdfn_e0c} under 
the assumption that $\dim(\al)\!=\!0$. 
For $\eta\!\in\!\cD_{\om}(\al)$ and $i\!\in\![k_{\bu}(\eta)]$,
we define $\eta\bsl i\!\in\!\cD_{\om}(\al\!-\!\al_i(\eta)\!)$ 
by~\eref{etabslidfn_e} with~$\wt\al$ replaced by~$\al$.
Similarly to the second statement of Lemma~\ref{rhoRsgn_lmm}, 
\BE{countequiv_e1}\begin{split}
&\M_{\eta;J}\!\fiber\!\!
\big(\!(j,\fb_{\al_j(\eta)})_{j\in[k_\bu(\eta)]-\{i\}};
(j,\Ga_j)_{\Ga_j\in L_\bu(\eta)}\big)\\
&\hspace{1.5in}\approx
(-1)^{i-1}
\M_{\eta\bsl i;J}^+\!\fiber\!\!
\big(\!(j\!+\!1,\fb_{\al_j(\eta\bsl i)})_{j\in[k_\bu(\eta\bsl i)]};
(j,\Ga_j)_{\Ga_j\in L_\bu(\eta\bsl i)}\big)
\end{split}\EE
Let
$$\al_{\pt}\equiv(0,\{\pt\},\eset)\!\in\!\cC_{\om}(Y)~~\forall\,\pt\!\in\!Y, \quad
\cC_{\om}^*(Y)=\big\{\al\!\in\!\cC_{\om}(Y)\!:\al\!\neq\!\al_{\pt}
~\forall\,\pt\!\in\!K\big\}.$$

\begin{proof}[{\bf{\emph{Proof of~\eref{equivdfn_e0a}}}}]
Since the dimension of $\fb_{\al'}$ is even for every $\al'\!\in\!\cC_{\om;\al}(Y)$ 
by~\eref{bdimprp_e}
and the dimension of~$Y$ is odd,
Lemma~\ref{fibersign_lmm}, Definition~\ref{bndch_dfn}\ref{BCprt_it},
and Lemma~\ref{fibprodflip_lmm} give
\begin{gather}
\label{ountequiv_e2}-\prt\big(\fb_{\al_1}\!\!\fiber\!\fb_{\al_2}\big)
=\fbb_{\al_1}\!\!\fiber\!\fb_{\al_2}\!+\!
\fb_{\al_1}\!\!\fiber\!\fbb_{\al_2}
=\fbb_{\al_1}\!\!\fiber\!\fb_{\al_2}\!-\!\fbb_{\al_2}\!\!\fiber\!\fb_{\al_1}\\
\notag
\forall\quad \al_1,\al_2\in \cC_{\om;\al}(Y) ~~\hbox{s.t.}~~
 \dim(\al_1),\dim(\al_2)\le n\!-\!2.
\end{gather}
By Definition~\ref{bndch_dfn}\ref{BC0_it},
the first equality above also holds~if
$$\al_1,\al_2\in\cC_{\om}^*(Y) \quad\hbox{and}\quad
\dim(\al_1)\!+\!\dim(\al_2)
=n\!-\!3~~\big(\hbox{i.e.}~\dim\,\fbb_{\al_1}\!\!\fiber\!\fb_{\al_2}=0\big);$$
the second equality holds for all~$\al_1,\al_2$.
If $\pt\!\in\!K(\al)$ and $\dim(\al)\!=\!0$, 
\eref{ountequiv_e2} implies that 
\BE{ountequiv_e3}\begin{split}
\bigsqcup_{\begin{subarray}{c}\al_1,\al_2\in\cC_\om^*(Y)\\ 
\al_1+\al_2=\al\end{subarray}}\hspace{-.27in}\fbb_{\al_1}\!\!\fiber\!\fb_{\al_2}
&=\bigsqcup_{\begin{subarray}{c}\al_1,\al_2\in\cC_\om^*(Y)\\ 
\al_1+\al_2=\al\\ \pt\in K(\al_1)\end{subarray}}\hspace{-.27in}
\big(\fbb_{\al_2}\!\!\fiber\!\fb_{\al_1}\!\sqcup\!\prt(-\fb_{\al_1}\!\!\fiber\!\fb_{\al_2})\!\big)
\!\sqcup\!
\bigsqcup_{\begin{subarray}{c}\al_1,\al_2\in\cC_\om^*(Y)\\ 
\al_1+\al_2=\al\\ \pt\in K(\al_2)\end{subarray}}\hspace{-.27in}\fbb_{\al_1}\!\!\fiber\!\fb_{\al_2}\\
&=2\hspace{-.2in} 
\bigsqcup_{\begin{subarray}{c}\al_1,\al_2\in\cC_\om^*(Y)\\ 
\al_1+\al_2=\al\\ \pt\in K(\al_2)\end{subarray}}\hspace{-.27in}\fbb_{\al_1}\!\!\fiber\!\fb_{\al_2}
\sqcup~
\prt\hspace{-.2in} 
\bigsqcup_{\begin{subarray}{c}\al_1,\al_2\in\cC_\om^*(Y)\\ 
\al_1+\al_2=\al\\ \pt\in K(\al_1)\end{subarray}}\hspace{-.27in}
\big(\!-\!\fb_{\al_1}\!\!\fiber\!\fb_{\al_2}\big).
\end{split}\EE
Since the dimension of $\fbb_{\al'}$ is odd for every $\al'\!\in\!\cC_{\om;\al}(Y)$, 
\BE{ountequiv_e5}
-\big|\fbb_{\al_{\pt}^c}\!\!\fiber\!\fb_{\al_{\pt}}\big|^{\pm}
=\deg\,\fbb_{\al_{\pt}^c}
=\lr{L}_{\be;K(\al)-\{\pt\}}^{\om,\os}
\quad\forall\,\pt\!\in\!K(\al).\EE

\vspace{-.1in}

If $\eta\!\in\!\cD_\om(\al)$ and $i\!\in\![k_\bu(\eta)]$,
then
\begin{equation*}\begin{split}
&\M_{\eta;J}\!\fiber\!\!\big(\!(j,\fb_{\al_j(\eta)})_{j\in[k_\bu(\eta)]};
(j,\Ga_j)_{\Ga_j\in L_\bu(\eta)}\big)\\
&=\unbr{(-1)^{k_\bu(\eta)-i}}{\hbox{\tiny Lemma~\ref{fibprodflip_lmm}}}
\unbr{(-1)^{k_\bu(\eta)-1}}{\hbox{\tiny Lemma~\ref{fibprodisom_lmm1}}}\!
\Big(\!\M_{\eta;J}\!\!\fiber\!\!\big(\!(j,\fb_{\al_j(\eta)})_{j\in[k_\bu(\eta)]-\{i\}};
(j,\Ga_j)_{\Ga_j\in L_\bu(\eta)}\big)\!\!\Big){}_{\evb_i}\!\!\times_{\fb_{\al_i(\eta)}}\!\!
\big(\dom\,\fb_{\al_i(\eta)}\big)\\
&=(-1)^{i-1}\unbr{(-1)^{i-1}}{\hbox{\tiny\eref{countequiv_e1}}}
\Big(\!(-1)^{k_\bu(\eta\backslash i)}\fbb_{\eta\bsl i}\!\Big)\!\!\fiber\!\fb_{\al_i(\eta)}
=(-1)^{k_\bu(\eta)-1}\fbb_{\eta\bsl i}\!\!\fiber\!\fb_{\al_i(\eta)}.
\end{split}\end{equation*}
Taking $i\!=\!1$ above, we obtain 
\BE{ountequiv_e9}\begin{split}
&-\!\!\!\bigsqcup_{\eta\in\cD_\om(\al)}\!\!\!\!\!(-1)^{k_\bu(\eta)}
\M_{\eta;J}\!\fiber\!\!\big(\!(j,\fb_{\al_j(\eta)})_{j\in[k_\bu(\eta)]};
(j,\Ga_j)_{\Ga_j\in L_\bu(\eta)}\big)\\
&\hspace{1in}
=\bigsqcup_{\eta\in\cD_\om(\al)}\!\!\!\!\!\!\fbb_{\eta\bsl1}\!\!\fiber\!\fb_{\al_1(\eta)}
=\bigsqcup_{\begin{subarray}{c}\al_1,\al_2\in\cC_\om^*(Y)\\ 
\al_1+\al_2=\al\end{subarray}}\hspace{-.27in}\fbb_{\al_1}\!\!\fiber\!\fb_{\al_2}
\sqcup \bigsqcup_{\pt'\in K(\al)}\hspace{-.15in}\fbb_{\al_{\pt'}^c}\!\!\fiber\!\fb_{\al_{\pt'}}\,.
\end{split}\EE

\vspace{-.15in}

If in addition $\pt\!\in\!K(\al)$, then
\begin{equation*}\begin{split}
&-\!\!\!\bigsqcup_{\eta\in\cD_\om(\al)}\!\!\!\!\!\frac{(-1)^{k_\bu(\eta)}}{k_\bu(\eta)}
\M_{\eta;J}\!\fiber\!\!\big(\!(j,\fb_{\al_j(\eta)})_{j\in[k_\bu(\eta)]};
(j,\Ga_j)_{\Ga_j\in L_\bu(\eta)}\big)\\
&\hspace{.2in}=
\bigsqcup_{\begin{subarray}{c}\eta\in\cD_\om(\al)\\ i\in[k_{\bu}(\eta)]\\ 
\al_i(\eta)=\al_{\pt}\end{subarray}}\!\!\frac1{k_\bu(\eta)}
\fbb_{\eta\bsl i}\!\!\fiber\!\fb_{\al_{\pt}}
\sqcup\!\!\!\!\!
\bigsqcup_{\begin{subarray}{c}\eta\in\cD_\om(\al)\\ i\in[k_{\bu}(\eta)]\\
\pt\in K(\al_i(\eta)),\al_i(\eta)\neq\al_{\pt}\end{subarray}}\hspace{-.35in}\frac1{k_\bu(\eta)}
\fbb_{\eta\bsl i}\!\!\fiber\!\fb_{\al_i(\eta)}\\
&\hspace{.2in}=
\bigsqcup_{\begin{subarray}{c}\eta\in\cD_\om(\al)\\ \al_1(\eta)=\al_{\pt}\end{subarray}}
\!\!\!\!\!\!\!\!\fbb_{\eta\bsl1}\!\!\fiber\!\fb_{\al_{\pt}}
\sqcup\!\!
\bigsqcup_{\begin{subarray}{c}\eta\in\cD_\om(\al)\\
\pt\in K(\al_1(\eta)),\al_1(\eta)\neq\al_{\pt}\end{subarray}}
\hspace{-.54in}\fbb_{\eta\bsl1}\!\!\fiber\!\fb_{\al_1(\eta)}
=\fbb_{\al_{\pt}^c}\!\!\fiber\!\fb_{\al_{\pt}}
\sqcup\!\!\bigsqcup_{\begin{subarray}{c}\al_1,\al_2\in\cC_\om^*(Y)\\ 
\al_1+\al_2=\al\\ \pt\in K(\al_2)\end{subarray}}\hspace{-.27in}\fbb_{\al_1}\!\!\fiber\!\fb_{\al_2}
\,.
\end{split}\end{equation*}
If $\dim(\al)\!=\!0$,
this statement, \eref{ountequiv_e9}, and~\eref{ountequiv_e3} give
\begin{equation*}\begin{split}
&\bigsqcup_{\eta\in\cD_\om(\al)}\!\!\!\!\!(-1)^{k_\bu(\eta)}s^*(\eta)
\M_{\eta;J}\!\fiber\!\!\big(\!(j,\fb_{\al_j(\eta)})_{j\in[k_\bu(\eta)]};
(j,\Ga_j)_{\Ga_j\in L_\bu(\eta)}\big)\\
&\hspace{1in}=-\fbb_{\al_{\pt}^c}\!\!\fiber\!\fb_{\al_{\pt}}
\sqcup\frac12\!\!
\bigsqcup_{\pt'\in K(\al)}\hspace{-.15in}\fbb_{\al_{\pt'}^c}\!\!\fiber\!\fb_{\al_{\pt'}}
\sqcup\frac12\prt\hspace{-.2in} 
\bigsqcup_{\begin{subarray}{c}\al_1,\al_2\in\cC_\om^*(Y)\\ 
\al_1+\al_2=\al\\ \pt\in K(\al_1)\end{subarray}}\hspace{-.27in}
\big(\!-\!\fb_{\al_1}\!\!\fiber\!\fb_{\al_2}\big)\,.
\end{split}\end{equation*}
Along with~\eref{JSinvdfn_e2b} and~\eref{ountequiv_e5}, this implies~\eref{equivdfn_e0a}.
\end{proof}

\begin{proof}[{\bf{\emph{Proof of~\eref{equivdfn_e0c}}}}]
By the proof of~\eref{ountequiv_e9},
\begin{equation*}\begin{split}
&\bigsqcup_{\begin{subarray}{c}\eta\in\cD_\om(\al)\\
\Ga\not\in L_{\bu}(\eta)\end{subarray}}
\!\!\!\!(-1)^{k_\bu(\eta)}\!s^{\circ}(\eta)
\M_{\eta;J}\!\fiber\!\!\big(\!(j,\fb_{\al_j(\eta)})_{j\in[k_\bu(\eta)]};
(j,\Ga_j)_{\Ga_j\in L_\bu(\eta)}\big)\\
&\hspace{.4in}
=-\!\!\!\!\!
\bigsqcup_{\begin{subarray}{c}\eta\in\cD_\om(\al)\\i\in[k_{\bu}(\eta)]\\
\Ga\in L(\al_i(\eta)\!)\end{subarray}}\hspace{-.1in}\frac1{k_\bu(\eta)}
\fbb_{\eta\bsl i}\!\!\fiber\!\fb_{\al_i(\eta)}
=-\!\!\!\!\!\bigsqcup_{\begin{subarray}{c}\eta\in\cD_\om(\al)\\
\Ga\in L(\al_1(\eta)\!)\end{subarray}}\hspace{-.2in}
\fbb_{\eta\bsl 1}\!\!\fiber\!\fb_{\al_1(\eta)}
=-\!\!\!\!\!\bigsqcup_{\begin{subarray}{c}\al_1,\al_2\in\cC^*_\om(Y)\\
\al_1+\al_2=\al\\ \Ga\in L(\al_2)\end{subarray}}\hspace{-.27in}
\fbb_{\al_1}\!\!\fiber\!\fb_{\al_2}\,.
\end{split}\end{equation*}
Suppose in addition $\dim(\al)\!=\!0$.
By the proof of~\eref{ountequiv_e3},
$$\bigsqcup_{\begin{subarray}{c}\al_1,\al_2\in\cC_\om^*(Y)\\ 
\al_1+\al_2=\al\end{subarray}}\hspace{-.27in}\fbb_{\al_1}\!\!\fiber\!\fb_{\al_2}
=2\hspace{-.2in} 
\bigsqcup_{\begin{subarray}{c}\al_1,\al_2\in\cC_\om^*(Y)\\ 
\al_1+\al_2=\al\\ \Ga\in L(\al_2)\end{subarray}}\hspace{-.27in}\fbb_{\al_1}\!\!\fiber\!\fb_{\al_2}
\sqcup~
\prt\hspace{-.2in} 
\bigsqcup_{\begin{subarray}{c}\al_1,\al_2\in\cC_\om^*(Y)\\ 
\al_1+\al_2=\al\\ \Ga\in L(\al_1)\end{subarray}}\hspace{-.27in}
\big(\!-\!\fb_{\al_1}\!\!\fiber\!\fb_{\al_2}\big).$$
Combining the last two statements with \eref{ountequiv_e9}, 
\begin{equation*}\begin{split}
&\bigsqcup_{\eta\in\cD_\om(\al)}\!\!\!\!\!(-1)^{k_\bu(\eta)}s^*(\eta)
\M_{\eta;J}\!\fiber\!\!\big(\!(j,\fb_{\al_j(\eta)})_{j\in[k_\bu(\eta)]};
(j,\Ga_j)_{\Ga_j\in L_\bu(\eta)}\big)\\
&\hspace{1in}=\!\!\!
\bigsqcup_{\eta\in\cD_\om^{\Ga}(\al)}\!\!\!\!\!(-1)^{k_\bu(\eta)}s^\circ(\eta)
\M_{\eta;J}\!\fiber\!\!\big(\!(j,\fb_{\al_j(\eta)})_{j\in[k_\bu(\eta)]};
(j,\Ga_j)_{\Ga_j\in L_\bu(\eta)}\big)\\
&\hspace{2in}
\sqcup\frac12\!\!
\bigsqcup_{\pt'\in K(\al)}\hspace{-.11in}\fbb_{\al_{\pt'}^c}\!\!\fiber\!\fb_{\al_{\pt'}}
\sqcup\frac12\prt\hspace{-.2in} 
\bigsqcup_{\begin{subarray}{c}\al_1,\al_2\in\cC_\om^*(Y)\\ 
\al_1+\al_2=\al\\ \Ga\in L(\al_1)\end{subarray}}\hspace{-.27in}
\big(\!-\!\fb_{\al_1}\!\!\fiber\!\fb_{\al_2}\big)\,.
\end{split}\end{equation*}
Along with~\eref{JSinvdfn_e2b}, \eref{JSinvdfn_e2c}, and
\eref{ountequiv_e5}, this implies~\eref{equivdfn_e0c}.
\end{proof}

\section{Proof of Theorem~\ref{main_thm}}
\label{GWmainpf_sec}

We confirm the multilinearity of the open invariants~\eref{geomJSinvdfn_e},
the pair of properties determining these invariants for extreme values 
of the degree and constraints, \ref{deg0_it} and~\ref{ins1_it}, and 
the pair of properties involving geometrically special insertions,
\ref{div_it} and~\ref{sphere_it}, in Section~\ref{GWmainpf_subs1}.
The two remaining properties, both of which involve topological properties
of the Lagrangian~$Y$ in~$X$, are established in Section~\ref{LangOGW_subs}.

\subsection{Multilinearity and Properties \ref{deg0_it}-\ref{sphere_it}}
\label{GWmainpf_subs1}

Let $(X,\om,Y)$, $\os$, $k,l$, $\be$ be as in the statement of Theorem~\ref{main_thm}.
We fix $J\!\in\!\cJ_{\om}$ and $K\!\subset\!Y$ with $|K|\!=\!k$.
Given an element $\ga_i$ of $\wh{H}^{2*}(X,Y;R)$, 
we take a pseudocycle~$\Ga_i$ to~$X$ or $X\!-\!Y$ representing
the Poincare dual of~$\ga_i$. 
We assume that $J$,  $K$, and $L\!\equiv\{\Ga_1,\ldots,\Ga_l\}$ are chosen generically,
set $\al\!\equiv\!(\be,K,L)$, 
and take a bounding chain $(\fb_{\al'})_{\al'\in\cC_{\om;\al}(Y)}$ 
on $(\al,J)$.

\begin{proof}[{\bf{\emph{Proof of~multilinearity}}}]
By the symmetry of the open invariants~\eref{geomJSinvdfn_e} and~\eref{OGWdfn_e},
it is sufficient to establish the linearity of these invariants only in 
the first input.
Suppose $a',a''\!\in\!R$, $\Ga_1',\Ga_1''$ are generic pseudocycles to~$X$, 
and \hbox{$\Ga_1\!=\!a'\Ga_1'\!+\!a''\Ga_2''$}.
By the linearity of the intersection numbers,
\begin{equation*}\begin{split}
&\Big|\M_{\eta;J}\!\fiber\!\!\big(\!(i,\fb_{\al_i(\eta)})_{i\in[k_\bu(\eta)]};
(1,\Ga_1),(i,\Ga_i)_{i\in L_\bu(\eta)-\{\Ga_1\}}\big)\Big|^{\pm}\\
&\hspace{1.2in}=
a'\Big|\M_{\eta;J}\!\fiber\!\!\big(\!(i,\fb_{\al_i(\eta)})_{i\in[k_\bu(\eta)]};
(1,\Ga_1'),(i,\Ga_i)_{\Ga_i\in L_\bu(\eta)-\{\Ga_1\}}\big)\Big|^{\pm}\\
&\hspace{1.5in}+
a''\Big|\M_{\eta;J}\!\fiber\!\!\big(\!(i,\fb_{\al_i(\eta)})_{i\in[k_\bu(\eta)]};
(1,\Ga_1''),(i,\Ga_i)_{\Ga_i\in L_\bu(\eta)-\{\Ga_1\}}\big)\Big|^{\pm}
\end{split}\end{equation*}
for every $\eta\!\in\!\cD_{\om}^{\Ga_1}(\al)$.
This implies~that 
\BE{multilin_e1}
\blr{L\!-\!\{\Ga\}}_{\be;K}^{\Ga_1}
=a'\blr{L\!-\!\{\Ga\}}_{\be;K}^{\Ga_1'}\\+a''\blr{L\!-\!\{\Ga\}}_{\be;K}^{\Ga_1''}\,,
\EE
if all three counts  above arise from the bounding chain 
$(\fb_{\al'})_{\al'\in\cC_{\om;\al_{\Ga_1}^c}\!(Y)}$ on $(\al_{\Ga_1}^c,J)$.
This in particular establishes the multilinearity of the open invariants~\eref{geomJSinvdfn_e}
and~\eref{OGWdfn_e} under the independence assumptions of Theorem~\ref{main_thm}.
\end{proof}

The proofs of the two properties determining the open invariants~\eref{OGWdfn_e}
for extreme values of the degree and constraints,
\ref{deg0_it} and~\ref{ins1_it}, below are closely related.

\begin{proof}[{\bf{\emph{Proof of~\ref{deg0_it}}}}]
Let $\al\!=\!(0,K,L)$.
Since every degree~0 $J$-holomorphic map from $(\D^2,S^1)$ to~$(X,Y)$ is constant,
\BE{deg0M_e}\M_{\eta;J}\approx Y\!\times\!\cM_{k_{\bu}(\eta),L_{\bu}(\eta)}
\qquad\forall\,\eta\!\in\!\cD_{\om}(\al),\EE
where $\cM_{k_{\bu}(\eta),L_{\bu}(\eta)}$ is the moduli space of disks with 
$L_{\bu}(\eta)$ interior marked points and $k_{\bu}(\eta)$ boundary marked points ordered
by the position.
Both moduli spaces in~\eref{deg0M_e} are oriented in Section~\ref{Ms_subs}.
By the definition of the orientation on~$\M_{\eta;J}$ and
the CROrient~5a and~6a properties in \cite[Section~7.2]{SpinPin},
the isomorphism~\eref{deg0M_e} is orientation-preserving.
This implies that the degree of the~map
$$\evb_1\!: \M^\st_{1,1}(0;J)_{\evi_1}\!\!\times_{\id_X}\!\!X \lra Y$$
is~1 and establishes \ref{deg0_it} in the case $(k,l)\!=\!(1,1)$.

We now treat the case $(k,l)\!\neq\!(1,1)$ under the assumption
that \eref{dimcond_e} with $\be\!=\!0$ holds.
Since $\be\!=\!0$, $k\!\ge\!1$ by the condition~\eref{nosphbubb_e} and 
$\deg\ga_{i^*}\!=\!0$ for some $i^*\!\in\![l]$ by~\eref{dimcond_e}.
By the symmetry of the invariants~\eref{geomJSinvdfn_e}, we can assume that $i^*\!=\!1$.
If $k\!>\!1$, then $l\!>\!1$ by~\eref{dimcond_e}.
Thus, we can assume that $l\!\ge\!2$.

Suppose $\eta\!\in\!\cD_{\om}^{\Ga_1}(\al)$ and the fiber product in~\eref{JSinvdfn_e2c} 
corresponding to~$\eta$ is nonempty.
By~\eref{deg0M_e}, the fiber product in~\eref{JSinvdfn_e2c}
is the product of $\cM_{k_{\bu}(\eta),L_{\bu}(\eta)}$ with another space.
Since this fiber product has dimension~0, it then follows that
\BE{deg0_e3}\al_1\equiv\big(\be_{\bu}(\eta),k_{\bu}(\eta)\!-\!1,L_{\bu}(\eta)\!\big)
= \big(0,0,\{\Ga_1\}\!\big)
\qquad\hbox{and}\qquad 
\al_2\equiv\al_1(\eta)\in\cC_{\om}^*(Y).\EE
By the reasoning in the proof of the first equality in~\eref{ountequiv_e9}, 
\BE{deg0_e5}\big|\M_{\eta;J}\!\fiber\!\!\big(\!(1,\fb_{\al_2});(1,\Ga_1)\!\big)\big|^{\pm}
=\big|\fbb_{\al_1}\!\!\fiber\!\fb_{\al_2}\big|^{\pm}\,.\EE
Since $\dim(\al)\!=\!0$ and $\dim(\al_1)\!=\!n\!-\!1$, $\dim(\al_2)\!=\!-2$.
Thus, $\fb_{\al_2}\!=\!\eset$ by Definition~\ref{bndch_dfn}\ref{BC0_it}
and the number~\eref{deg0_e5} in fact vanishes.
This implies the vanishing of the number~\eref{JSinvdfn_e2c}.
In light of Theorem~\ref{countinv_thm}\ref{equivdfn_it}, 
the latter in turn implies the vanishing of 
the numbers~\eref{JSinvdfn_e2} and~\eref{JSinvdfn_e2b} under the same conditions.

The statement of~\ref{deg0_it} depends on the boundary insertions being 
of dimension~0.
The argument in the second case of~\ref{deg0_it} does not depend on~$n$ being odd.
\end{proof}

\begin{proof}[{\bf{\emph{Proof of~\ref{ins1_it}}}}]
By~\ref{deg0_it}, it remains to consider the case $\be\!\neq\!0$ under 
the assumption that \eref{dimcond_e} holds.
Let $\eta\!\in\!\cD_{\om}^{\Ga_1}(\al)$ and thus $\Ga_1\!\in\!L_{\bu}(\eta)$.
Suppose first that 
$$\big(\be_{\bu}(\eta),k_{\bu}(\eta),L_{\bu}(\eta)\!\big)\neq  
\big(0,1,\{\Ga_1\}\!\big),\big(0,2,\{\Ga_1\}\!\big).$$
Since $\ga_1\!=\!1\!\in\!H^2(X;R)$,
the forgetful morphism dropping the first interior marked point then induces a fibration
\begin{equation*}\begin{split}
&\M_{\eta;J}\!\fiber\!\!\big(\!(i,\fb_{\al_i(\eta)})_{i\in[k_\bu(\eta)]};
(i,\Ga_i)_{\Ga_i\in L_\bu(\eta)}\big)\\
&\hspace{.5in}\lra
\M^\st_{k_\bu(\eta),L_\bu(\eta)-\{\Ga_1\}}(\be_\bu(\eta);J)
\!\fiber\!\!\big(\!(i,\fb_{\al_i(\eta)})_{i\in[k_\bu(\eta)]};
(i,\Ga_i)_{\Ga_i\in L_\bu(\eta)-\{\Ga_1\}}\big)
\end{split}\end{equation*}
with two-dimensional fibers.
If $(\be_{\bu}(\eta),k_{\bu}(\eta),L_{\bu}(\eta)\!)\!=\!(0,2,\{\Ga_1\})$, 
the first fiber product above is the product of $\ov\cM_{2,1}$ with another space
by~\eref{deg0M_e}.
Since the first fiber product above is of dimension~0 by the assumption that \eref{dimcond_e} 
holds, neither of these two conclusions is possible if it is nonempty.

Suppose  $(\be_{\bu}(\eta),k_{\bu}(\eta),L_{\bu}(\eta))\!=\!(0,1,\{\Ga_1\})$.
Since $\be\!\neq\!0$, it follows that~\eref{deg0_e3} holds in this case as well.
Thus, \eref{deg0_e5} and the four sentences immediately after also apply in this case.
So do the considerations in the last paragraph of the proof of~\ref{deg0_it}.
\end{proof}

The proof of the divisor relation, \ref{div_it}, stays below
within the same definition, \eref{JSinvdfn_e2}, \eref{JSinvdfn_e2b}, or~\eref{JSinvdfn_e2c}, 
of the invariants~\eref{geomJSinvdfn_e} 
and inductively compares all closed pseudocycles~\eref{fbbdfn_e} 
associated with the two open invariants in~\ref{div_it}.
The proof of~\ref{sphere_it} makes use of Theorem~\ref{countinv_thm}\ref{equivdfn_it} 
identifying two of the definitions of the open invariants~\eref{OGWdfn_e}
and adapts the proof of Proposition~2.1 of~\cite{RealWDVV3},
establishing the same property in a related setting, to the present situation.

\begin{proof}[{\bf{\emph{Proof of~\ref{div_it}}}}]
Let $\Ga_0$ be a generic pseudocycle to $X\!-\!Y$ representing
the Poincare dual of~$\ga_0$.
Define
\BE{div_e0}L^{\dag}=\{\Ga_0\}\!\sqcup\!L, \qquad 
\al^{\dag}=\big(\be,K,L^{\dag}\big).\EE
An element $\al'\!\in\!\cC_{\om;\al^{\dag}}(Y)$ then lies in $\cC_{\om;\al}(Y)$
if and only if $\Ga_0\!\not\in\!L(\al')$.
For $\al'\!\in\!\cC_{\om;\al^{\dag}}(Y)$, let
\BE{div_e1}
\al'^c_{\Ga_0}\equiv\big(\be(\al'),K(\al'),L(\al')\!-\!\{\Ga_0\}\!\big),\quad
\fb_{\al'}'\equiv\begin{cases}
\fb_{\al'},&\hbox{if}~\Ga_0\!\not\in\!L(\al');\\
\lr{\ga_0,\be(\al')}\fb_{\al_{\Ga_0}'^c},&\hbox{if}~\Ga_0\!\in\!L(\al').
\end{cases}\EE
Let $\fbb_{\al'}'$ be the map determined by the collection 
$(\fb_{\al''}')_{\al''\in\cC_{\om;\al^{\dag}}(Y)}$ as in~\eref{fbbdfn_e}.
We show below~that
\BE{div_e3} \fbb_{\al'}'=\begin{cases}
\fbb_{\al'},&\hbox{if}~\Ga_0\!\not\in\!L(\al');\\
\lr{\ga_0,\be(\al')}\fbb_{\al_{\Ga_0}'^c},&\hbox{if}~\Ga_0\!\in\!L(\al').
\end{cases}\EE
It then follows that $(\fb_{\al'}')_{\al'\in\cC_{\om;\al^{\dag}}(Y)}$ 
is a bounding chain on $(\al^{\dag},J)$.
Furthermore, the number~$\lr{L^{\dag}}_{\be;K}^{\om,\os}$
as in~\eref{JSinvdfn_e2} determined 
by the bounding chain $(\fb_{\al'}')_{\al'\in\cC_{\om;\al^{\dag}}(Y)}$
is the number~$\lr{L}_{\be;K}^{\om,\os}$ in~\eref{JSinvdfn_e2} determined 
by the bounding chain $(\fb_{\al'})_{\al'\in\cC_{\om;\al}(Y)}$
times~$\lr{\ga_0,\be}$.

Let $\al'\!\in\!\cC_{\om;\al^{\dag}}(Y)$.
The first case in~\eref{div_e3} follows immediately from
the first case in~\eref{div_e1}.
We thus assume that $\Ga_0\!\in\!L(\al')$.
Let $\eta'\!\in\!\cD_{\om}(\al')$, 
$\fbb_{\eta'}'$ be the restriction of~$\fbb_{\al'}'$ to 
the subspace of its domain corresponding to~$\eta'$, and 
\BE{div_e4}\eta_{\Ga_0}'^c=
\big(\be_{\bu}(\eta'),k_{\bu}(\eta'),L_{\bu}(\eta')\!-\!\{\Ga_0\},
(\al_i(\eta')_{\Ga_0}^c)_{i\in[k_{\bu}(\eta')]}\big)\in\cD_{\om}(\al).\EE
By the assumption that $\Ga_0$ is disjoint from~$Y$ and~\eref{div_e1},
\BE{div_e5}\fbb_{\eta'}'=\begin{cases}\eset,&\hbox{if}~\be_{\bu}(\eta')\!=\!0,
\,\Ga_0\!\in\!L_{\bu}(\eta');\\
\lr{\ga_0,\be_i(\eta_{\Ga_0}'^c)}\fbb_{\eta_{\Ga_0}'^c},
&\hbox{if}~i\!\in\![k_{\bu}(\eta_{\Ga_0}'^c)],\,\Ga_0\!\in\!L_i(\eta'). 
\end{cases}\EE
Suppose $\be_{\bu}(\eta')\!\neq\!0$ and $\Ga_0\!\in\!L_{\bu}(\eta')$.
The forgetful morphism dropping the first interior marked point induces a map
\begin{equation*}\begin{split}
&\ff\!:\M^+_{\eta';J}\!\fiber\!\!
\Big(\!(i\!+\!1,\fb_{\al_{i}(\eta')})_{i\in[k_{\bu}(\eta')]};
(i\!+\!1,\Ga_i)_{\Ga_i\in L_{\bu}(\eta')}\!\Big)\\
&\hspace{1.5in}
\lra \M^+_{\eta_{\Ga_0}'^c;J}\!\fiber\!\!
\Big(\!(i\!+\!1,\fb_{\al_i(\eta_{\Ga_0}'^c)})_{i\in[k_{\bu}(\eta_{\Ga_0}'^c)]};
(i,\Ga_i)_{\Ga_i\in L_{\bu}(\eta_{\Ga_0}'^c)}\!\Big)
\end{split}\end{equation*}
intertwining the maps~$\evb_1$.
The map~$\ff$ restricts to a covering projection of degree $\lr{\ga_0,\be_{\bu}(\eta')}$
outside of a  codimension~2 subspace of its target.
Thus,
$$\fbb_{\eta'}'=\blr{\ga_0,\be_{\bu}(\eta_{\Ga_0}'^c)}\fbb_{\eta_{\Ga_0}'^c}$$
outside of codimension~1 subspaces of the domains of the two sides. 
Combining this with~\eref{div_e5}, we~obtain
$$\bigsqcup_{\begin{subarray}{c}\eta'\in\cD_{\om}(\al')\\ \eta_{\Ga_0}'^c=\eta\end{subarray}}
\hspace{-.18in}\fbb_{\eta'}'=\blr{\ga_0,\be(\al_{\Ga_0}'^c)}\fbb_{\eta}
\qquad\forall\,\eta\!\in\!\cD_{\om}(\al_{\Ga_0}'^c)\,.$$
This establishes the second case in~\eref{div_e3}.

By the same reasoning with $\M^+_{\eta';J}$ replaced by~$\M_{\eta';J}$,
\begin{equation*}\begin{split}
\sum_{\begin{subarray}{c}\eta'\in\cD_{\om}(\al')\\ 
\eta'^c_{\Ga_0}=\eta\end{subarray}}
&\!\!\!\!\!
\big|\M_{\eta';J}\!\fiber\!\!\big(\!(i,\fb_{\al_i(\eta')})_{i\in[k_\bu(\eta')]};
(i\!+\!1,\Ga_i)_{\Ga_i\in L_\bu(\eta')}\big)\!\big|^{\pm}\\
&\hspace{.8in}=\blr{\ga_0,\be(\al_{\Ga_0}'^c)}
\big|\M_{\eta;J}\!\fiber\!\!\big(\!(i,\fb_{\al_i(\eta)})_{i\in[k_\bu(\eta)]};
(i,\Ga_i)_{\Ga_i\in L_\bu(\eta)}\big)\!\big|^{\pm}
\end{split}\end{equation*}
for all $\eta\!\in\!\cD_{\om}(\al'^c_{\Ga_0})$.
Thus, the number $\lr{L^{\dag}}_{\be;K}^*$ as in~\eref{JSinvdfn_e2b} 
determined  by the bounding chain $(\fb_{\al'}')_{\al'\in\cC_{\om;\al^{\dag}}(Y)}$
is the number $\lr{L}_{\be;K}^*$ in~\eref{JSinvdfn_e2b} determined 
by the bounding chain $(\fb_{\al'})_{\al'\in\cC_{\om;\al}(Y)}$
times~$\lr{\ga_0,\be}$.
In light of~\eref{equivdfn_e0c}, 
the same statement holds for the analogous numbers~\eref{JSinvdfn_e2c}.
This in particular establishes the divisor property of 
the open invariants~\eref{geomJSinvdfn_e}
and~\eref{OGWdfn_e}
under the independence assumptions of Theorem~\ref{main_thm}.
\end{proof}

\begin{proof}[{\bf{\emph{Proof of~\ref{sphere_it}}}}]
Let $p_0\!\in\!Y$ be a generic point. 
With $\Ga_0\!:S(\cN_{p_0}Y)\!\lra\!X\!-\!Y$ denoting the inclusion of 
a small generic sphere in the fiber~$\cN_{p_0}Y$ of the normal bundle
of~$Y$ in~$X$, we define $\al^{\dag}\!\in\!\cC_{\om}(Y)$
by~\eref{div_e0}.
We show~that 
\BE{sphere_e1}
\blr{L}_{\be;K}^{\Ga_0}=\blr{L}_{\be;K}^{\om,\os}\EE
if both counts above arise from the bounding chain 
$(\fb_{\al'})_{\al'\in\cC_{\om;\al}(Y)}$ on~$(\al,J)$.
This in particular establishes the identity in~\ref{sphere_it}
under the independence assumptions of Theorem~\ref{main_thm}.

We can assume that \eref{dimcond_e} with $k$ replaced by $k\!+\!1$ holds.
This implies~that
\BE{sphere_e3}
\dim\,\M_{\eta;J}\!\fiber\!\!\big(\!(i,\fb_{\al_i(\eta)})_{i\in[k_\bu(\eta)]};
(i,\Ga_i)_{\Ga_i\in L_\bu(\eta)}\big)=n\!-\!1\in2\Z
\quad\forall\,\eta\!\in\!\cD_{\om}(\al).\EE
For $\eta\!\in\!\cD_{\om}(\al)$, let 
\hbox{$[\eta]\!\subset\!\cD_{\om}(\al)$} be the orbit of~$\eta$
under the action of the rotation as in~\eref{rhodfn_e},
$$\evb_{\eta}\!\!\equiv\!\evb_1\!:
\M_{\eta;J}^+\!\fiber\!\!\big(\!(i\!+\!1,\fb_{\al_i(\eta)})_{i\in[k_\bu(\eta)]};
(i,\Ga_i)_{\Ga_i\in L_\bu(\eta)}\big)\lra Y, \quad
\evb_{[\eta]}\equiv\bigsqcup_{\eta''\in[\eta]}\!\!\!\evb_{\eta''}\,.$$
If $p_0\!\in\!Y$ is generic, $\evb_{[\eta]}^{-1}(p_0)$ is a finite set of signed points;
the sign of a point~$\wt\u$ is plus if~$\nd_{\wt\u}\evb_{[\eta]}$ is orientation-preserving. 
For $\eta'\!\in\!\cD_{\om}^{\Ga_0}(\al^{\dag})$, let 
$\eta'^c_{\Ga_0}\!\in\!\cD_{\om}(\al)$ be as in~\eref{div_e4} and
$$\evi_{\eta'}\!=\!\evi_1\!: 
\M_{\eta';J}\!\fiber\!\!\big(\!(i,\fb_{\al_i(\eta')})_{i\in[k_\bu(\eta')]};
(i\!+\!1,\Ga_i)_{\Ga_i\in L_\bu(\eta')-\{\Ga_0\}}\big)\lra X.$$
If $p_0\!\in\!Y$ is generic, $\evi_{\eta'}^{-1}(S(\cN_{p_0}Y)\!)$ is a finite set of signed points;
the sign of a point~$\wt\u$ is plus if the composition of $\nd_{\wt\u}\evi_{\eta'}$
with the projection to the normal bundle of~$S(\cN_{p_0}Y)$ in~$X$
is orientation-preserving. 
Since the dimensions of~$X$ and~$\Ga_i$ are even,
\BE{sphere_e7}
\M_{\eta';J}\!\fiber\!\!\big(\!(i,\fb_{\al_i(\eta')})_{i\in[k_\bu(\eta')]};
(i\!+\!1,\Ga_i)_{\Ga_i\in L_\bu(\eta')}\big)
\approx \evi_{\eta'}^{-1}\big(S(\cN_{p_0}Y)\!\big)\,.\EE
If $\be_{\bu}(\eta')\!=\!0$, $\evi_{\eta'}^{-1}(S(\cN_{p_0}Y)\!)\!=\!\eset$
because $\Ga_0$ is disjoint from~$Y$.

Suppose $\be_{\bu}(\eta')\!\neq\!0$. 
The short exact sequence 
\begin{equation*}\begin{split}
0\lra T_{z_1}\D \lra T_{\wt{\u}}
\big(\M_{\eta';J}\!\fiber\!\!\big(\!(i,\fb_{\al_i(\eta')})_{i\in[k_\bu(\eta')]};
(i\!+\!1,\Ga_i)_{\Ga_i\in L_\bu(\eta')-\{\Ga_0\}}\big)\!\big)&\\ 
\lra T_{\u}\big(\M_{\eta'^c_{\Ga_0};J}\!\fiber\!\!
\big(\!(i,\fb_{\al_i(\eta'^c_{\Ga_0})})_{i\in[k_\bu(\eta'^c_{\Ga_0})]};
(i,\Ga_i)_{\Ga_i\in L_\bu(\eta'^c_{\Ga_0})}\big)\!\big)&\lra0
\end{split}\end{equation*}
of vector spaces induced by the forgetful morphism dropping the first interior marked
point is orientation-preserving
for every element $[\wt\u]\!=\![\u,z_1]$ of the first fiber product above.
The short exact sequence 
\begin{equation*}\begin{split}
0\lra T_{x_1}S^1 \lra T_{\wt{\u}}
\big(\M_{\eta;J}^+\!\fiber\!\!\big(\!(i\!+\!1,\fb_{\al_i(\eta)})_{i\in[k_\bu(\eta)]};
(i,\Ga_i)_{\Ga_i\in L_\bu(\eta)}\big)\!\big)&\\ 
\lra T_{\u}\big(\M_{\eta;J}\!\fiber\!\!\big(\!(i,\fb_{\al_i(\eta)})_{i\in[k_\bu(\eta)]};
(i,\Ga_i)_{\Ga_i\in L_\bu(\eta)}\big)\!\big)&\lra0
\end{split}\end{equation*}
of vector spaces induced by the forgetful morphism dropping the first boundary marked
point is also orientation-preserving
for every $\eta\!\in\![\eta'^c_{\Ga_0}]$ and
every element $[\wt\u]\!=\![\u,x_1]$ of the first fiber product above.
Along with~\eref{sphere_e3}, the last two statements imply that the homotopy classes
of isomorphisms
\begin{equation*}\begin{split}
&T_{\wt{\u}}\big(\M_{\eta';J}\!\fiber\!\!\big(\!(i,\fb_{\al_i(\eta')})_{i\in[k_\bu(\eta')]};
(i\!+\!1,\Ga_i)_{\Ga_i\in L_\bu(\eta')-\{\Ga_0\}}\big)\!\big)\\
&\hspace{1in}\approx 
T_{\u}\big(\M_{\eta'^c_{\Ga_0};J}\!\fiber\!\!
\big(\!(i,\fb_{\al_i(\eta'^c_{\Ga_0})})_{i\in[k_\bu(\eta'^c_{\Ga_0})]};
(i,\Ga_i)_{\Ga_i\in L_\bu(\eta'^c_{\Ga_0})}\big)\!\big)
\!\oplus\!T_{z_1}\D,\\
&T_{\wt{\u}}
\big(\M_{\eta;J}^+\!\fiber\!\!\big(\!(i\!+\!1,\fb_{\al_i(\eta)})_{i\in[k_\bu(\eta)]};
(i,\Ga_i)_{\Ga_i\in L_\bu(\eta)}\big)\!\big)\\
&\hspace{1in}\approx 
T_{\u}\big(\M_{\eta;J}\!\fiber\!\!\big(\!(i,\fb_{\al_i(\eta)})_{i\in[k_\bu(\eta)]};
(i,\Ga_i)_{\Ga_i\in L_\bu(\eta)}\big)\!\big)\!\oplus\!T_{x_1}S^1
\end{split}\end{equation*}
are orientation-preserving.
Combining this with the proof of the equality of the right-hand sides 
of the two equations in \cite[(6.20)]{RealWDVV3}, we obtain
\BE{sphere_e11}
\big|\evi_{\eta'}^{-1}\big(S(\cN_{p_0}Y)\!\big)\big|^{\pm}
= \big|\evb_{[\eta'^c_{\Ga_0}]}^{-1}(p_0)\big|^{\pm}\,.\EE

\vspace{-.1in}

By~\eref{JSinvdfn_e2c}, \eref{sphere_e7}, and \eref{sphere_e11}, 
$$\blr{L}_{\be;K}^{\Ga_0}
=\sum_{\begin{subarray}{c}\eta'\in\cD_{\om}^{\Ga_0}(\al^{\dag})\\
\be_{\bu}(\eta')\neq0\end{subarray}}\hspace{-0.2in}
(-1)^{k_{\bu}(\eta')}s^{\circ}(\eta')\big|\evi_{\eta'}^{-1}\big(S(\cN_{p_0}Y)\!\big)\big|^{\pm}
=\sum_{\begin{subarray}{c}\eta\in\cD_{\om}(\al)\\ \be_{\bu}(\eta)\neq0\end{subarray}}
\hspace{-0.15in}(-1)^{k_{\bu}(\eta)}\big|\evb_{\eta}^{-1}(p_0)\big|^{\pm}.$$
If $\be_{\bu}(\eta)\!=\!0$, $\evb_{\eta}^{-1}(p_0)\!=\!\eset$ by the reasoning
in the proof of~\ref{deg0_it}.
Thus, the last sum above equals to the right-hand side of~\eref{sphere_e1}.
\end{proof}

\subsection{Properties~\ref{lagl_it1} and~\ref{lagl_it2}}
\label{LangOGW_subs}

Throughout this section, we assume that 
$$\mu_Y^{\om}(\be)\!+\!(n\!-\!3)=(n\!-\!1)k\!+\!
\sum_{i=1}^l\big(\deg\,\ga_i\!-\!2) ;$$
otherwise, all invariants in~\ref{lagl_it1} and~\ref{lagl_it2} vanish
for dimensional reasons.
Both of these properties then follow immediately from Proposition~\ref{LangOGW_prp} below.
We state this proposition, which is analogous to Lemma~4.3 in~\cite{JS3},  
in a slightly greater generality than necessary for
the purposes of establishing~\ref{lagl_it1} and~\ref{lagl_it2} to make it readily
usable for a geometric translation of the proof of WDVV-type relations of~\cite{JS3}.
With $(X,\om)$ and $Y$ as in Theorem~\ref{main_thm}, let 
$\io_Y\!:Y\!\lra\!X$ be the inclusion.

\begin{prp}\label{LangOGW_prp}
Let $(X,\om,Y)$, $n$, $\os$, $\al\!\equiv\!(\be,K,L)$, $J$, and 
$(\fb_{\al'})_{\al'\in\cC_{\om;\al}(Y)}$ be as in Definition~\ref{bndch_dfn}.
If $\Ga_0$ is a generic dimension~$n$ pseudocycle to~$X$, possibly with boundary,
and \hbox{$\dim(\al)\!=\!n\!-\!1$}, then
\BE{LangOGW_e}\begin{split}
&\big|\io_Y\!\!\fiber\!\Ga_0\big|^{\pm}\blr{L}^{\om,\os}_{\be;K}
+\blr{L(\al)}^{\prt\Ga_0}_{\be;K}\\
&\hspace{.2in}
=\sum_{B\in q_Y^{-1}(\be)}\!\!\!\!\!\!
(-1)^{\lr{w_2(\os),B}}\big|\M^\C_{\{0,1\}\sqcup L}(B;J)\!\!\fiber\!\!
\big(\!(0,Y),(1,\Ga_0),(i\!+\!1,\Ga_i)_{\Ga_i\in L}\big)\big|^{\pm}\,.
\end{split}\EE  
\end{prp}

\begin{proof}[{\bf{\emph{Proof of~\ref{lagl_it1}}}}]
If $\Ga_0$ is a closed pseudocycle to~$X$, 
$$\big|\io_Y\!\!\fiber\!\Ga_0\big|^{\pm}=\lr{\PD_X([\Ga_0]_X),[Y]_X}\in R.$$
The signed cardinality on the right-hand side of~\eref{LangOGW_e}
is then the degree~$B$ closed GW-invariant
$$\blr{\PD_X([Y]_X),\PD_X([\Ga_0]_X),
\big(\PD_X([\Ga_i]_X)\!\big)_{\Ga_i\in L}}_{\!B}^{\!\om}\in R$$
of $(X,\om)$.
The claim of~\ref{lagl_it1} is thus the $K\!=\!\eset$ case of
Proposition~\ref{LangOGW_prp} with $\Ga_0,\Ga_1,\ldots,\Ga_l$ being
generic pseudocycle representatives for the Poincare duals of $\ga_0,\ga_1,\ldots,\ga_l$.
 \end{proof}

\begin{proof}[{\bf{\emph{Proof of~\ref{lagl_it2}}}}]
The right-hand side of~\eref{LangOGW_e} vanishes if $K\!\neq\!\eset$ for dimensional reasons.
The claim of~\ref{lagl_it2} thus follows from Proposition~\ref{LangOGW_prp} with
$\Ga_0,\Ga_1,\ldots,\Ga_l$ being generic pseudocycles to~$X$ such that 
$|\io_Y\!\!\fiber\!\Ga_0|^{\pm}\!\neq\!0$ and 
the Poincare duals of $\Ga_1,\ldots,\Ga_l$ are $\ga_1,\ldots,\ga_l$.
\end{proof}

The remainder of this section establishes Proposition~\ref{LangOGW_prp}.
Let $\al$ and $\Ga_0$ be as in its statement 
and~$L^{\dag}$ and $\al^{\dag}$ be as in~\eref{div_e0}.
We~define 
\BEnum{$\bu$}

\item  $\cD_{\om}^{\Ga_0}(\al^{\dag})$ by~\eref{cDomdfn_e} 
with~$\Ga_0\!\in\!L_{\bu}\!\subset\!L^{\dag}$,

%\item $\cD^{2;\Ga_0}_\om(\al^{\dag})$ by~\eref{cDonm2dfn_e} with $\cD_{\om}(\al)$ replaced
%by $\cD_{\om}^{\Ga_0}(\al^{\dag})$,

\item $\cD_{\om}^{2'}(\al^{\dag})$ by~\eref{cDonm2dfn_e}
with $\cD_{\om}(\al)$ replaced by $\cD_{\om}^{\Ga_0}(\al^{\dag})$,
$[k_{\bu}(\eta)\!+\!2]$ replaced by $\{0\}\!\sqcup\![k_{\bu}(\eta)\!+\!1]$,
and
$$\big(\be_{\bu}^2,k_2\!-\!1\!-\!k_1,L_{\bu}^2\big)\neq
(0,0,\eset),(0,1,\eset),
\big(\be_{\bu}(\eta),k_{\bu}(\eta)\!-\!1,L_{\bu}(\eta)\!\big),
\big(\be_{\bu}(\eta),k_{\bu}(\eta),L_{\bu}(\eta)\!\big),$$

\item  $\ov\cD_{\om}^{2*}(\al^{\dag})$ 
by~\eref{BCpseudo_e11} with $\cD_{\om}(\al)$ replaced by~$\cD_{\om}^{\Ga_0}(\al^{\dag})$
and $(\be_{\bu}(\eta),k_{\bu}(\eta),L_{\bu}(\eta)\!)\!\neq\!(0,2,\eset)$.

\EEnum
Let 
$$\cD_{\om}^{2*}(\al^{\dag})=\big\{\!\big(\eta;\be_{\bu}^2,k_1,k_2,L_{\bu}^2\big)
\!\in\!\cD_{\om}^{2'}(\al^{\dag})\!: k_1\!\ge\!1~\hbox{or}~k_{\bu}(\eta)\!=\!0\big\}.$$
The construction above~\eref{BCpseudo_e11} determines a bijection
\BE{cDom2prbi_e}
\big\{(\eta;\be_{\bu}^2,k_1,k_2,L_{\bu}^2)\!\in\!\cD_{\om}^{2'}(\al^{\dag})\!:
\Ga_0\!\not\in\!L_{\bu}^2\big\}\lra\ov\cD_\om^{2*}(\al^{\dag}),
~
\big(\eta;\be_{\bu}^2,k_1,k_2,L_{\bu}^2\big)\lra\big(\eta^1;k_1\!+\!1,\eta^2\big)\,.\EE
We define ``rotations"
\BE{rhocD4_e} \rho\!: \ov\cD_{\om}^{2*}(\al^{\dag})\lra \ov\cD_{\om}^{2*}(\al^{\dag})
\qquad\hbox{and}\qquad 
\rho\!: \cD_{\om}^{2*}(\al^{\dag})\lra\cD_{\om}^{2*}(\al^{\dag})\EE
by~\eref{rhoovcDdfn_e} and~\eref{rhocD2_e}.
Let
\begin{gather}\label{auxref_e}
R\!:\big\{(\eta;\be^2_\bu,k_1,k_2,L^2_\bu)\!\in\!\cD_{\om}^{2'}(\al^{\dag})\!:k_1\!=\!0\big\}
\lra \cD_{\om}^{2'}(\al^{\dag}),\\
\notag
R\big(\eta;\be^2_\bu,0,k_2,L^2_\bu\big)
=\big(\eta;\be_\bu(\eta)\!-\!\be^2_\bu,k_2\!-\!1,k_\bu(\eta)\!+\!1,L_\bu(\eta)\!-\!L^2_\bu\big).
\end{gather}
Below~\eref{bdry_e}, we use the equivalence relation on $\cD^{2'}_\om(\al)$ 
generated by the maps~$\rho$ in~\eref{rhocD4_e} and~$R$ in~\eref{auxref_e}.

Let $(\fb_{\al'})_{\al'\in\cC_{\om;\al}(Y)}$ be as in the statement
of Proposition~\ref{LangOGW_prp}.
For $\eta\!\in\!\cD_{\om}^{\Ga_0}(\al^{\dag})$, 
we define $\M_{\eta;J}$ and $\M_{\eta;J}^+$ as in~\eref{fMetaJdfn_e}
and $\fbb_{\eta}$ as in~\eref{fbbetadfn_e0}.
For $(\eta;\vt)\!\in\!\cD^{2'}_\om(\al^{\dag})$, define
$\cS_\eta(\vt)\!\subset\!\M_{\eta;J}$ as in~\eref{BCpseudo_e12}.
All boundary strata of~$\M_{\eta;J}$ are of this form;
two elements of~$\cD^{2'}_\om(\al^{\dag})$ describe the same boundary stratum 
if and only if one can be mapped to the other by compositions of 
the map~$R$ in~\eref{auxref_e}.
For $(\eta;\vt)\!\in\!\cD^{2'}_\om(\al^{\dag})$
and $(\eta;i,\eta')\!\in\!\ov\cD^{2*}_\om(\al^{\dag})$, let
\begin{equation*}\begin{split}
\fbb^*_\eta(\vt)&\equiv \cS_\eta(\vt)\!\fiber\!
\big(\!(i,\fb_{\al_i(\eta)})_{i\in[k_\bu(\eta)]};
(i,\Ga_i)_{\Ga_i\in L_\bu(\eta)}\big),\\
\fbb^*_\eta(i,\eta')&\equiv \M_{\eta;J}\!\fiber\!
\big(\!(j,\fb_{\al_j(\eta)})_{j\in[k_\bu(\eta)]-\{i\}},(i,\fbb_{\eta'});
(j,\Ga_j)_{\Ga_j\in L_\bu(\eta)}\big)\,.
\end{split}\end{equation*}
For $(\eta;\vt)\!\in\!\cD^{2*}_\om(\al^{\dag})$ and 
$(\eta;i,\eta')\!\in\!\ov\cD^{2*}_\om(\al^{\dag})$,
the ``backwards" cyclic permutations of the boundary marked points of the moduli space
and of the pseudocycles to~$Y$ induce diffeomorphisms
$$\rho_{\eta;\vt}\!:\fbb^*_{\eta}(\vt)\lra\fbb^*_{\rho(\eta)}(\vt'), \qquad
\rho_{\eta;i,\eta'}\!:\fbb^*_{\eta}(i,\eta')\lra
\begin{cases}\fbb^*_{\rho(\eta)}(i\!-\!1,\eta'),&\hbox{if}~i\!>\!1;\\
\fbb^*_{\rho(\eta)}(k_{\bu}(\eta),\eta'),&\hbox{if}~i\!=\!1.
\end{cases}$$

\begin{lmm}\label{rhoRsgn_lmm2}
The diffeomorphisms $\rho_{\eta;\th}$ with $(\eta;\vt)\!\in\!\cD_{\om}^{2*}(\al^{\dag})$
and $\rho_{\eta;i,\eta'}$ with  \hbox{$(\eta;i,\eta')\!\in\!\ov\cD_{\om}^{2*}(\al^{\dag})$}
above are orientation-preserving.
\end{lmm}

\begin{proof}
The proof is similar to that of Lemma~\ref{rhoRsgn_lmm}.
The cyclic permutations of the boundary marked points of the elements of $\M_{\eta;J}$
and of the pseudocycles to~$Y$ induce a commutative diagram
$$\xymatrix{\M_{\eta;J}\ar[rr]^>>>>>>>>>>{\ev}\ar[d]_{\rho_{\M}}&&
Y^{k_{\bu}(\eta)}\!\!\times\!\!X^{L_{\bu}(\eta)} \ar[d]_{\rho_Y}&&
\prod\limits_{j\in[k_{\bu}(\eta)]}\hspace{-.15in}(\dom\,\fb_{\al_{j}(\eta)})
\!\times\!\!
\prod\limits_{\Ga_j\in L_{\bu}(\eta)}\hspace{-.14in}(\dom\,\Ga_j) \ar[ll]\ar[d]^{\rho_{\fb}}\\ 
\M_{\eta;J}\ar[rr]^>>>>>>>>>{\ev}&&
Y^{k_{\bu}(\eta)}\!\!\times\!\!
X^{L_{\bu}(\eta)} &&
\prod\limits_{j\in[k_{\bu}(\eta)]}\hspace{-.15in}(\dom\,\fb_{\al_{j}(\eta)})
\!\times\!\!
\prod\limits_{\Ga_j\in L_{\bu}(\eta)}\hspace{-.14in}(\dom\,\Ga_j) \ar[ll]}$$
so that the vertical arrows are diffeomorphisms.
Since the dimensions of all $\fb_{\al_{j}(\eta)}$ are even,
the diffeomorphism~$\rho_{\fb}$ is orientation-preserving.
By the construction of the orientation on~$\M_{\eta;J}$ in Section~\ref{Ms_subs},
the sign of the diffeomorphism~$\rho_{\M}$ is~$(-1)^{k_{\bu}(\eta)-1}$.
Since the dimension of~$Y$ is odd, this is also the sign of the diffeomorphism~$\rho_Y$.
The first claim of the lemma now follows from Lemma~\ref{fibprodflip_lmm}.
The claim concerning~$\rho_{\eta;i,\eta'}$ is obtained in the same way
with~$\fb_{\al_i(\eta)}$ replaced by~$\fbb_{\al_i(\eta)}$.
\end{proof}

\begin{lmm}\label{bbetavt_lmm2}
If $(\eta^1;i,\eta^2)\!\in\!\ov\cD_{\om}^{2*}(\al^{\dag})$ is the image of 
\hbox{$(\eta;\vt)\!\in\!\ov\cD_{\om}^{2'}(\al^{\dag})$} 
under~\eref{cDom2prbi_e}, then
\BE{bbcompare_e}
(-1)^{k_\bu(\eta)}\fbb^*_\eta(\vt)
=-(-1)^{k_\bu(\eta^1)}\fbb^*_{\eta^1}(i,\eta^2). \EE
\end{lmm}

\begin{proof}
Let $\vt\!=\!(\be^2_\bu,k_1,k_2,L^2_\bu)$.
By the definition of the map~\eref{cDom2prbi_e}, $\Ga_0\!\not\in\!L^2_{\bu}$.
The proof of the $B\!=\!\{\pt\}$, $I\!=\![k_{\bu}(\eta)]$, and $J\!=\!L_{\bu}(\eta)$ 
case of Corollary~\ref{Scompare_crl} applies with $\ep_r$ for $i\!=\!1,2,3$
replaced by $\ep_r\!+\!\ep_r'$, where
$$\big(\ep_1',\ep_2',\ep_3'\big)=
\big(k_2\!-\!k_1\!+\!1,k_2\!-\!k_1\!-\!1,1\big).$$
Thus, Corollary~\ref{Scompare_crl} yields the same conclusion with $\ep\!=\!k_1\!+\!k_2\!+\!1$
under the above assumptions.
Since
$${k_\bu}(\eta^1)=k_{\bu}(\eta)\!-\!(k_2\!-\!k_1\!-\!1)\!+\!1,$$
this establishes the claim.
\end{proof}

\begin{proof}[{\bf{\emph{Proof of~Proposition~\ref{LangOGW_prp}}}}]
For $\eta\!\in\!\cD_\om^{\Ga_0}(\al^{\dag})$, let $s^\circ(\eta)$ be as in 
Section~\ref{OpenGWs_subs} and 
$$s'(\eta)=\begin{cases}s^\circ(\eta),&\hbox{if}~k_\bu(\eta)\!\neq\!0,\\
\frac{1}{2}s^\circ(\eta),&\hbox{if}~k_\bu(\eta)\!=\!0.\end{cases}$$
We denote~by $\cS^0_\eta\!\subset\!\prt\ov\M_{\eta;J}$ 
the \sf{sphere-bubbling stratum} with the boundary orientation;
it consists of the maps from $(\D^2,S^1)$ to $(X,Y)$ with $S^1$ contracted to
a point in~$Y$. 
This stratum is empty unless
$$ k_{\bu}(\eta)=0 \qquad\hbox{and}\qquad 
\be_{\bu}(\eta)\in\Im\big(q_Y\!:H_2(X;\Z)\!\lra\!H_2(X,Y;\Z)\!\big).$$
%Suppose $J,p_1,\ldots,p_k,\Ga_1,\ldots,\Ga_l$ is as in Definition~\ref{bndch_dfn}, 
%but with $\dim\,\Ga_1$ odd (and $\dim\,\Ga_i$ even for all $i\!\neq\!1$ as before).
By Lemma~\ref{fibersign_lmm} and Definition~\ref{bndch_dfn},
\BE{bdry_e}\begin{split}
&-\prt\bigg(\bigsqcup_{\eta\in\cD_\om^{\Ga_0}(\al^{\dag})}\hspace{-.2in}
s^\circ(\eta)\M_{\eta;J}\!\fiber\!\!\big(\!(i,\fb_{\al_i(\eta)})_{i\in[k_\bu(\eta)]};
(i\!+\!1,\Ga_i)_{\Ga_i\in L_\bu(\eta)}\big)\!\!\bigg)\\
&\quad=\!\!\!\bigsqcup_{\eta\in\cD_\om^{\Ga_0}(\al^{\dag})}\hspace{-.2in}
\cS^0_\eta\!\fiber\!\!\big(;(i\!+\!1,\Ga_i)_{\Ga_i\in L(\al^{\dag})}\big)\\
&\hspace{.4in}
\sqcup \!\!\!\!\!\!\!\!
\unbr{\bigsqcup_{(\eta;\vt)\in\cD^{2*}_\om(\al^{\dag})}\hspace{-.3in}
(-1)^{k_\bu(\eta)}\!s'\!(\eta)\fbb^*_\eta(\vt)}{\bf II}
\sqcup \!\!\!\!\!\!
\unbr{\bigsqcup_{\begin{subarray}{c}(\eta;i,\eta')\in\ov\cD^{2*}_\om(\al^{\dag})\\
\dim\,\fbb_{\eta'}<n\end{subarray}}\hspace{-.25in}
(-1)^{k_\bu(\eta)}\!s^{\circ}(\eta)\fbb^*_{\eta}(i,\eta')}{\bf III}\\
&\hspace{.4in}\sqcup\!\!\!\!\!
\bigsqcup_{\eta\in\cD_\om^{\Ga_0}(\al^{\dag})}\hspace{-.18in}
(-1)^{k_\bu(\eta)+1}s^\circ(\eta)
\M_{\eta;J}\!\fiber\!\!\big(\!(i,\fb_{\al_i(\eta)})_{i\in[k_\bu(\eta)]};
(1,\prt\Ga_0),(i\!+\!1,\Ga_i)_{\Ga_i\in L_\bu(\eta)-\{\Ga_0\}}\big).
\end{split}\EE

For $\vt\!=\!(\be^2_\bu,k_1,k_2,L^2_\bu)$, we define
$k_2(\vt)\!\equiv\!k_2$ and $L^2_\bu(\vt)\!=\!L^2_{\bu}$.
The orbit of $(\eta;\vt)\!\in\!\cD^{2*}_\om(\al^{\dag})$ under the action of
the second map~$\rho$ in~\eref{rhocD4_e} consists of $1/s'(\eta)$ elements
and is contained in the equivalence class of~$(\eta;\vt)$ in~$\cD^{2'}_\om(\al^{\dag})$.
This equivalence class 
has a unique element $(\eta^*;\vt^*)$ with 
$\Ga_0\!\not\in\!L^2_\bu(\vt^*)$ and $k_2(\vt^*)\!=\!k_\bu(\eta^*)\!+\!1$. 
The orbit of $(\eta;i,\eta')\!\in\!\ov\cD^{2*}_\om(\al^{\dag})$
under the action of the first map~$\rho$ in~\eref{rhocD4_e}
consists of $1/s^{\circ}(\eta)$ elements. 
Each such orbit has a unique element of the form $(\eta;k_\bu(\eta),\eta')$. 
Along with Lemma~\ref{rhoRsgn_lmm2}, this implies that 
$${\bf II}=\bigsqcup_{\begin{subarray}{c}(\eta;\vt)\in\cD^{2*}_\om(\al^{\dag})\\
\Ga_0\not\in L^2_\bu(\vt)\\ k_2(\vt)=k_\bu(\eta)+1\end{subarray}}\hspace{-.3in}
(-1)^{k_\bu(\eta)}\fbb^*_\eta(\vt)\qquad\hbox{and}\qquad
{\bf III}=
\bigsqcup_{\begin{subarray}{c}(\eta;k_\bu(\eta),\eta')\in\ov\cD^{2*}_\om(\al^{\dag})\\
\dim\,\fbb_{\eta'}<n\end{subarray}}\hspace{-.4in}
(-1)^{k_\bu(\eta)}\fbb^*_{\eta}(k_\bu(\eta),\eta')\,.$$
Since~\eref{cDom2prbi_e} restricts to a bijection 
$$\big\{(\eta;\vt)\!\in\!\cD^{2*}_\om(\al^{\dag})\!:
\Ga_0\!\not\in\!L^2_\bu(\vt),~k_2(\vt)\!=\!k_\bu(\eta)\!+\!1\big\}
\lra
\big\{(\eta;i,\eta')\!\in\!\ov\cD^{2*}_\om(\al^{\dag})\!:i\!=\!k_\bu(\eta)\big\}\,,$$
Lemma~\ref{bbetavt_lmm2} thus gives
\BE{bdry_e3}
{\bf II}+{\bf III}=
-\!\!\!\bigsqcup_{\begin{subarray}{c}(\eta;k_\bu(\eta),\eta')\in\ov\cD^{2*}_\om(\al^{\dag})\\
\dim\,\fbb_{\eta'}\ge n\end{subarray}}\hspace{-.42in}
(-1)^{k_\bu(\eta)}\fbb^*_{\eta}(k_\bu(\eta),\eta')\,.\EE

Since $\dim(\al^{\dag})\!=\!1$ and the dimension of each 
$\fbb^*_{\eta}(k_\bu(\eta),\eta')$ in~\eref{bdry_e3} is~0, 
\begin{equation*}\begin{split}
&\hbox{RHS of~\eref{bdry_e3}}
=\M_{1,1}^{\st}(0;J)\!\fiber\!\!\big(\!(1,\fbb_{\al});(1,\Ga_0)\!\big)
%&\hspace{1.2in}
=-\big(\M_{1,1}^{\st}(0;J)\!\fiber\!\!(;(1,\Ga_0)\!)\!\big){}_{\evb_1}\!
\!\times_{\fbb_{\al}}\!\!\big(\dom\,\fbb_{\al}\big);
\end{split}\end{equation*}
the first equality above holds by the reasoning below~\eref{etadimsplit_e}
with the two disk components interchanged,
while the second by Lemmas~\ref{fibprodflip_lmm} and~\ref{fibprodisom_lmm1}.
Since the isomorphism in~\eref{deg0M_e} is orientation-preserving, 
$$\big|\M_{1,1}^{\st}(0;J)\!\fiber\!\!(;(1,\Ga_0)\!)\big|^{\pm}
=\big|\io_Y\!\!\fiber\!\Ga_0\big|^{\pm}.$$
Thus,
$$\hbox{RHS of~\eref{bdry_e3}}
=-\big|\io_Y\!\!\fiber\!\Ga_0\big|^{\pm}\big(\!\deg \fbb_{\al}\big).$$
Along with~\eref{bdry_e3} and~\eref{JSinvdfn_e2}, this gives
$${\bf II}+{\bf III}
=-\big|\io_Y\!\!\fiber\!\Ga_0\big|^{\pm}\blr{L}^{\om,\os}_{\be;K}\,.$$
Combining this statement with~\eref{bdry_e}, \eref{JSinvdfn_e2}, 
and Corollary~\ref{Scompare2_crl}, 
we obtain~\eref{LangOGW_e}.
\end{proof}

\section{Orientations}
\label{Ori_sec}

Section~\ref{Fp_subs} specifies our orientation conventions for fiber products and
establishes their properties that are used throughout the paper.
We describe the relevant moduli spaces of stable disk maps 
and specify their orientations in Section~\ref{Ms_subs}.
Sections~\ref{diskbubb_subs} and~\ref{sphbubb_subs} compare the induced
orientations on the two types of boundary strata of
these moduli spaces with natural intrinsic orientations of these spaces.

\subsection{Fiber products}
\label{Fp_subs}

We say a short exact sequence of oriented vector spaces
$$0\lra V'\lra V\lra V''\lra0$$
is \sf{orientation-compatible} if, for an oriented basis $(v'_1,\ldots,v'_m)$ of $V'$, an oriented basis $(v''_1,\ldots,v''_n)$ of $V''$, and a splitting $j:V''\to V$, $(v'_1,\ldots,v'_m,j(v''_1),\ldots,j(v''_n))$ is an oriented basis of $V$. 
We say it has sign $(-1)^\ep$ if it becomes orientation-compatible 
after twisting the orientation of $V$ by $(-1)^\ep$. 
We use the analogous terminology for
short exact sequences of Fredholm operators 
with respect to orientations of their determinants;
see \cite[Section~2]{detLB}.

Let $M$ be an oriented manifold with boundary $\prt M$. We orient the normal bundle $\cN$ to $\prt M$ by the outer normal direction and orient $\prt M$ so that the short exact sequence  
$$0\lra T_p\prt M\lra T_pM\lra \cN\lra0$$
is orientation-compatible at each point $p\in\prt M$. 
We refer to this orientation of $\prt M$ as the \sf{boundary orientation}. 

We orient $M\!\times\!M$ by the usual product orientation and 
the diagonal $\De_M\subset M\!\times\!M$ by the diffeomorphism 
$$M\lra\De_M, \qquad p\lra(p,p).$$
We orient the normal bundle $\cN\De_M$ of $\De_M$ so that the short exact sequence 
\BE{cNDeor_e}
0\lra T_{(p,p)}\De_M\lra T_{(p,p)}(M\!\times\!M)\lra \cN\De_M|_{(p,p)}\lra0\EE
is orientation-compatible for each point $p\!\in\!M$. 
Thus, the isomorphism
$$\cN\De_M|_{(p,p)} \lra T_pM, \qquad [v,w]\lra w\!-\!v,$$
respects the orientations.
This in turn implies that the isomorphism
\BE{cNsplit_e}\begin{split}
\cN\De_{M_1}\big|_{(p_1,p_1)}\!\oplus\!\cN\De_{M_2}\big|_{(p_2,p_2)}
&\lra\cN\De_{M_1\times M_2}\big|_{((p_1,p_2),(p_1,p_2))}, \\
\big([v_1,w_1],[v_2,w_2]\big)&\lra\big[(v_1,v_2),(w_1,w_2)\big],
\end{split}\EE
is orientation-preserving for all oriented manifolds $M_1,M_2$
and points $p_1\!\in\!M_1$ and $p_2\!\in\!M_2$.

For maps $f\!:\!M\lra X$ and $g\!:\!\Ga\lra X$, we denote by 
$$f\!\!\fiber\!g\equiv
M_f\!\times_g\!\Ga\equiv\{(p,q)\!\in\!M\!\times\!\Ga\!:\,f(p)=g(q)\}$$
their fiber product. 
If $M,\Ga$, and $X$ are oriented manifolds ($M,\Ga$ possibly with boundary) and 
$f,f|_{\prt M}$ are transverse to $g,g|_{\prt\Ga}$, 
we orient $M_f\!\times_g\!\Ga$ so that the short exact sequence
$$0\lra T_{(p,q)}(M_f\!\times_g\!\Ga)\lra T_{(p,q)}(M\!\times\!\Ga)
\xrightarrow{[d_p\!f,d_qg]} \cN\De_X|_{(f(p),g(q))}\lra0$$
is orientation-compatible for every $(p,q)\!\in\!M_f\!\times_g\!\Ga$. 
The exact sequence 
$$0\lra T_{(p,q)}(M_f\!\times_g\!\Ga)\lra 
T_{(p,q)}(M\!\times\!\Ga)\xrightarrow{d_qg-d_pf}T_{f(p)}X\lra0$$
is then orientation-compatible as well. 
We refer to this orientation of $M_f\!\times_g\!\Ga$ as 
the \sf{fiber product orientation}. The next two observations are straightforward. 

\begin{lmm}\label{fibprodemd_lmm}
If $f\!:M\!\lra\!X$ is an open embedding, so is the projection $M_f\!\!\times_g\!\Ga\!\lra\!\Ga$.
It has sign $(-1)^{(\dim\,X)(\dim\,\Ga+1)}$ with respect to the fiber product orientation
on the left-hand side.
\end{lmm}

\begin{lmm}\label{fibersign_lmm}
If $\Ga_1,\ldots,\Ga_m$ are oriented manifolds with boundary and  
$\Ga=\Ga_1\!\times\!\ldots\!\times\!\Ga_m$, then 
\begin{equation*}\begin{split}
\prt(M_f\!\times_g\!\Ga)=(-1)^{\dim\,X}\bigg(\!
&(-1)^{\dim\,\Ga}(\prt M)_f\!\times_g\!\Ga\\ &\sqcup\,
\bigsqcup_{i=1}^m(-1)^{\sum_{j=i+1}^m\dim\,\Ga_j}M_f\!\times_g\!
\big(\Ga_1\!\times\!\ldots\!\times\!\Ga_{i-1}\!\times\!\prt\Ga_i\!\times\!\Ga_{i+1}
\!\times\!\ldots\!\times\!\Ga_m\big)\!\!\bigg).
\end{split}\end{equation*}
\end{lmm}

For a diffeomorphism $\si\!:M\!\lra\!M$ between oriented manifolds,
we define $\sgn\,\si\!=\!1$ if $\si$ is everywhere orientation-preserving 
and $\sgn\,\si\!=\!-1$ if $\si$ is everywhere orientation-reversing;
this notion is also well-defined if $M$ is orientable and $\si$ preserves each connected
component of~$M$.

\begin{lmm}\label{fibprodflip_lmm}
Suppose $M,\Ga,X$ and $f,g$ are as above Lemma~\ref{fibprodemd_lmm}
and $\si_M,\si_{\Ga},\si_X$ are diffeomorphisms of $M,\Ga,X$, respectively,
with well-defined signs.
If the diagram 
\BE{fibprodflip_e0}\begin{split}
\xymatrix{M\ar[rr]^f \ar[d]_{\si_M}&& X \ar[d]_{\si_X} && \Ga\ar[ll]_g\ar[d]^{\si_{\Ga}}\\
M\ar[rr]^f&& X  && \Ga\ar[ll]_g}
\end{split}\EE
commutes, then the sign of the diffeomorphism
$$M_f\!\!\times_g\!\Ga\lra M_f\!\!\times_g\!\Ga, \qquad
(p,q)\lra \big(\si_M(p),\si_{\Ga}(q)\!\big),$$
is $(\sgn\,\si_M)(\sgn\,\si_{\Ga})(\sgn\,\si_X)$.
\end{lmm}

\begin{proof} Let $(p,q)\!\in\!M_f\!\!\times_g\!\Ga$.
The commutative diagram~\eref{fibprodflip_e0} induces an isomorphism
$$\xymatrix{0\ar[r]& T_{(p,q)}\!
\big(M_f\!\!\times_g\!\Ga\big) \ar[r]\ar[d]&
T_{(p,q)}\!\big(M\!\times\!\Ga\big) \ar[r]\ar[d]& T_{f(p)}X\ar[r]\ar[d]& 0\\
0\ar[r]& T_{(\si_M(p),\si_{\Ga}(q)\!)}\!\big(M_f\!\times_g\!\Ga\big) \ar[r]&
T_{(\si_M(p),\si_{\Ga}(q)\!)}\!\big(M\!\times\!\Ga\big) \ar[r]&
T_{\si_X(f(p)\!)}X\ar[r]& 0}$$
of exact sequences.
The signs of the middle and right isomorphisms are $(\sgn\,\si_M)(\sgn\,\si_{\Ga})$
and $\sgn\,\si_X$, respectively.
This establishes the claim.
\end{proof}

Let $M,\Ga,X$ and $f,g$ be as above Lemma~\ref{fibprodemd_lmm}.
Suppose in addition that $e\!:M\!\lra\!Y$ and $h\!:C\!\lra\!Y$.
Let \hbox{$e'\!:M_f\!\times_g\!\Ga\!\lra\!Y$} be the map induced by~$e$;
see the top diagram in Figure~\ref{fibprodisom_fig1}.
There is then a natural bijection
\BE{fibprodisom_e1} \big(M_f\!\times_g\!\Ga\big)\!\,_{e'}\!\!\times_h\!C \approx  
M_{(f,e)}\!\!\times_{g\times h}\!(\Ga\!\times\!C)\,.\EE
If $C,Y$ are oriented manifolds and all relevant maps are transverse,
then both sides of this bijection inherit fiber product orientations.
They are compared in Lemma~\ref{fibprodisom_lmm1} below.

For any map $h\!:M\!\lra\!Z$ between manifolds, let
$$\codim\,h=\dim\,Z-\dim\,M\,.$$

\begin{lmm}\label{fibprodisom_lmm1}
The diffeomorphism~\eref{fibprodisom_e1} has sign $(-1)^{(\dim\,X)(\codim\,h)}$ with respect 
to the fiber product orientations on the two sides.
\end{lmm}

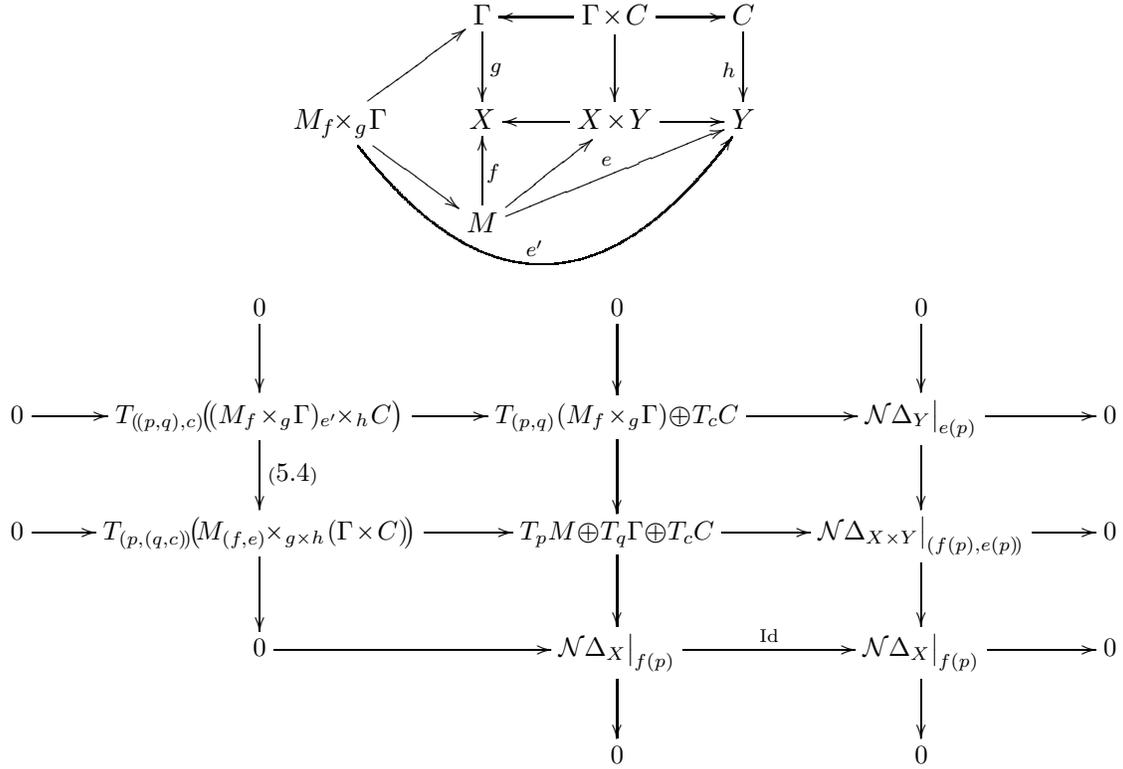
\begin{figure}
$$\xymatrix{& \Ga\ar[d]^g& \Ga\!\times\!C \ar[d]\ar[l]\ar[r]& C\ar[d]_h& \\
M_f\!\!\times_g\!\Ga \ar[ur]\ar[dr] \ar@/_4.5pc/[rrr]^{\!\!\!\!e'}& 
X&  X\!\times\!Y \ar[l]\ar[r] & Y\\
& M\ar[u]_{\!f}\ar[ur]\ar[urr]^{\!e}}$$
\begin{small}
$$\xymatrix{& 0\ar[d]& 0\ar[d]& 0\ar[d]&\\
0\ar[r]& T_{(\!(p,q),c)}\!
\big(\!(M_f\!\times_g\!\Ga)\,\!_{e'}\!\!\times_h\!C\big)\ar[d]^{\eref{fibprodisom_e1}}\ar[r]& 
T_{(p,q)}(M_f\!\times_g\!\Ga)\!\oplus\!T_cC\ar[d]\ar[r]&  
\cN\De_Y\big|_{e(p)}\ar[r]\ar[d]& 0\\
0\ar[r]& T_{(p,(q,c)\!)}\!\big(\!M_{(f,e)}\!\!\times_{g\times h}\!(\Ga\!\times\!C)\!\big)\ar[d]\ar[r]&
T_pM\!\oplus\!T_q\Ga\!\oplus\!T_cC \ar[d]\ar[r]& 
\cN\De_{X\times Y}\big|_{(f(p),e(p)\!)} \ar[r]\ar[d]& 0\\
& 0\ar[r]& 
\cN\De_X\big|_{f(p)}\ar[d]\ar[r]^{\Id}&  \cN\De_X\big|_{f(p)}\ar[d]\ar[r]& 0\\
&& 0&0&}$$
\end{small}
\caption{The maps of Lemma~\ref{fibprodisom_lmm1} and 
a commutative square of exact sequences for its proof.}
\label{fibprodisom_fig1}
\end{figure}

\begin{proof}
Suppose $(\!(p,q),c)\!\in\!(M_f\!\times_g\!\Ga)\!\,_{e'}\!\!\times_h\!C$.
We use the commutative square of exact sequences in Figure~\ref{fibprodisom_fig1}. 
The right column is induced by the isomorphism~\eref{cNsplit_e};
it is compatible with the canonical orientations on the normal bundles
if and only if $(\dim\,X)(\dim\,Y)$ is even.
The top and middle rows are orientation-compatible with respect to 
the fiber-product orientations on the left-hand and right-hand sides of~\eref{fibprodisom_e1},
respectively.
The middle column is orientation-compatible with respect to the fiber-product orientation
on $M_f\!\times_g\!\Ga$ if and only~if $(\dim\,X)(\dim\,C)$ is even.
Thus, the diffeomorphism~\eref{fibprodisom_e1} is orientation-preserving at $(\!(p,q),c)$ 
if and only~if
$$(\dim\,X)(\dim\,Y)\!+\!(\dim\,X)(\dim\,C)\in2\Z;$$
see Lemma~6.3 in~\cite{RealWDVV}.
\end{proof}

Let $M,\Ga,X$ and $f,g$ be as above with 
$$ g\!=\!g_1\!\times\!g_2\!: \Ga\!=\!\Ga_1\!\times\!\Ga_2\lra X\!=\!X_1\!\times\!X_2.$$
Suppose in addition that 
$$e_1\!:M_1\lra Y, \quad e_2\!:M_2\lra Y, \quad
f_1\!:M_1\lra X_1, \quad\hbox{and}\quad f_2\!:M_2\lra X_2$$
are maps such that
$$M=(M_1)_{e_1}\!\times_{e_2}\!\!M_2 \quad\hbox{and}\quad
f\!=\!f_1\!\times\!f_2\big|_M\,.$$
Let $e_1'\!:(M_1)_{f_1}\!\times_{g_1}\!\Ga_1\!\lra\!Y$ and 
$e_2'\!:(M_2)_{f_2}\!\times_{g_2}\!\Ga_2\!\lra\!Y$ be the maps induced
by~$e_1$ and~$e_2$, respectively;
see the top diagram in Figure~\ref{fibprodisom_fig2}.
There are natural bijections
\BE{fibprodisom_e2} M_f\!\times_g\!\Ga \approx  
\big(\!(M_1)_{f_1}\!\!\times_{g_1}\!\!\Ga_1\big)\!\,_{e_1'}\!\!\times_{e_2'}
\!\!\big(\!(M_2)_{f_2}\!\!\times_{g_2}\!\!\Ga_2\big)
\approx (M_1)_{(f_1,e_1)}\!\times_{g_1\times e_2'}\!\!
\big(\Ga_1\!\times\!\!(\!(M_2)_{f_2}\!\!\times_{g_2}\!\!\Ga_2)\!\big)\,.\EE
If $M_1,M_2,Y$ are oriented manifolds and all relevant maps are transverse,
then the middle and right spaces above inherit orientations as fiber products of 
fiber products.

\begin{figure}
\begin{small}
$$\xymatrix{& \Ga_1\ar[d]^{g_1}& \Ga_1\!\times\!\Ga_2 \ar[d]_g\ar[l]\ar[r]& \Ga_2\ar[d]_{g_2}& \\
(M_1)_{f_1}\!\times_{g_1}\!\!\Ga_1 \ar[ur]\ar[dr] \ar@/_2pc/[ddrr]^{\!\!\!\!e_1'}& 
X_1&  X_1\!\times\!X_2 \ar[l]\ar[r] & X_2 & 
(M_2)_{f_2}\!\times_{g_2}\!\!\Ga_2 \ar[ul]\ar[dl]\ar@/^2pc/[ddll]_{e_2'\!\!}\\
& M_1\ar[u]_{f_1}\ar[dr]^{\!\!e_1}& 
(M_1)_{e_1}\!\times_{e_2}\!\!M_2\ar[u]^f\ar[l]\ar[r]& M_2\ar[u]^{f_2}\ar[dl]_{e_2}& \\
&&Y}$$
$$\xymatrix{& 0\ar[d]& 0\ar[d]& 0\ar[d]&\\
0\ar[r]& T_{(p,q)}(M_f\!\times_g\!\Ga)\ar[d]\ar[r]& 
T_{(p_1,q_1)}\big(\!(M_1)_{f_1}\!\!\times_{g_1}\!\!\Ga_1\big)\!\oplus\!
T_{(p_2,q_2)}\big(\!(M_2)_{f_2}\!\!\times_{g_2}\!\!\Ga_2\big)\ar[d]\ar[r]&  
\cN\De_Y\big|_{e_1(p_1)}\ar[r]\ar[d]_{\Id}& 0\\
0\ar[r]& T_pM\!\oplus\!T_q\Ga\ar[d]\ar[r]&
T_{(p_1,q_1)}(M_1\!\times\!\Ga_1)\!\oplus\!T_{(p_2,q_2)}(M_2\!\times\!\Ga_2) \ar[d]\ar[r]& 
\cN\De_Y\big|_{e_1(p_1)} \ar[r]\ar[d]& 0\\
0\ar[r]& \cN\De_X\big|_{f(p)}\ar[d]\ar[r]& 
\cN\De_{X_1}\big|_{f_1(p_1)}\!\oplus\!\cN\De_{X_2}\big|_{f_2(p_2)} \ar[d]\ar[r]&  0\\
& 0& 0&}$$
\end{small}
\caption{The maps of Lemma~\ref{fibprodisom_lmm2} and 
a commutative square of exact sequences for its proof.}
\label{fibprodisom_fig2}
\end{figure}

\begin{lmm}\label{fibprodisom_lmm2}
The first diffeomorphism in~\eref{fibprodisom_e2} has sign $(-1)^{\ep}$ with respect 
to the fiber product orientations on the two sides, where
$$\ep=(\dim\,M_2)(\codim\,g_1)\!+\!(\dim\,X_1)(\codim\,g_2)\!+\!(\dim\,Y)(\codim\,g).$$
\end{lmm}

\begin{proof}
Suppose $(p,q)\!\in\!M_f\!\times_g\!\Ga$ with $p\!\equiv\!(p_1,p_2)\!\in\!M_1\!\times\!M_2$ 
and $q\!\equiv\!(q_1,q_2)\!\in\!\Ga_1\!\times\!\Ga_2$.
We use the commutative square of exact sequences in Figure~\ref{fibprodisom_fig2}. 
The nonzero isomorphism in the bottom is the inverse of~\eref{cNsplit_e};
it respects the canonical orientations on the normal bundles.
The left column and the top row are orientation-compatible with respect to 
the fiber-product orientations on the left and middle spaces in~\eref{fibprodisom_e2},
respectively.
The middle row is orientation-compatible with respect to the fiber-product orientation on~$M$
if and only~if
$$(\dim\,\Ga_1)(\dim\,M_2)+(\dim\,\Ga)(\dim\,Y)\in2\Z.$$
The middle column is orientation-compatible with respect to the fiber-product orientations 
on $(M_1)_{f_1}\!\!\times_{g_1}\!\!\Ga_1$ and $(M_2)_{f_2}\!\!\times_{g_2}\!\!\Ga_2$
if and only~if $(\dim\,X_1)(\dim(\!(M_2)_{f_2}\!\!\times_{g_2}\!\!\Ga_2)\!)$ is even.
Thus, the first diffeomorphism in~\eref{fibprodisom_e2} is orientation-preserving at $(p,q)$ if
and only~if
$$(\dim\,\Ga_1)(\dim\,M_2)+(\dim\,\Ga)(\dim\,Y)
+(\dim\,X_1)(\dim(M_2)_{f_2}\!\!\times_{g_2}\!\!\Ga_2)
+(\dim\,X)(\dim\,Y) \in2\Z;$$
see Lemma~6.3 in~\cite{RealWDVV}.
\end{proof}

\begin{crl}\label{fibprodisom_crl}
The composition of the two diffeomorphisms in~\eref{fibprodisom_e2} 
has sign $(-1)^{\ep}$ with respect 
to the fiber product orientations on the two sides, where
$$\ep=(\dim\,M_2)(\codim\,g_1)\!+\!(\dim\,X_1)(\codim\,e_2)\!+\!(\dim\,Y)(\codim\,g).$$
\end{crl}

\begin{proof}
By Lemma~\ref{fibprodisom_lmm1} with $M,\Ga,X,C,f,g,e,h$ replaced by 
$$M_1,\quad \Ga_1, \quad X_1, \quad (M_2)_{f_2}\!\!\times_{g_2}\!\!\Ga_2, 
\quad f_1, \quad f_2, \quad e_1, \quad e_2',$$
respectively, the second diffeomorphism in~\eref{fibprodisom_e2} has sign $(-1)^{\ep_2}$ with
$$\ep_2=(\dim\,X_1)(\codim\,e_2')=(\dim\,X_1)\big(\codim\,g_2\!+\!\codim\,e_2\!\big)\,.$$
Combining this with Lemma~\ref{fibprodisom_lmm2}, we obtain the claim.
\end{proof}

Let $M,\Ga,X,\Ga_1,\ldots,\Ga_m$ and $f,g$ be as in Lemma~\ref{fibersign_lmm} with 
$$f\!=\!(f_1,\ldots,f_m)\!:M\lra X\!=\!X_1\!\times\!\ldots\!\times\!X_m
\quad\hbox{and}\quad 
g\!=\!g_1\!\times\!\ldots\!\times\!g_m\!:
\Ga\!=\!\Ga_1\!\times\!\ldots\!\times\!\Ga_m \lra X.$$
Suppose in addition that $B$ is another manifold and
$e\!:M\!\lra\!B$ and $e_i\!:\Ga_i\!\lra\!B$ with $i\!\in\![m]$ are maps.
Define
\begin{equation*}\begin{split}
\wt{f}=\!\big(\!(e,f_1),\ldots,(e,f_m)\!\big)\!:
M&\lra \wt{X}\!\equiv\!(B\!\times\!X_1)\!\times\!\ldots\!\times\!(B\!\times\!X_m),\\
\wt{g}=\!(e_1,g_1)\!\times\!\ldots\!\times\!(e_m,g_m)\!:
\Ga&\lra \wt{X}\!\equiv\!(B\!\times\!X_1)\!\times\!\ldots\!\times\!(B\!\times\!X_m).
\end{split}\end{equation*}
Let $e'\!:M_{\wt{f}}\!\times_{\wt{g}}\!\Ga\!\lra\!B$ be the map induced by~$e$.
For each $b\!\in\!B$, define
$$M_b=e^{-1}(b), \quad f_b\!=\!f|_{M_b}\!\!:M_b\lra X, \quad
\Ga_b\!=\!e_1^{-1}(b)\!\times\!\ldots\!\times\!e_m^{-1}(b), \quad
 g_b\!=\!g|_{\Ga_b}\!\!:\Ga_b\lra X;$$
see the top diagram in Figure~\ref{Bdrop_fig}.
Let $\io_b\!:\{b\}\lra\!B$ be the inclusion map.
The natural~map
\BE{Bdrop_e} \{b\}_{\io_b}\!\!\times_{e'}\!\!\big(M_{\wt{f}}\!\!\times_{\wt{g}}\!\Ga\big)
\lra (M_b)_{f_b}\!\!\times_{g_b}\!\!\Ga_b \subset M_{\wt{f}}\!\!\times_{\wt{g}}\!\Ga\EE
dropping the $b$ component is then a bijection.

Suppose also that the maps $e$ and $e_i$ are smooth,
the maps $\wt{f},\wt{f}|_{\prt M}$ are transverse to $\wt{g},\wt{g}|_{\prt\Ga}$, 
and~$b$ is a regular value of $e$, $e_i$ with $i\!\in\![m]$, and~$e'$.
This implies that the spaces~$M_b$ and $\Ga_b$ are smooth manifolds, 
the sequences 
\begin{gather}\label{Mprojses_e}
0\lra T_pM_b\lra T_pM  \xlra{\nd_pe} T_bB \lra 0 \qquad\hbox{and}\\
\label{Gaprojses_e}
0\lra T_{q_i}\!\big(e_i^{-1}(b)\!\big) \lra T_{q_i}\Ga_i \xlra{\nd_{q_i}e_i}  T_bB\lra 0,
\end{gather}
are exact for all $p\!\in\!M_b$, $q\!\equiv\!(q_1,\ldots,q_m)\!\in\!\Ga_b$, and $i\!\in\![m]$,
and the bijection~\eref{Bdrop_e} is a diffeomorphism between smooth manifolds.

\begin{figure}
\begin{small}
$$\xymatrix{\Ga_b \ar[rr]\ar[d]_{g_b}&& \Ga\ar[d]_{\wt{g}}\ar[dr]\ar[dll]_g\\
X \ar[rr]&& \wt{X}\ar[r]& B^m\ni b^m& M_{\wt{f}}\!\!\times_{\wt{g}}\!\Ga\ar[ull]\ar[dll]\ar[d]^{e'}\\
M_b\ar[rr]\ar[u]^{f_b}&& M\ar[u]^{\wt{f}}\ar[rr]^e\ar[ull]^f&& B\ni b
}$$
$$\xymatrix{& 0\ar[d]& 0\ar[d]& 0\ar[d]&\\
0\ar[r]& T_{(p,q)}(\!(M_b)_{f_b}\!\!\times_{g_b}\!\!\Ga_b)\ar[d]\ar[r]& 
T_{(p,q)}(M_{\wt{f}}\!\!\times_{\wt{g}}\!\Ga)\ar[d]\ar[r]^>>>>>>>>{\nd_{(p,q)}e'}&  
T_bB\ar[r]\ar[d]& 0\\
0\ar[r]& T_pM_b\!\oplus\!T_q\Ga_b\ar[d]_{\nd_qg_b-\nd_pf_b}\ar[r]&
T_pM\!\oplus\!T_q\Ga  \ar[d]_{\nd_q\wt{g}-\nd_p\wt{f}}\ar[r]& 
T_bB\!\oplus\!m T_bB \ar[r]\ar[d]\ar[r]& 0\\
0\ar[r]& T_{f(p)}X \ar[d]\ar[r]& T_{\wt{f}(p)}\wt{X} \ar[d]\ar[r]& m T_bB\ar[r]\ar[d]&  0\\
& 0& 0&0}$$
\end{small}
\caption{The maps of Lemma~\ref{Bdrop_lmm} and 
a commutative square of exact sequences for its proof.}
\label{Bdrop_fig}
\end{figure}

\begin{lmm}\label{Bdrop_lmm}
Suppose the manifolds $M$, $\Ga$, $B$, $M_b$, and $\Ga_b$ are oriented so that 
the exact sequences~\eref{Mprojses_e} and~\eref{Gaprojses_e} have signs
$$(-1)^{(\dim\,M_b)(\dim\,B)} \qquad\hbox{and}\qquad (-1)^{(\dim\,e_i^{-1}(b)\!)(\dim\,B)},$$
respectively, for all $p\!\in\!M_b$, $q\!\equiv\!(q_1,\ldots,q_m)\!\in\!\Ga_b$, and $i\!\in\![m]$.
The diffeomorphism~\eref{Bdrop_e} then has sign~$(-1)^{\ep}$ with respect to the fiber product 
orientations on the two sides, where
$$\ep= (\dim\,B)\bigg(\!\dim\,M_b\!+\!
\sum_{i=1}^m (i\!+\!1)\codim\,g_i\big|_{e_i^{-1}(b)}\!\bigg)\,.$$
\end{lmm}

\begin{proof} Let $(p,q)\!\in\!(M_b)_{f_b}\!\!\times_{g_b}\!\!\Ga_b$ with $q\!\equiv\!(q_1,\ldots,q_m)$.
We use the commutative square of exact sequences in Figure~\ref{Bdrop_fig}.
The two maps in the right column are given by 
$$v\lra \big(v,(v,\ldots,v)\!\big) \qquad\hbox{and}\qquad 
(v,(w_1,\ldots,w_m)\!\big)\lra \big(w_1\!-\!v,\ldots,w_m\!-\!v\big);$$
this column is thus compatible with the direct sum orientations.
The bottom row is compatible with the product orientations if and only~if
$$(\dim\,B)\sum_{i=1}^m i\,\dim\,X_i\in 2\Z.$$
The middle row is orientation-compatible if and only~if
$$(\dim\,B)\bigg(\!\dim\,M_b\!+\!\sum_{i=1}^m (i\!+\!1)\dim\,e_i^{-1}(b)\!\bigg)\in2\Z.$$
The left and middle columns are compatible with the fiber product orientations on
the top spaces.
The top row is compatible with the orientation on the left-hand side of~\eref{Bdrop_e}
and the fiber product orientation on $T_{(p,q)}(M_{\wt{f}}\!\!\times_{\wt{g}}\!\Ga)$.
Along with Lemma~6.3 in~\cite{RealWDVV}, this implies that the diffeomorphism~\eref{Bdrop_e}
is orientation-preserving at $(b,p,q)$ if and only~if
$$(\dim\,B)\sum_{i=1}^m i\,\dim\,X_i+
(\dim\,B)\bigg(\!\dim\,M_b\!+\!\sum_{i=1}^m (i\!+\!1)\dim\,e_i^{-1}(b)\!\bigg)
+(\dim\,X)(\dim\,B)\in2\Z\,.$$
This establishes the claim.
\end{proof}

\subsection{Moduli spaces}
\label{Ms_subs}

Suppose $k,l\!\in\!\Z^{\ge0}$ and $k\!+\!2l\!\ge\!3$. 
Let $\cM_{k,l}$ be the moduli space of $k$ distinct boundary marked points $x_1,\ldots,x_k$ placed {\it in counter-clockwise order} and $l$ distinct interior marked points $z_1,\ldots,z_l$ on the unit disk $\D$. We orient $\cM_{k,l}$ as follows. Let $\cM_{1,1}$ and $\cM_{3,0}$ be plus points. We identify $\cM_{0,2}$ with the interval $(0,1)$ by taking $z_1\!=\!0$ and $z_2\!\in\!(0,1)$ and orient it by the negative orientation of $(0,1)$. We then orient other $\cM_{k,l}$ inductively. If $k\!\ge\!1$, we orient $\cM_{k,l}$ so that the short exact sequence
\BE{cml_e}0\lra T_{x_k}S^1\lra T\cM_{k,l}\xrightarrow{\nd\ff^\R_k} T\cM_{k-1,l}\lra0\EE
induced by the forgetful morphism $\ff^\R_k$ dropping $x_k$ has sign $(-1)^k$ with respect to the counter-clockwise orientation of $S^1$. Thus, 
$$T\cM_{k,l}\approx T\cM_{k-1,l}\oplus T_{x_k}S^1.$$
If $l\!\ge\!1$, we orient $\cM_{k,l}$ so that the short exact sequence
\BE{cmk_e}0\lra T_{z_l}\D\lra T\cM_{k,l}\xrightarrow{\nd\ff^\C_l} T\cM_{k,l-1}\lra0\EE
induced by the forgetful morphism $\ff^\C_l$ dropping $z_l$ is orientation-compatible with respect to the complex orientation of $\D$. 
By a direct check, the orientations of~$\cM_{1,2}$ induced from~$\cM_{0,2}$ via~\eref{cml_e} 
and from~$\cM_{1,1}$ via~\eref{cmk_e} are the same, and 
the orientations of~$\cM_{3,1}$ induced from $\cM_{1,1}$ via~\eref{cml_e} and 
from~$\cM_{3,0}$ via~\eref{cmk_e} are also the same. 
Since the fibers of $\ff^{\C}_l$ are even-dimensional, 
it follows that the orientation on~$\cM_{k,l}$ above is well-defined. 
This orientation extends to the Deligne-Mumford compactification $\ov\cM_{k,l}$ of $\cM_{k,l}$. 

Fix a symplectic manifold $(X,\om)$ of dimension~$2n$, 
a Lagrangian submanifold~$Y$, 
a relative OSpin-structure $\os$ on $Y$, and $\be\!\in\!H^\om_2(X,Y)$.
Let $J\!\in\!\cJ_\om$.
For 
\BE{udfn_e} [\u]\equiv\big[u\!:(\D,S^1)\!\lra\!(X,Y),(x_i)_{i\in[k]},(z_i)_{i\in[l]}\big]
\in \M_{k,l}(\be;J)\,,\EE
let
$$D_{J;\u}\!:\Ga\big(u^*TX,u|_{S^1}^*TY\big)\lra
\Ga\big(T^*\D^{0,1}\!\otimes_{\C}\!u^*(TX,J)\big)$$
be the linearization of the $\{\dbar_J\}$-operator on the space of maps from~$(\D,S^1)$ to $(X,Y)$. 
By Proposition~8.1.1 in~\cite{FOOO}, 
the OSpin-structure $\os$ determines an orientation on $\det(D_{J;\u})$. 

Suppose $B$ is a smooth manifold (possibly with boundary) and
$\wt J\equiv(J_t)_{t\in B}$ is a smooth generic family in~$\cJ_\om$.
We define the moduli spaces
$$\M_{k,l}(\be;\wt J)\subset \M_{k,l}^{\st}(\be;\wt J)$$
of $\wt{J}$-holomorphic degree~$\be$ maps, evaluation maps
$$\evb_i\!:\M_{k,l}^{\st}(\be;\wt{J})\lra B\!\times\!Y,~i\!\in\![k], 
\quad\hbox{and}\quad
\evi_i\!:\M_{k,l}^{\st}(\be;\wt{J})\lra B\!\times\!X,~i\!\in\![l],$$
and the fiber products involving these spaces and maps as at the end of 
Section~\ref{Notation_subs}.
If $Y$ admits a relative OSpin-structure and thus is orientable, 
the Maslov index~\eref{Maslovdfn_e} is~even.
If in addition $n$ is odd, then
\BE{fMddimpar_e} \dim\,\M_{k,l}^{\st}(\be;\wt{J})\equiv k\!+\!\dim\,B \qquad
\hbox{mod}~2.\EE

For each $t\!\in\!B$ and $[\u]\!\in\!\M_{k,l}(\be;J_t)$ as above, define
\begin{gather*}
\wt{D}_{t;\u}\!:T_tB\!\oplus\!\Ga\big(u^*TX,u|_{S^1}^*TY\big)\lra
\Ga\big(T^*\D^{0,1}\!\otimes_{\C}\!u^*(TX,J)\big), \\
\wt{D}_{t;\u}(v,\xi)=\frac12 v(\wt{J})\!\circ\!\nd u\!\circ\!\fj+
D_{J_t;\u}\xi,
\end{gather*}
where $\fj$ is the complex structure on~$\D$.
We orient the determinant of $\wt{D}_{t;\u}$ so that the short exact sequence 
$$0\lra D_{t;\u}\lra \wt{D}_{t;\u}\lra (T_tB\!\lra\!0)\lra0$$
of operators has sign $(-1)^{(\dim\,B)(\dim\,Y)}$.
Thus,
$$\ker\wt{D}_{t;\u}\approx T_tB\!\oplus\!\big(\ker D_{t;\u}\big)$$
if the operator $D_{t;\u}$ is surjective.
We orient $\M^\st_{k,l}(\be;\wt J)$ by requiring the short exact sequence 
\BE{fMorientdfn_e}0\lra \ker \wt D_{t;\u}\lra T_{(t,\u)}\M_{k,l}(\be;\wt J)
\xrightarrow{d\ff} T_{\ff(\u)}\ov\cM_{k,l}\lra0\EE
to be orientation-compatible, where $\ff$ is the forgetful morphism dropping the map part of~$\u$.

\begin{rmk}\label{Morient_rmk}
The above paragraph endows $\M^\st_{k,l}(\be;\wt J)$ with an orientation under
the assumption that $k\!+\!2l\!\ge\!3$.
If $k\!+\!2l\!<\!3$, one first stabilizes the domain of~$\u$ by adding one or 
two interior marked points, then orients 
the tangent space of the resulting map as above, and 
finally drops the added marked points using the canonical complex orientation of~$\D$;
see the proof of Corollary~1.8 in~\cite{Penka1}.
\end{rmk}

\vspace{-.1in}

If $L$ is a finite set,
we orient $\M_{k,L}^{\st}\!(\be;\wt{J})$ from $\M_{k,|L|}^{\st}(\be;\wt{J})$
by identifying~$L$ with~$[|L|]$ as~sets.
The resulting orientation does not depend on the choice of 
the identification.

\subsection{Disk bubbling strata}
\label{diskbubb_subs}

In this section,
we compare two natural orientations on the disk bubbling strata 
of $\M_{k,l}^{\st}(\be;J)$ and on associated fiber products.
Corollary~\ref{Scompare_crl} at the end of this section ensures
pairwise cancellations of boundary components
of fiber product spaces
in the proofs of Lemmas~\ref{BCpseudo_lmm} and~\ref{psisot_lmm} in 
Section~\ref{bcisot_subs}, of~\eref{cob_e} in Section~\ref{countinv_subs},
and of~\ref{lagl_it1} and~\ref{lagl_it2} in Section~\ref{LangOGW_subs}.

We continue with the setup of Section~\ref{Ms_subs}.
Suppose in addition $\be_1,\be_2\in H^\om_2(X,Y)$ with $\be_1\!+\!\be_2\!=\!\be$,
$[l]\!=\!L_1\!\sqcup\!L_2$, and $k_1,k_2\!\in\!\{0\}\!\sqcup\![k\!+\!1]$ with $k_1\!<\!k_2$. 
Let $\cS\subset\prt\M^\st_{k,l}(\be;\wt J)$ consist of $J$-holomorphic maps 
from $(\D^2\!\v\!\D^2,S^1\!\v\!S^1)$
to $(X,Y)$ of degrees~$\be_1$ and~$\be_2$ on the two components,
with the second component carrying the boundary marked points indexed by~${k_1\!+\!1,\ldots,k_2\!-\!1}$ 
and the interior marked points indexed by~$L_2$. 
We re-index the boundary marked points on the second disk in the counterclockwise order starting
with the node and on the first disk starting with~$x_1$ if $k_1\!\ge\!1$ and
with the node otherwise.
Define 
$$\M_1\equiv\M_{k-(k_2-k_1)+2,L_1}(\be_1;\wt J),\quad\M_2\equiv\M_{k_2-k_1,L_2}(\be_2;\wt J).$$
As a space, $\cS$ is the fiber product 
$$\cS={\M_1}_{\,\,\evb_{k_1\!+\!1}}\!\!\!\times_{\evb_1}\!\M_2.$$

\begin{lmm}\label{Scompare_lmm}
If $n\!\equiv\!\dim\,Y$ is odd, the orientation of $\cS$ as a boundary of $\M^\st_{k,l}(\be;\wt J)$ differs from the orientation of $\cS$ as the above fiber product by $(-1)^{\ep}$, where
$$\ep=k_1k_2\!+\!kk_1\!+\!kk_2\!+\!k_1\!+\!1+(\dim\,B)\big(k_1\!+\!k_2\!+\!1\big).$$ 
\end{lmm}

\begin{proof}
Suppose first that $B$ is a point.
The conclusion of Lemma~6.4 in~\cite{RealWDVV3} and its proof apply to any $n\notin2\Z$ 
and imply the claim in this special case.
The difference in the sign is due to the placement of the node of the first disk according to its cyclic order position here instead of the last position in~\cite{RealWDVV3}. 
Below we deduce the general case. 

\begin{figure}
\begin{small}
$$\xymatrix{ & 0\ar[d] & 0\ar[d] & 0\ar[d] & \\
0\ar[r] & T_\u\cS_t \ar[r]\ar[d] & T_\u\cS \ar[r]\ar[d] & T_tB\ar[r]\ar[d] & 0 \\
0\ar[r] & T_{\u_1}\M_{1;t}\!\oplus\!T_{\u_2}\M_{2;t} 
\ar[r]\ar[d] & T_{\u_1}\M_1\!\oplus\!T_{\u_2}\M_2 \ar[r]\ar[d] 
& T_tB\!\oplus\!T_tB \ar[r]\ar[d] & 0\\
0\ar[r] & \cN\De_Y\big|_{(\evb_1(\u_2),\evb_1(\u_2)\!)} 
\ar[r]\ar[d] & \cN\De_{B\times Y}\big|_{(\!(t,\evb_1(\u_2)),(t,\evb_1(\u_2)\!)\!)} 
\ar[r]\ar[d] & \cN\De_B\big|_{(t,t)} \ar[r]\ar[d] & 0 \\ 
& 0 & 0 & 0 }$$
\end{small}
\caption{A commutative square of exact sequences for the proof of Lemma~\ref{Scompare_lmm}.}
\label{Sbdryvsfp_fig}
\end{figure}

For $t\!\in\!B$, denote 
$$\M_{1;t}\equiv\M_{k-(k_2-k_1)+2,L_1}(\be_1;J_t),\quad
\M_{2;t}\equiv\M_{k_2-k_1,L_2}(\be_2;J_t),\quad
\cS_t=\cS\cap\M^\st_{k,l}(\be;J_t)\,.$$ 
For $\u\!\in\!\cS_t$,  let $\u_1\!\in\!\M_{1;t}$ and $\u_2\!\in\!\M_{2;t}$ 
be the corresponding component maps.
For the simplicity of terminology, we assume that $D_{J_t;\u}$ is onto.
We use the commutative square of exact sequences in Figure~\ref{Sbdryvsfp_fig}.
The last column is as in~\eref{cNDeor_e}; it is thus orientation-compatible.
The bottom row is induced by the isomorphism~\eref{cNsplit_e};
it is compatible with the canonical orientations on the normal bundles
if and only~if $(\dim\,Y)(\dim\,B)$ is even.
The middle row is orientation-compatible if and only~if
$$(\dim\,\M_{1;t})(\dim\,B)\!+\!(\dim\,\M_{2;t})(\dim\,B)
\!+\!(\dim\,\M_{2;t})(\dim\,B) \equiv 
(\dim\,B)\big(k\!+\!k_1\!+\!k_2\big) ~~\hbox{mod}~2$$
is even; the congruence above follows from~\eref{fMddimpar_e}.
The top row is compatible with the boundary orientations of~$\cS_t$ and~$\cS$
if and only~if
$$(\dim\,\cS_t)(\dim\,B)\equiv(\dim\,B)(k\!-\!1) ~~\hbox{mod}~2$$
is even.
The left and middle columns respect the fiber-product orientations on~$\cS_t$ and~$\cS$,
respectively.
Along with the $B\!=\!\pt$ case above and Lemma~6.3 in~\cite{RealWDVV},
this implies that the boundary and fiber-product orientations on~$T_{\u}\cS$ are the same
if and only~if
\begin{equation*}\begin{split}
(\dim\,Y)(\dim\,B)&+(\dim\,B)\big(k\!+\!k_1\!+\!k_2\big)+(\dim\,B)(k\!-\!1)\\
&+\big(k_1k_2\!+\!kk_1\!+\!kk_2\!+\!k_1\!+\!1\big)
\!+\!(\dim\,Y)(\dim\,B)\equiv \ep ~~\hbox{mod}~2
\end{split}\end{equation*}
is even.
\end{proof}

With the setup as above Lemma~\ref{Scompare_lmm}, denote 
$$K_2\equiv\big\{k_1\!+\!1,\ldots,k_2\!-\!1\big\},
~~K^1_1\equiv\big\{1,\ldots,k_1\big\}, ~~K^2_1\equiv\{k_2,\ldots,k\},
 ~~K_1\equiv K^1_1\!\cup\!K^2_1.$$ 
Suppose in addition $I\!\subset\![k]$ and $J\!\subset\![l]$ are such that 
$K_2\!\cup\!K^2_1\!\subset\!I$.
Let
$$\big\{\fb_i\!:\!Z_{\fb_i}\lra B\!\times\!Y\big\}_{i\in I} \quad\hbox{and}\quad
\big\{\Ga_i\!:\!Z_{\Ga_i}\lra B\!\times\!X\big\}_{i\in J}$$ 
be smooth maps from oriented manifolds in general positions so~that 
$$\codim\,\fb_i\not\in2\Z~~\forall\,i\!\in\!I  \quad\hbox{and}\quad
\codim\,\Ga_i\in 2\Z~~\forall\,i\!\in\!J\,.$$
Define 
$$\fbb\equiv\evb_1:
\M_2\!\!\fiber\!\!\big(\!(i\!-\!k_1\!+\!1,\fb_i)_{i\in K_2};(i,\Ga_i)_{i\in J\cap L_2}\big)
\lra B\!\times\!Y.$$
Under the assumptions above, this is a smooth map of even codimension.

\begin{crl}\label{Scompare_crl}
The natural isomorphism
\BE{ScompareCrl_e}\begin{split}
\cS\fiber&\!\big(\!(i,\fb_i)_{i\in I};(i,\Ga_i)_{i\in J}\big)\\
&\approx\M_1\!\!\fiber\!\!
\big(\!(i,\fb_i)_{i\in I\cap K^1_1},(k_1\!+\!1,\fbb),(i\!-\!k_2\!+\!k_1\!+\!2,\fb_i)_{i\in K^2_1};
(i,\Ga_i)_{i\in J\cap L_1}\big)
\end{split}\EE
has sign  $(-1)^{\ep}$ with respect to the boundary orientation of~$\cS$, where
$$\ep=k_1\!+\!k_2\!+\!
(\dim\,B)\big(k\!+\!k_1\!+\!(k_1\!+\!k_2\!+\!1)|I|\big).$$
\end{crl}

\begin{proof} For $r\!=\!1,2$, let 
$$g_r\!\equiv\!\prod_{i\in I\cap K_r}\!\!\!\!\!\fb_i\times\!
\prod_{i\in J\cap L_r}\!\!\!\!\!\Ga_i\!:
G_r\!\equiv\!\prod_{i\in I\cap K_r}\!\!\!\!\!Z_{\fb_i}\!\times\!
\prod_{i\in J\cap L_r}\!\!\!\!\!Z_{\Ga_i}\lra 
X_r\!\equiv\!(B\!\times\!Y)^{I\cap K_r}\!\!\times\!(B\!\times\!X)^{J\cap L_r}\,.$$
We orient $G_r$ and $X_r$ based on the orderings of the elements of $I\!\cap\!K_r$ and $J\!\cap\!L_r$.
By Corollary~\ref{fibprodisom_crl} with 
\begin{alignat*}{2}
f_1\!\equiv\!\big(\!(\evb_i)_{i\in I\cap K_1^1},(\evb_{i-k_2+k_1+2})_{i\in K_1^2},
(\evi_i)_{i\in J\cap L_1}\big)\!:\M_1&\lra X_1,
&~~ e_1\!\equiv\!\evb_{k_1+1}\!:\M_1&\lra B\!\times\!Y,\\
f_2\!\equiv\!\big(\!(\evb_{i-k_1+1})_{i\in K_2},(\evi_i)_{i\in J\cap L_2}\big)\!:\M_2&\lra X_2,
&~~ e_2\!\equiv\!\evb_1\!:\M_2&\lra B\!\times\!Y,
\end{alignat*}
and~\eref{fMddimpar_e}, 
the sign of the diffeomorphism~\eref{ScompareCrl_e} 
with respect to the fiber product orientation on~$\cS$, 
the orientations $G_1\!\times\!G_2$ and $X_1\!\times\!X_2$ on
the domain and target of $g_1\!\times\!g_2$, and 
the orientations
$$G_1\!\times\!
\Big(\M_2\!\!\fiber\!\!\big(\!(i\!-\!k_1\!+\!1,\fb_i)_{i\in K_2};(i,\Ga_i)_{i\in J\cap L_2}
\big)\!\!\Big) \quad\hbox{and}\quad 
X_1\!\times\!(B\!\times\!Y)$$ 
on the domain and target of $g_1\!\times\!e_2'$
is $(-1)^{\ep_1}$, where 
\begin{equation*}\begin{split}
\ep_1=(\dim\,B\!+\!k_2\!-\!k_1)|I\!\cap\!K_1|
&+\big((\dim\,B)(|I\!\cap\!K_1|\!+\!|J\!\cap\!L_1|)\!+\!|I\!\cap\!K_1|\big)(k_2\!-\!k_1\!+\!1)\\
&+(\dim\,B\!+\!1)|I|.
\end{split}\end{equation*}

By Lemma~\ref{fibprodflip_lmm},
the above orientations on the domain and target of $g_1\!\times\!g_2$ twist the resulting orientation
on the left-hand side of~\eref{ScompareCrl_e} by $(-1)^{\ep_2}$, where
$$\ep_2=(k\!-\!k_2\!+\!1)(k_2\!-\!k_1\!-\!1)+(k_2\!-\!k_1\!-\!1)|J\!\cap\!L_1|(\dim\,B).$$
The above orientations on the domain and target of $g_1\!\times\!e_2'$  twist the resulting orientation
on the right-hand side of~\eref{ScompareCrl_e} by $(-1)^{\ep_3}$, where
$$\ep_3=(k\!-\!k_2\!+\!1)(\dim\,B\!+\!1).$$
Along with Lemma~\ref{Scompare_lmm}, this implies that the claim of the corollary holds
with
$$\ep=k_1k_2\!+\!kk_1\!+\!kk_2\!+\!k_1\!+\!1+(\dim\,B)\big(k_1\!+\!k_2\!+\!1\big)
\!+\!\ep_1\!+\!\ep_2\!+\!\ep_3.$$
This completes the proof.
\end{proof}

\subsection{Sphere bubbling strata}
\label{sphbubb_subs}

We next establish an analogue of Corollary~\ref{Scompare_crl} for 
the sphere bubbling stratum 
$$\cS_{\be,L}^0\subset\prt\ov\M_{0,L}(\be;J)$$
of the stable map compactification $\ov\M_{0,L}(\be;J)$
of $\M_{0,L}^{\st}(\be;J)$.
Corollary~\ref{Scompare2_crl} at the end of this section is used to prove
Proposition~\ref{LangOGW_prp} in Section~\ref{LangOGW_subs}.
As before, we assume that $(X,\om)$ is a symplectic manifold, 
$Y\!\subset\!X$ is a Lagrangian submanifold, 
$\os$ is a relative OSpin-structure on~$Y$, 
$\be\!\in\!H_2^{\om}(X,Y)$, and $L$ is a finite set.
However, the dimension of~$Y$ need not be odd for the purposes of 
the present section.

Let $B\!\in\!H_2(X;\Z)$ and
$$\cS\subset \cS_{q_X(B),L}^0 \subset\prt\ov\M_{0,L}\big(q_X(B);J\big)$$
be the open subspace of the sphere bubbling stratum consisting of 
the maps from~$\D^2$ that descend to degree~$B$ maps from~$\P^1$.
This codimension~1 stratum inherits a boundary orientation from the orientation
of $\M_{0,L}(q_X(B);J)$ induced by the relative OSpin-structure~$\os$ on~$Y$.
This induced orientation depends on the orientation~$\fo$ on~$Y$ determined by~$\os$
and on $\lr{w_2(\os),B}$ only.
As a space, $\cS$ is the fiber product 
\BE{cSsphspl_e}\cS=  \M^\C_{\{0\}\sqcup L}(B;J)_{\ev_0}\!\!\times_{\io_Y}\!\!Y,\EE
where $\io_Y\!:Y\!\lra\!X$ is the inclusion.

\begin{lmm}\label{Scompare2_lmm}
The orientation of $\cS$ as a boundary of $\ov\M_{0,L}(q_Y(B);J)$ 
differs from the orientation of $\cS$ as the above fiber product 
by~$(-1)^{\lr{w_2(\os),B}}$.
\end{lmm}

\begin{proof}
In light of Remark~\ref{Morient_rmk},
it is sufficient to establish the claim under the assumption that $|L|\!\ge\!2$.
We denote by $\cM_{\{0\}\sqcup L}^{\C}$ the moduli space of 
distinct points on~$\P^1$ labeled by the set $\{0\}\!\sqcup\!L$.
Let $\cS_{0,L}\!\subset\!\ov\cM_{0,L}$ be the sphere bubbling stratum,
$\cN_{0,L}$ be its oriented normal bundle, and 
$\cN$ be the oriented normal bundle of~$\cS$ in $\ov\M_{0,L}(q_Y(B);J)$.
Let $\wt\u\!\in\!\cS$, $(\u,y)$ be the corresponding element of the fiber product in~\eref{cSsphspl_e},
and~$\wt\cC$ and~$\cC$ be the marked domains of~$\wt\u$ and~$\u$, respectively. 
We denote~by
$$D_{J;\u}^{\C}\!:\Ga\big(u^*TX\big)\lra
\Ga\big((T^*\P^1)^{0,1}\!\otimes_{\C}\!u^*(TX,J)\big)$$
the linearization of the $\{\dbar_J\}$-operator on the space of maps from~$\P^1$ to~$X$
at~$\u$.
The determinant of this Fredholm operator has a canonical complex orientation.

A forgetful morphism $\ff\!:\ov\cM_{0,L}\!\lra\!\ov\cM_{0,2}$ dropping all but two
of the marked points induces the short exact sequences given by 
the columns in the first diagram of Figure~\ref{Scompare2_fig}.
The middle and right columns and the top row in this diagram
respect the orientations;
the middle row is orientation-compatible with respect to the boundary orientation
on~$T_{\wt\cC}\cS_{0,L}$.
By the definition of the orientation on~$\cM_{0,2}$,
the single element $\wt\cC_2\!\in\!\cS_{0,2}$ is a plus point with respect 
to the boundary orientation.
Thus, the isomorphism in the bottom row is orientation-preserving.
Along with Lemma~6.3 in~\cite{RealWDVV}, this implies that the isomorphism
in the left column is also orientation-preserving with respect to the complex orientation
on its domain and the boundary orientation on its target.

\begin{figure}
\begin{small}
$$\xymatrix{ & 0\ar[d] & 0\ar[d] \\
0\ar[r] & T_\cC\cM_{\{0\}\sqcup L}^{\C} \ar[r]^{\Id}\ar[d]^{\approx} & 
T_\cC\cM_{\{0\}\sqcup L}^{\C} \ar[r]\ar[d] & 0\ar[d] \\
0\ar[r] & T_{\wt\cC}\cS_{0,L}  \ar[r]\ar[d]^{\nd_{\wt\cC}\ff} 
& T_{\wt\cC}\cM_{0,L} \ar[r]\ar[d]^{\nd_{\wt\cC}\ff}  
& \cN_{0,L}|_{\wt\cC} \ar[r]\ar[d]^{\nd_{\wt\cC}\ff}  & 0\\
0\ar[r] & T_{\wt\cC_2}\cS_{0,2}
\ar[r]\ar[d] & T_{\wt\cC_2}\cM_{0,2}
\ar[r]^{\approx}\ar[d] & \cN_{0,2}|_{\wt\cC_2} \ar[r]\ar[d] & 0 \\ 
& 0 & 0 & 0 }$$
$$\xymatrix{ & 0\ar[d] & 0\ar[d] \\
0\ar[r] & \ker D_{J;\wt\u} \ar[r]^{\Id}\ar[d]& 
\ker D_{J;\wt\u} \ar[r]\ar[d] & 0\ar[d] \\
0\ar[r] & T_{\wt\u}\cS  \ar[r]\ar[d]^{\nd_{\wt\u}\ff} 
& T_{\wt\u}\M_{0,L}(q_Y(B);J) \ar[r]\ar[d]^{\nd_{\wt\u}\ff}  
& \cN|_{\wt\u} \ar[r]\ar[d]^{\nd_{\wt\u}\ff}  & 0\\
0\ar[r] & T_{\wt\cC}\cS_{0,L}  \ar[r]\ar[d] 
& T_{\wt\cC}\cM_{0,L} \ar[r]\ar[d] 
& \cN_{0,L}|_{\wt\cC} \ar[r]\ar[d]  & 0\\
& 0 & 0 & 0 }$$
$$\xymatrix{ & 0\ar[d] & 0\ar[d] & 0\ar[d]\\
0\ar[r] & \ker D_{J;\wt\u} \ar[r]\ar[d] & 
\ker D_{J;\u}^{\C}\!\oplus\!T_yY \ar[r]\ar[d] &  T_yX \ar[r]\ar[d]^{\Id} & 0 \\
0\ar[r] & T_{\wt\u}\cS  \ar[r]\ar[d]^{\nd_{\wt\u}\ff} 
& T_{\u}\M_{\{0\}\sqcup L}^{\C}(B;J)\!\oplus\!T_yY \ar[r]\ar[d]_{\nd_{\u}\ff}  
& T_yX \ar[r]\ar[d]  & 0\\
0\ar[r] & T_{\wt\cC}\cS_{0,L}  \ar[r]^>>>>>>>{\approx}\ar[d] 
& T_{\cC}\cM_{\{0\}\sqcup L}^{\C} \ar[r]\ar[d] & 0\\
& 0 & 0}$$
\end{small}
\caption{Commutative squares of exact sequences for the proof of Lemma~\ref{Scompare2_lmm}.}
\label{Scompare2_fig}
\end{figure}  

The forgetful morphism $\ff\!:\M_{0,L}(q_Y(B);J)\!\lra\!\ov\cM_{0,L}$ dropping 
the map component induces the short exact sequences given by 
the columns in the second diagram of Figure~\ref{Scompare2_fig}.
The middle and right columns and the top row in this diagram
respect the orientations;
the middle and bottom rows are orientation-compatible with respect to 
the boundary orientations on~$T_{\wt\u}\cS$ and~$T_{\wt\cC}\cS_{0,L}$, respectively.
Along with Lemma~6.3 in~\cite{RealWDVV}, this implies that the isomorphism
in the left column is also orientation-compatible with respect to 
the boundary orientations on~$T_{\wt\u}\cS$ and~$T_{\wt\cC}\cS_{0,L}$.

The forgetful morphism $\ff$ of the previous paragraph and its complex analogue
induce the short exact sequences given by 
the columns in the third diagram of Figure~\ref{Scompare2_fig}.
The middle column respects the complex orientations on $\ker D_{J;\u}^{\C}$,
$T_{\u}\M_{\{0\}\sqcup L}^{\C}(B;J)$, and $T_{\cC}\cM_{\{0\}\sqcup L}^{\C}$.
By the proof of Proposition~8.1.1 in~\cite{FOOO}, 
the sign of the top row in this diagram with respect to 
the orientation on $\ker D_{J;\wt\u}$ induced by the relative OSpin-structure~$\os$
and the complex orientation on~$\ker D_{J;\u}^{\C}$ is~$(-1)^{\lr{w_2(\os),B}}$.
By the conclusion concerning the first diagram,
the isomorphism in the bottom row is orientation-preserving with respect 
to the boundary orientation on~$T_{\wt\cC}\cS_{0,L}$ and
the complex orientation on~$T_{\cC}\cM_{\{0\}\sqcup L}^{\C}$.
By the conclusion concerning the second diagram,
the left column in the third diagram respects the boundary orientations
on~$T_{\wt\u}\cS$ and~$T_{\wt\cC}\cS_{0,L}$. 
By definition, the middle row respects the fiber product orientation on~$T_{\wt\u}\cS$.
Along with Lemma~6.3 in~\cite{RealWDVV}, this implies that 
the boundary and fiber product orientations on~$T_{\wt\u}\cS$ differ
by~$(-1)^{\lr{w_2(\os),B}}$.
\end{proof}

From Lemmas~\ref{Scompare2_lmm} and~\ref{fibprodisom_lmm1}, 
we immediately obtain the following statement.

\begin{crl}\label{Scompare2_crl}
Let $(X,\om)$, $Y$, $\os$, and $\be$ be as above Lemma~\ref{Scompare2_lmm} and
\hbox{$L\!\equiv\!(\Ga_1,\ldots,\Ga_l)$} 
be a tuple of smooth maps from oriented manifolds
to~$X$ in general position.
The natural isomorphism
$$\cS_{\be,L}^0\!\fiber\!\big(;(i,\Ga_i)_{\Ga_i\in L}\big)
\approx\bigsqcup_{B\in q_Y^{-1}(\be)}\!\!\!\!\!\!
(-1)^{\lr{w_2(\os),B}}\M^\C_{\{0\}\sqcup L}(B;J)\!\!\fiber\!\!
\big(\!(1,Y),(i\!+\!1,\Ga_i)_{\Ga_i\in L})$$
is then orientation-preserving with respect to the boundary orientation of~$\cS_{\be,L}^0$
and the fiber product orientation on the right-hand side.
\end{crl}

\appendix

\section{General symplectic manifolds}
\label{vfc_app}

We now sketch an adaptation of the geometric construction described in
this paper to general symplectic manifolds, 
dropping the positivity assumptions~\eref{strongpos_e1} and~\eref{strongpos_e2}, 
in a way compatible with standard virtual class approaches,
such as in~\cite{LT,FO, HWZ,Pardon}.
As we only need evaluation maps from the disk moduli spaces to be pseudocycles,
a full virtual class construction and gluing across all strata of 
these spaces are not necessary.
Throughout this appendix, $R$ is a commutative ring containing~$\Q$,
$(X,\om)$ is a compact symplectic manifold, $Y\!\subset\!X$ is a compact
Lagrangian submanifold, and $\os$ is a relative OSpin-structure on~$Y$.

We denote $H_2(X,Y;\Z)$ modulo torsion by $H_2(X,Y)$. 
Let $\cB$ be a basis for $H_2(X,Y)$ so that all elements of $H_2^{\om}(X,Y)$
are linear combinations of the elements of~$\cB$ with nonnegative coefficients
and $\wt{H}_2(X,Y)$ be the collection of finite subsets of $\cB\!\times\!\Z$.
We write an element~$\wt\be$ of~$\wt{H}_2(X,Y)$ as $(A_b)_{b\in\cB}$,
with each $A_b\!\subset\!\Z$ being a finite subset (possibly empty).
For such an element~$\wt\be$, define
$$A_b(\wt\be)=A_b\subset\Z~~\forall\,b\!\in\!\cB, \quad
A(\wt\be)!=\prod_{b\in\cB}\!\big(|A_b(\wt\be)|!\big)\in\Z^+, \quad
|\wt\be|=\sum_{b\in\cB}\!\big|A_b(\wt\be)\big|b\in H_2(X,Y)\,.$$ 
For $\wt\be,\wt\be'\!\in\!\wt H_2(X,Y)$ such that $A_b(\wt\be)\!\cap\!A_b(\wt\be')\!=\!\eset$ 
for every $b\!\in\!\cB$, define
$$\wt\be\!+\!\wt\be'=\big(A_b(\wt\be)\!\sqcup\!A_b(\wt\be')\!\big)_{b\in\cB}\,.$$
The subset 
$$\wt{H}_2^{\om}(X,Y)\equiv\big\{\wt\be\!\in\!\wt{H}_2(X,Y)\!:
|\wt\be|\!\in\!\wt{H}_2^{\om}(X,Y)\big\}$$
of $\wt{H}_2(X,Y)$ has a natural partial order.
We define $\cC_{\om}(Y)$, $\cC_{\om;\al}(Y)$, $\cD_{\om}(\al)$, 
$\wt\cC_{\om}(Y)$, $\wt\cC_{\om;\wt\al}(Y)$, and $\cD_{\om}(\wt\al)$
as the collections $\cC_{\om}(Y)$, $\cC_{\om;\al}(Y)$, $\cD_{\om}(\al)$, 
$\wt\cC_{\om}(Y)$, $\wt\cC_{\om;\wt\al}(Y)$, and $\cD_{\om}(\wt\al)$
in Section~\ref{Notation_subs} with $H_2^{\om}(X,Y)$ replaced 
by~$\wt{H}_2^{\om}(X,Y)$.

Let $J\!\in\!\cJ_\om$.
For $k\!\in\!\Z^{\ge0}$, a finite set~$L$, and $\wt\be\!\in\!\wt{H}_2^{\om}(X,Y)$, let
$$\M^\st_{k,L}\big(\wt\be;J\big)=A(\wt\be)!\,\M^\st_{k,L}\big(|\wt\be|;J\big)\,.$$
The natural immersion 
$$\io_{\prt\wt\be}\!: 
\bigsqcup_{\begin{subarray}{c}\wt\be_1,\wt\be_2\in \wt{H}_2^{\om}(X,Y)\\
\wt\be_1+\wt\be_2=\wt\be\end{subarray}} \hspace{-.32in}
\M_{1,\eset}(\wt\be_1;J)_{\evb_1}\!\!\times\!\!_{\evb_1}
\M_{1,\eset}(\wt\be_2;J)\lra \M^\st_{0,0}\big(\wt\be;J\big)$$
is then a degree~2 covering map.
This statement extends to the disk moduli spaces with marked points.

We take a generic collection $\nu\!\equiv\!(\nu_{\wt\be})_{\wt\be\in\wt{H}_2^{\om}(X,Y)}$
of multi-valued inhomogeneous perturbations for the moduli spaces 
$$\ov\M_{0,\eset}\big(\wt\be;J\big)\equiv A(\wt\be)!\,\ov\M_{0,\eset}\big(|\wt\be|;J\big)$$
compatible with the immersions~$\io_{\prt\wt\be}$, with $\nu_{\eset}\!\equiv\!0$.
They lift to multi-valued inhomogeneous perturbations for the moduli spaces
$$\ov\M_{k,L}\big(\wt\be;J\big)\equiv A(\wt\be)!\,\ov\M_{k,L}\big(|\wt\be|;J\big)$$
via the forgetful morphisms; we denote the lifted perturbations also by~$\nu_{\wt\be}$.
These perturbations are compatible with the analogues of the immersions~$\io_{\prt\wt\be}$
for the disk moduli spaces with marked points and vanish on the moduli spaces 
$\ov\M_{k,L}(\eset;J)$ of degree~0 disks.
All combinations of the evaluation maps~$\evb_i$ and~$\evi_i$ from 
main strata of the distinct moduli spaces
$$\M_{k,L}\big(\wt\be;J,\nu_{\wt\be}\big)\equiv 
A(\wt\be)!\,\M_{k,L}\big(|\wt\be|;J,\nu_{\wt\be}\big)$$
of $(J,\nu)$-disks are transverse.

We denote~by
$$\M_{k,L}^{\st}\big(\wt\be;J,\nu_{\wt\be}\big)\subset\ov\M_{k,L}\big(\wt\be;J,\nu_{\wt\be}\big)$$
the subspace of maps from $(\D^2,S^1)$ and $(\D^2\!\vee\!\D^2,S^1\!\vee\!S^1)$ to~$(X,Y)$
and extend the use of the notation~$\fiber$ defined at the end of Section~\ref{Notation_subs}
to the moduli spaces of $(J,\nu)$-disks.
For $\al\!\in\!\cC_{\om}(Y)$ and $\eta\!\in\!\cD(\al)$, we replace~\eref{fMetaJdfn_e} with
$$\M_{\eta;J}\equiv\M^\st_{k_\bu(\eta),L_\bu(\eta)}
\big(\wt\be_\bu(\eta);J,\nu_{\wt\be_\bu(\eta)}\big),\qquad
\M^+_{\eta;J}\equiv
\M^\st_{k_\bu(\eta)+1,L_\bu(\eta)}\big(\wt\be_\bu(\eta);J,\nu_{\wt\be_\bu(\eta)}\big).$$
We then define a \sf{bounding chain} on $(\al,J,\nu)$ as in Definition~\ref{bndch_dfn} and
and~$\fbb_{\al'}$ in~\eref{fbbdfn_e}.
For the sake of proper scaling, we divide 
the right-hand sides of~\eref{JSinvdfn_e2}, \eref{JSinvdfn_e2b}, and~\eref{JSinvdfn_e2c}
by~$A(\wt\be)!$ when defining the associated disk counts
with $\be\!\in\!H_2^{\om}(X,Y)$ replaced by $\wt\be\!\in\!\wt{H}_2^{\om}(X,Y)$.

For a path $\wt J\!\equiv\!(J_t)_{t\in[0,1]}$ in~$\cJ_{\om}$ and a compatible path 
$\wt\nu\!\equiv\!(\nu_t)_{t\in[0,1]}$ of inhomogeneous perturbations, 
we modify the definitions of $\M_{\wt\eta;\wt{J}}$, $\M^+_{\wt\eta;\wt{J}}$,
pseudo-isotopy, and~$\fbb_{\wt\al'}$ in Section~\ref{Bc_subs} analogously.
The statements and proofs in \hbox{Sections~\ref{main_sec}-\ref{GWmainpf_sec}}
then readily adapt. 
In particular, pseudo-isotopic bounding chains still determine the same disk 
counts~\eref{OGWdfn_e}.
The proofs of Propositions~\ref{bndch_prp} and~\ref{psisot_prp} now 
ensure the existence of a bounding chain for every $(\al,J,\nu)$ 
and its uniqueness up to pseudo-isotopy, without 
the positivity assumptions~\eref{strongpos_e1} and~\eref{strongpos_e2}.
If $Y$ is an $R$-homology sphere, we then obtain disk counts~\eref{OGWdfn_e}
that depend only on~$|\wt\be|$, rather than~$\wt\be$, and thus well-defined
open GW-invariants~\eref{geomJSinvdfn_e}.

\section{Real Gromov-Witten invariants}
\label{RealGWs_sec}

We next interpret Solomon-Tukachinsky's adaptation of 
the main construction in~\cite{JS2} to the ``real setting" geometrically.
Throughout this appendix, 
$(X,\om,\phi)$ is a compact \sf{real symplectic manifold},  i.e.~$\phi$ 
is an involution on~$X$ so that \hbox{$\phi^*\om\!=\!-\om$}, 
$Y\!\subset\!X^{\phi}$ is a topological component of the fixed locus
(which is a Lagrangian submanifold of~$X$), 
and $\os$ is an OSpin-structure on~$Y$.
We assume that the dimension of~$X$ is~$2n$ with $n$ odd.
Let
$$\cJ_{\om}^{\phi}=\big\{J\!\in\!\cJ_{\om}^{\phi}\!:\phi_*J\!=\!-J\big\},\quad
\wh{H}^{2*}_{\phi}(X,Y;R)
=\bigoplus_{p=0}^n\big\{\ga\!\in\!\wh{H}^{2p}(X,Y;R)\!:\phi^*\ga\!=\!(-1)^p\ga\big\}.$$
There is a natural \sf{doubling map}
\BE{fdwchXdfn_e}\fd_Y\!:H_2\big(X,Y;\Z\big)  \lra H_2(X;\Z),\EE
which glues each map $f\!:(\Si,\prt\Si)\!\lra\!(X,Y)$ from an oriented bordered surface
with the map $\phi\!\circ\!f$ from $(\Si,\prt\Si)$ with the opposite orientation; 
see \cite[Sec~1.1]{RealWDVV3}.
This homomorphism vanishes on the image of the homomorphism $\Id\!+\!\phi_*$ 
on $H_2(X,Y;\Z)$ and thus descends to a homomorphism
$$\fd_Y\!:H_2\big(X,Y;\Z\big)\big/\Im\big\{\Id\!+\!\phi_*\big\}\lra H_2(X;\Z).$$

\vspace{-.1in}

We call a bordered pseudocycle 
$$\Ga\!: Z\lra X \qquad \big(\hbox{resp}.~~\Ga\!: Z\lra[0,1]\!\times\!X)$$
\sf{$\phi$-invariant} if there exists an involution~$\phi_Z$ on~$Z$, 
which is either orientation-preserving or reversing, such~that 
\BE{phiZdfn_e}\phi\!\circ\!\Ga=\Ga\!\circ\!\phi_Z
\qquad \big(\hbox{resp.~~}
\big\{\id_{[0,1]}\!\times\!\phi\big\}\!\circ\!\Ga=\Ga\!\circ\!\phi_Z\big)\,.\EE
We define $\sgn_{\phi}\Ga$ to be $+1$ if $\phi_Z$ above is orientation-preserving 
and $-1$ if $\phi_Z$ is orientation-reversing. 
We call a $\phi$-invariant bordered pseudocycle~$\Ga$ \sf{real}
as above
if the codimension of~$\Ga$ is~even and 
\BE{sgncond_e}\sgn_{\phi}\Ga=\begin{cases}+1,&\hbox{if}~\dim\,\Ga\!\equiv\!0,1~\hbox{mod}~4;\\
-1,&\hbox{if}~\dim\,\Ga\!\equiv\!2,3~\hbox{mod}~4.
\end{cases}\EE
If $2\!\in\!R$ is a unit, the Poincare dual of every element of $\wh{H}^{2*}_{\phi}(X,Y;R)$ 
can be represented by a pseudocycle with~$R$ coefficients uniquely up 
to pseudocycle equivalence.

\subsection{Bounding chains and curve counts}
\label{RealGWsStat_subs}

We denote by $\PC_{\phi}(X)$ and $\wt\PC_{\phi}(X)$
the collections of real pseudocycles to~$X$ and
of real bordered pseudocycles to~$[0,1]\!\times\!X$, respectively,
with coefficients in~$R$.
Let $\FPC_{\phi}(X)$ and $\wt\FPC_{\phi}(X)$ be the collections of finite subsets
of $\PC_{\phi}(X)$ and $\wt\PC_{\phi}(X)$, respectively.
The action of~$-\phi_*$ on $H_2(X,Y;\Z)$ restricts to an action 
of the cone $H_2^{\om}(X,Y)$ defined in~\eref{H2omXYdfn_e}.
Let
\BE{H2phiomdfn_e}H_{2;\phi}^{\om}(X,Y)\equiv\big\{\be\!\in\!H_2^{\om}(X,Y)\!\big\}
\!\big/\Im\big\{\Id\!+\!\phi_*\big\} \subset 
H_2\big(X,Y;\Z\big)\big/\Im\big\{\Id\!+\!\phi_*\big\}.\EE
Since $\om$ vanishes on $\Im\{\Id\!+\!\phi_*\}$,
the natural partial order on $H_2^{\om}(X,Y)$ descends to a partial order 
on~$H_{2;\phi}^{\om}(X,Y)$.

For $B\!\in\!H_{2;\phi}^{\om}(X,Y)$, $k\!\in\!\Z^{\ge0}$, a finite set~$L$,
and $J\!\in\!\cJ_{\om}^{\phi}$, let 
$$\M_{k,L}^{\st}(B;J)= \bigsqcup_{\begin{subarray}{c}\be\in H_2^{\om}(X,Y)\\
[\be]=B\end{subarray}}\hspace{-.22in}\M_{k,L}^{\st}(\be;J)\,.$$
For a path $\wt{J}$ in $\cJ_{\om}^{\phi}$,
we define the moduli space $\M_{k,L}^{\st}(B;\wt{J})$ analogously.
We apply the notation for evaluation maps and fiber products from the disk moduli 
spaces introduced in Section~\ref{Notation_subs} to 
$\M_{k,L}^{\st}(B;J)$ and $\M_{k,L}^{\st}(B;\wt{J})$ as well.

We define $\cC_{\om}^{\phi}(Y)$, $\cC_{\om;\al}^{\phi}(Y)$, $\cD_{\om}^{\phi}(\al)$, 
$\wt\cC_{\om}^{\phi}(Y)$, $\wt\cC_{\om;\wt\al}^{\phi}(Y)$, and $\cD_{\om}^{\phi}(\wt\al)$
as the collections $\cC_{\om}(Y)$, $\cC_{\om;\al}(Y)$, $\cD_{\om}(\al)$, 
$\wt\cC_{\om}(Y)$, $\wt\cC_{\om;\wt\al}(Y)$, and $\cD_{\om}(\wt\al)$
 in Section~\ref{Notation_subs}
with $H_2^{\om}(X,Y)$, $\FPC(X)$, and $\wt\FPC(X)$
replaced by $H_{2;\phi}^{\om}(X,Y)$, $\FPC_{\phi}(X)$, and $\wt\FPC_{\phi}(X)$, respectively.
For $\al\!\in\!\cC_{\om}^{\phi}(Y)$ and $J\!\in\!\cJ_{\om}^{\phi}$,
we call a collection $(\fb_{\al'})_{\al'\in\cC_{\om;\al}^{\phi}(Y)}$ 
of bordered pseudocycles to~$Y$ a \sf{real bounding chain} on~$(\al,J)$ 
if it satisfies all conditions of Definition~\ref{bndch_dfn}
with $\cC_{\om;\al}(Y)$ replaced by~$\cC_{\om;\al}^{\phi}(Y)$ and 
\BE{Rfbcond_e}\fb_{\al'}=\eset \qquad\forall~\al'\!\in\!\cC_{\om;\al}^{\phi}(Y)
~\hbox{s.t.}~\dim(\al')\equiv0~\hbox{mod}~4.\EE
Suppose $\al_0,\al_1\!\in\!\cC_{\om}^{\phi}(Y)$ and $\wt\al\!\in\!\wt\cC_{\om}^{\phi}(Y)$
are generic and satisfy~\eref{prtwtalcond_e}, 
$J_0,J_1\!\in\!\cJ_{\om}^{\phi}$ are generic, 
$\wt{J}$ is a generic path in~$\cJ_{\om}^{\phi}$ between~$J_0$ and~$J_1$, and
$(\fb_{0;\al'})_{\al'\in\cC_{\om;\al_0}^{\phi}(Y)}$ and 
$(\fb_{1;\al'})_{\al'\in\cC_{\om;\al_1}^{\phi}(Y)}$ are real
bounding chains on~$(\al_0,J_0)$ and~$(\al_1,J_1)$, respectively.
We call a collection $(\fb_{\wt\al'})_{\wt\al'\in\wt\cC_{\om;\wt\al}(Y)}$ 
of bordered pseudocycles to~$[0,1]\!\times\!Y$ 
a \sf{real pseudo-isotopy} on $(\wt\al,\wt{J})$ 
between $(\fb_{0;\al'})_{\al'\in\cC_{\om;\al_0}^{\phi}(Y)}$ and 
$(\fb_{1;\al'})_{\al'\in\cC_{\om;\al_1}^{\phi}(Y)}$ 
if it satisfies all conditions of Definition~\ref{psisot_dfn}
with $\wt\cC_{\om;\wt\al}(Y)$ replaced by~$\wt\cC_{\om;\wt\al}^{\phi}(Y)$ and 
\BE{Rfbbcond_e}\fb_{\wt\al'}=\eset \qquad\forall~\wt\al'\!\in\!\cC_{\om;\wt\al}^{\phi}(Y)
~\hbox{s.t.}~\dim(\wt\al')\equiv0~\hbox{mod}~4.\EE
Propositions~\ref{bndch_prp2} and~\ref{psisot_prp2} below are the analogues 
of Propositions~\ref{bndch_prp} and~\ref{psisot_prp} for the real setting.
They are also geometric analogues of
the surjectivity and injectivity statements of 
(a suitably modified version of)~\cite[Thm~3]{JS2}.

\begin{prp}\label{bndch_prp2}
Let $\al$ and $J$ be as above.
If $2\!\in\!R$ is a unit, $n\!\equiv\!3$ mod~4, and 
\BE{bndch2_e} H^p(Y;R)\approx H^p(S^n;R) \qquad\forall~p\!\equiv\!0,3~\tn{mod}~4,\EE 
then
there exists a real bounding chain 
$(\fb_{\al'})_{\al'\in\cC_{\om;\al}^{\phi}(Y)}$ on~$(\al,J)$.
\end{prp}

\begin{prp}\label{psisot_prp2}
Let $\al_0,\al_1,\wt\al$, $J_0,J_1,\wt{J}$, and 
$(\fb_{0;\al'})_{\al'\in\cC_{\om;\al_0}^{\phi}(Y)}$ and 
$(\fb_{1;\al'})_{\al'\in\cC_{\om;\al_1}^{\phi}(Y)}$
be as above Proposition~\ref{bndch_prp2} so that $\wt\Ga\!\cap\!Y\!=\!\eset$ 
for every $\wt\Ga\!\in\!\wt{L}(\wt\al)$ with \hbox{$\dim\,\wt\Ga\!=\!n$}.
If $R$, $n$, and~$Y$ satisfy the conditions in Proposition~\ref{bndch_prp2},
then there exists a real pseudo-isotopy 
$(\fb_{\wt\al'})_{\wt\al'\in\wt\cC_{\om;\wt\al}(Y)}$
on $(\wt\al,\wt{J})$ 
between $(\fb_{0;\al'})_{\al'\in\cC_{\om;\al_0}^{\phi}(Y)}$ and 
$(\fb_{1;\al'})_{\al'\in\cC_{\om;\al_1}^{\phi}(Y)}$. 
\end{prp}

\vspace{.1in}

The proofs of Lemmas~\ref{BCpseudo_lmm} and~\ref{psisot_lmm} 
apply verbatim in the real setting of Propositions~\ref{bndch_prp2} and~\ref{psisot_prp2}.
The key new feature in this setting 
vs.~the open setting of Propositions~\ref{bndch_prp} and~\ref{psisot_prp} is that 
\BE{psisot2_e2}
\fbb_{\al'},\fbb_{\wt\al'}=0 
\qquad\forall~
\al'\!\in\!\cC_{\om;\al}^{\phi}(Y),\,
\wt\al'\!\in\!\wt\cC_{\om;\wt\al}^{\phi}(Y)
~\hbox{s.t.}~
\dim(\al),\dim(\wt\al)\equiv0~\tn{mod}~4.\EE
As shown in Section~\ref{RealGWsPf_subs}, 
this is a consequence of the computation of the sign of the involution on the moduli space
of $J$-holomorphic disks obtained in~\cite{Jake}.
The bordered pseudocycles $\fb_{\al'},\fb_{\wt\al'}\!=\!\eset$ 
thus satisfy the conditions of Definitions~\ref{bndch_dfn}\ref{BCprt_it} 
and~\ref{psisot_dfn}\ref{isoprt_it} whenever $\dim(\al')$ and $\dim(\wt\al')$ 
are divisible by~4.
Analogously to Propositions~\ref{bndch_prp} and~\ref{psisot_prp}, 
Propositions~\ref{bndch_prp2} and~\ref{psisot_prp2} give rise to 
open GW-invariants
\BE{tauinvdfn_e}
\blr{\cdot,\ldots,\cdot}_{\be,k}^{\phi,\os}\!: 
\bigoplus_{l=0}^{\i} \wh{H}^{2*}_{\phi}(X,Y;R)^{\oplus l}\lra R, \qquad 
\be\!\in\!H_{2;\phi}^{\om}(X,Y),~k\!\in\!\Z^{\ge0},\EE
enumerating $J$-holomorphic disks of degree~$\be$.

Theorems~\ref{WelReal_thm} and~\ref{PenkaReal_thm} below are analogues
of Theorems~5 and~6 in~\cite{JS2}.
They relate the open GW-invariants~\eref{geomJSinvdfn_e} and~\eref{tauinvdfn_e}
constructed from bounding chains $(\fb_{\al'})_{\al'\in\cC_{\om;\al}(Y)}$
and $(\fb_{\al'})_{\al'\in\cC_{\om;\al}^{\phi}(Y)}$ to 
some of the previously constructed invariants enumerating real rational curves
in real symplectic $2n$-dimensional manifolds $(X,\om,\phi)$  
with $n$~odd and orientable fixed locus $Y\!=\!X^{\phi}$.
In the settings when both types of invariants are defined, 
the bounding chains~$\fb_{\al'}$ can be taken empty for $\al'\!\neq\!(0,\{\pt\},\eset)$
for any $\pt\!\in\!K(\al)$.
The open GW-invariants~\eref{geomJSinvdfn_e} and~\eref{tauinvdfn_e} then arise
only from the elements 
\BE{etaptsdfn_e}\begin{split}
&\eta=\big(\be(\al),|K(\al)|,L(\al),(\al_i)_{i\in[|K(\al)|]}\big)  
\in \cD_{\om}(\al),\cD_{\om}^{\phi}(\al) \qquad\hbox{with}\\
&\hspace{1in} \al_i=(0,\{\pt\},\eset)~~\hbox{for some}~\pt\!\in\!Y.
\end{split}\EE
The resulting disk counts~\eref{JSinvdfn_e2} and~\eref{JSinvdfn_e2b} then 
match previous definitions of counts of real rational curves up to
uniform signs and scaling.

Invariant signed counts of real $J$-holomorphic degree~$B$ spheres in compact real symplectic sixfolds
$(X,\om,\phi)$ passing through $l$~general points in~$X\!-\!X^{\phi}$ and 
$$k\equiv \frac12\blr{c_1(X,\om)}-2l$$
general points in a topological component~$Y$ of~$X^{\phi}$ 
were first defined in~\cite{Wel6,Wel6b} 
under some restrictions on~$(B,k)$. 
These counts depend on the choice of an OSpin-structure~$\os$ on~$Y$;
we denote them by~$N_{B,l}^{\phi,\os}$.
The interpretation of these counts in terms of $J$-holomorphic maps from $(\D^2,S^1)$
to~$(X,Y)$ in~\cite{Jake} relaxed the restriction on~$(B,k)$ and led to open GW-invariants
\BE{Solinvdfn_e}
\big\{\cdot,\ldots,\cdot\big\}_{\be,k}^{\phi,\os}\!: 
\bigoplus_{l=1}^{\i} \wh{H}^{2*}_{\phi}(X,Y;\Q)^{\oplus l}\lra\Q,
 \qquad \be\!\in\!H_{2;\phi}^{\om}(X,Y),~k\!\in\!\Z^+,\EE
enumerating $J$-holomorphic degree~$\be$ disks so~that 
\BE{WelvsSol_e}N_{B,l}^{\phi,\os}=2^{l-1}(-1)^{\binom{k}2}\!\!\!\!\!
\sum_{\begin{subarray}{c}\be\in H_{2;\phi}^{\om}(X,Y)\\
\fd_Y(\be)=B\end{subarray}}\!\!\!\!\!\!
\big\{\unbr{\PD_X([\pt]),\ldots,\PD_X([\pt])}{l}\big\}_{\be,k}^{\phi,\os}\,.\EE
The scaling factor of $2^{l-1}$ above is because each real $J$-holomorphic sphere
passing through $l$~conjugate pairs of points corresponds to 2~disks
passing through $l$~conjugate pairs of half-points.
The precise sign is provided by \cite[Thm~13.2]{SpinPin}.
Since $R\!=\!\Q$ and $(X,\om,Y)$ satisfy the conditions of 
Proposition~\ref{bndch_prp2} and~\ref{psisot_prp2},
the open GW-invariants~\eref{tauinvdfn_e} are well-defined in this case.
In light of this identity, the next statement is analogous to \cite[Thm~5]{JS2}.

\begin{thm}\label{WelReal_thm}
Suppose $(X,\om)$ is a compact real symplectic sixfold,
$Y$ is a topological component of~$X^{\phi}$, and
$\os$ is an OSpin-structure on~$Y$.
If $(X,\om,Y)$ satisfies~\eref{strongpos_e1} and~\eref{strongpos_e2},
then the disk counts~\eref{tauinvdfn_e} and~\eref{Solinvdfn_e} agree.
\end{thm}

We now return to compact real symplectic manifolds $(X,\om,\phi)$ of dimensions~$2n$
with arbitrary odd~$n$.
Let $Y\!\subset\!X^{\phi}$ be a topological component as before,
\begin{gather*}
\wh{H}^{2*}_{\phi;-}(X,Y;\Q)=\big\{\ga\!\in\!\wh{H}^{2*}_{\phi}(X,Y;\Q)\!:
\deg\ga\!\not\in\!4\Z\big\}, \\
H^{2*}_{\phi;-}(X;\Q)=\big\{\ga\!\in\!H^{2*}(X;\Q)\!:
\phi^*\ga=-\ga,~\deg\ga\!\not\in\!4\Z\big\}.
\end{gather*}
A \sf{real bundle pair} over~$(X,\phi)$  consists of a complex vector bundle $E\!\lra\!X$
with a conjugation~$\wt\phi$ lifting~$\phi$.
The fixed locus~$E^{\wt\phi}$ of~$\phi$ is then a real vector bundle over~$X^{\phi}$.
Let 
$$\mu_Y^{\wt\phi}\!:H_2(X,Y;\Z)\lra\Z$$
be the Maslov index of $(E,E^{\wt\phi}|_Y)$.

As introduced by Georgieva in Definition~1.3 of~\cite{Penka2}, 
a \sf{$\phi$-orientation} on $(X,\phi,Y)$ consists of a real bundle pair~$(E,\wt\phi)$ 
over~$(X,\phi)$ such that 
\BE{Penkacond_e}
\mu_Y^{\om}(\be)\equiv 2\mu_Y^{\wt\phi}(\be)~~\hbox{mod}~4
\qquad\forall~\be\!\in\!H_2(X,Y;\Z)~\hbox{with}~\phi_*\be\!=\!-\be,\EE
an OSpin-structure~$\os_E$ on $E|_Y\!\oplus\!TY$, and
a choice of representatives $\be_i\!\in\!H_2^{\om}(X,Y)$ for the elements
of $H_{2;\phi}^{\om}(X,Y)$ in~\eref{H2phiomdfn_e}.
The existence of an OSpin-structure on $E|_Y\!\oplus\!TY$
implies that $w_2(TY)\!=\!w_2(E)|_Y$.
A $\phi$-orientation $(E,\wt\phi;\os_E)$ determines an orientation 
of $\det D_u$ for every $J$-holomorphic map \hbox{$u\!:(D^2,S^1)\!\lra\!(X,Y)$}
and thus an orientation on the moduli spaces
$\M^\st_{k,L}(\be;J)$ of $J$-holomorphic disks.
A $\phi$-orientation also determines real GW-invariants
\BE{Penkainvdfn_e}
\big\{\cdot,\ldots,\cdot\big\}_{\be}^{\phi}\!: 
\bigoplus_{l=1}^{\i}H^{2*}_{\phi;-}(X;\Q)^{\oplus l}\lra\Q,
 \qquad \be\!\in\!H_{2;\phi}^{\om}(X,Y)\!-\!q_Y\big(H_2(X;\Z)\!\big),\EE
enumerating real $J$-holomorphic degree~$\fd_Y(\be)$ spheres
without point constraints in~$Y$.
The equality of the dimensions of the moduli space and 
the constraints implies that $\mu_Y^{\om}(\be)\!\equiv\!n\!-\!3$ mod~4
whenever the invariants~\eref{Penkainvdfn_e} do not vanish.

\begin{thm}\label{PenkaReal_thm}
Suppose $(X,\om)$ is a compact symplectic manifold of real dimension~$2n$ with $n\!\ge\!3$ odd,
$Y$~is a topological component of~$X^{\phi}$, 
\hbox{$\ga_1,\ldots,\ga_l\!\in\!\wh{H}^{2*}_{\phi;-}(X,Y;\Q)$},
$(E,\wt\phi;\os_E)$ is a $\phi$-orientation on $(X,\phi,Y)$,
and $J\!\in\!\cJ_{\om}^{\phi}$ is generic.
Let $\al\!\in\!\cC_{\om}^{\phi}(Y)$ be generic so that 
$\be(\al)\!\not\in\!q_Y(H_2(X;\Z)\!)$, $K(\al)\!=\!\eset$, and
$L(\al)\!\equiv\!\{\Ga_1,\ldots,\Ga_l\}$ be a collection of representatives
for the Poincare duals of $\ga_1,\ldots,\ga_l$.
If $(X,\om,Y)$ satisfies~\eref{strongpos_e1} and~\eref{strongpos_e2} and
the disk moduli spaces are oriented by the $\phi$-orientation as in~\cite{Penka2},
then the collection $(\fb_{\al'}\!\equiv\!\eset)_{\al'\in\cC_{\om;\al}^{\phi}(Y)}$
is a real bounding chain on~$(\al,J)$.
The associated disk count~\eref{JSinvdfn_e2b} satisfies
\BE{PenkaReal_e}
\lr{L(\al)}_{\be(\al);\eset}^*=2^{1-l}
\big\{\ga_1|_X,\ldots,\ga_l|_X\big\}_{\be(\al)}^{\phi}\,.\EE
\end{thm}

\vspace{.1in}

A $\phi$-orientation $(E,\wt\phi;\os_E)$ on $(X,\phi,Y)$ determines
an \sf{associated relative OSpin-structure}~$\os$ on~$Y$ with $w_2(\os)\!=\!w_2(E)$.
By the proof of Lemma~7.3 in~\cite{Penka2} and the statement of Corollary~3.8(1)
in~\cite{RGWsII}, there exists a map
\begin{gather}\notag
\ep\!:H_2(X,Y;\Z)\lra\Z_2 \qquad\hbox{with}\\
\label{Penkainvdfn_e3}
\ep(\be_i)=\flr{\frac{\mu_Y^{\wt\phi}(\be_i)}{2}}~~\forall\,\be_i,\quad
\ep(-\phi_*\be)=\ep(\be)\!+\!\frac{\mu_Y^{\om}(\be)\!-\!2\mu_Y^{\wt\phi}(\be)}{2}
~~\forall\,\be\!\in\!H_2(X,Y;\Z),
\end{gather}
such that the orientations on $\M^\st_{k,L}(\be;J)$ induced 
by $(E,\wt\phi;\os_E)$ and~$\os$ differ by $(-1)^{\ep(\be)}$.
By the second condition in~\eref{Penkainvdfn_e3}, 
$\ep(-\phi_*\be)=\ep(\be)$ if $\be$ satisfies the congruence in~\eref{Penkacond_e}.
If the latter is the case for all $\be\!\in\!H_2(X,Y;\Z)$, as assumed in Theorem~6 of~\cite{JS2},
the first claim of Theorem~\ref{PenkaReal_thm} also holds for the orientations
on the disk moduli spaces induced by~$\os$ (the same proof applies);
the second claim holds up to the multiplication by $(-1)^{\ep(\be_i)}$.
Thus, Theorem~\ref{PenkaReal_thm} is a more general and precise version of
\cite[Thm~6]{JS2}.

By Theorem~\ref{PenkaReal_thm}, the open GW-invariants~\eref{tauinvdfn_e}
depend only on the restrictions of the cohomology insertions~$\ga_1,\ldots,\ga_l$
to~$X$ whenever the assumptions of this theorem are satisfied.
This implication is non-vacuous if $n\!\equiv\!5$ mod~4.
The resulting invariants need not vanish, as illustrated by Table~2 in~\cite{RealEnum}
in the case of~$\P^5$.

\vspace{.1in}

\begin{rmk}\label{RealGWs_rmk}
For a relative OSpin-structure~$\os$ on~$Y$, 
\eref{fMsign_e} would also include $\lr{w_2(\os),\fd_Y(\be_\bu(\eta)\!)}$
to account for the difference in the trivializations of $u^*(TX,TY)$
and $\{\phi\!\circ\!u\}^*(TX,TY)$ induced by~$\os$
as in the CROrient~1$\os$(1) property in Section~7.2 of~\cite{SpinPin}.
For this reason, Propositions~\ref{bndch_prp2} and~\ref{psisot_prp2}
do not extend to relative OSpin-structures.
\end{rmk}

\subsection{Proofs of~\eref{psisot2_e2} and Theorems~\ref{WelReal_thm} and~\ref{PenkaReal_thm}}
\label{RealGWsPf_subs}

For $\al\!\in\!\cC_{\om}^{\phi}(Y)$, define 
$$\phi_*\!:\cD_\om^{\phi}(\al)\lra\cD_\om^{\phi}(\al), \quad
\phi_*\eta=\big(\be_\bu(\eta),k_\bu(\eta),L_\bu(\eta),
(\al_{k_\bu(\eta)+1-i}(\eta)\!)_{i\in[k_\bu(\eta)]}\big).$$ 
For a collection $(\fb_{\al'})_{\al'\in\cC_{\om;\al}^{\phi}(Y)}$ of maps to~$Y$ and
$\eta\!\in\!\cD_\om^{\phi}(\al)$, 
the composition of the map component to~$X$ with~$\phi$ induces a diffeomorphism 
\begin{equation*}\begin{split}
&\phi_{\eta}^+\!\!:
\M_{\eta;J}^+\!\fiber\!\!\big(\!(i\!+\!1,\fb_{\al_i(\eta)})_{i\in[k_\bu(\eta)]};
(i,\Ga_i)_{\Ga_i\in L_\bu(\eta)}\big)\\
&\hspace{1in}\lra 
\M_{\phi_*\eta;J}^+\!\fiber\!\!\big(\!(i\!+\!1,\fb_{\al_i(\phi_*\eta)})_{i\in[k_\bu(\phi_*\eta)]};
(i,\Ga_i)_{\Ga_i\in L_\bu(\phi_*\eta)}\big).
\end{split}\end{equation*}
For $\wt\al\!\in\!\wt\cC_{\om}^{\phi}(Y)$, $\wt\eta\!\in\!\cD_\om^{\phi}(\wt\al)$,
and a collection $(\fb_{\wt\al'})_{\wt\al'\in\cC_{\om;\wt\al}^{\phi}(Y)}$ of maps 
to~$[0,1]\!\times\!Y$, we define
\begin{gather*}
\phi_*\!:\cD_\om^{\phi}(\wt\al)\lra\cD_\om^{\phi}(\wt\al) \qquad\hbox{and}\\
\begin{split}
&\phi_{\wt\eta}^+\!\!:\,
\M_{\wt\eta;\wt{J}}^+\!\fiber\!\!\big(\!(i\!+\!1,\fb_{\wt\al_i(\wt\eta)})_{i\in[k_\bu(\wt\eta)]};
(i,\wt\Ga_i)_{\wt\Ga_i\in\wt{L}_\bu(\wt\eta)}\big)\\
&\hspace{1in}\lra 
\M_{\phi_*\wt\eta;\wt{J}}^+\!\fiber\!\!
\big(\!(i\!+\!1,\fb_{\wt\al_i(\phi_*\wt\eta)})_{i\in[k_\bu(\phi_*\wt\eta)]};
(i,\wt\Ga_i)_{\wt\Ga_i\in\wt{L}_\bu(\phi_*\wt\eta)}\big)
\end{split}\end{gather*}
in the same way.

\begin{lmm}\label{flipsign_lmm}
Suppose $n\!\equiv\!3$ mod~4 and $\os$ is an OSpin-structure on~$Y$. 
If $(\fb_{\al'})_{\al'\in\cC_{\om;\al}^{\phi}(Y)}$ is a real bounding chain
on~$(\al,J)$, the signs of the diffeomorphism~$\phi_{\eta}^+$ is 
$(-1)^{\dim(\al)/2+1}$ for every $\eta\!\in\!\cD_{\om}^{\phi}(\al)$.
If $(\fb_{\wt\al'})_{\wt\al'\in\wt\cC_{\om;\wt\al}^{\phi}(Y)}$ is 
a real pseudo-isotopy on~$(\wt\al,\wt{J})$, 
the sign of the diffeomorphism~$\phi_{\wt\eta}^+$ is 
$(-1)^{\dim(\wt\al)/2+1}$ for every $\wt\eta\!\in\!\cD_{\om}^{\phi}(\wt\al)$.
\end{lmm}

\begin{proof} We denote by $\phi_{\M}\!:\M^+_{\eta;J}\!\lra\!\M^+_{\phi_*\eta;J}$
the diffeomorphism induced by the composition of the map component with~$\phi$.
Let
$$\si\!:Y^{k_\bu(\eta)}\lra Y^{k_\bu(\phi_*\eta)} \qquad\hbox{and}\qquad
\si_{\fb}\!:\prod_{i\in k_\bu(\eta)}\!\!\!\!\dom(\fb_{\al_i(\eta)})
\lra \prod_{i\in k_\bu(\phi_*\eta)}\!\!\!\!\!\!\!\dom(\fb_{\al_i(\phi_*\eta)})$$
be the diffeomorphisms reversing the orders of the components.
For each $\Ga_i\!\in\!L(\al)$, let $\phi_{\Ga_i}$ be an automorphism
of $\dom(\Ga_i)$ as in~\eref{phiZdfn_e} with \hbox{$Z\!=\!\dom(\Ga_i)$}.
Define
$$\phi_{\Ga}\!\equiv\!\prod_{\Ga_i\in L_\bu(\eta)}\!\!\!\!\!\!\phi_{\Ga_i}\!: 
\prod_{\Ga_i\in L_\bu(\eta)}\!\!\!\!\!\!\dom(\Ga_i) \lra 
\prod\limits_{\Ga_i\in L_\bu(\phi_*\eta)}\!\!\!\!\!\!\!\!\!\dom(\Ga_i), \quad
\ep_{\Ga}(\eta)=\prod_{\Ga_i\in L_\bu(\eta)}\!\!\!\!\!\!\sgn_{\phi}\Ga_i\,.$$
The diagram 
$$\xymatrix{ \M^+_{\eta;J}\ar[rr]\ar[d]^{\phi_{\M}} &&
Y^{k_\bu(\eta)}\!\times\!X^{L_\bu(\eta)}
\ar@<-20pt>[d]_{\si} \ar@<15pt>[d]^{\phi^{L_{\bu}(\eta)}}&&
\prod\limits_{i\in k_\bu(\eta)}\!\!\!\!\!\dom(\fb_{\al_i(\eta)})
\!\times\!\prod\limits_{\Ga_i\in L_\bu(\eta)}\!\!\!\!\!\!\!\dom(\Ga_i)  
\ar[ll]\ar@<-25pt>[d]^{\si_{\fb}}\ar@<40pt>[d]^{\phi_{\Ga}}\\
\M^+_{\phi_*\eta;J}\ar[rr] && Y^{k_\bu(\phi_*\eta)}\!\times\!X^{L_\bu(\phi_*\eta)} &&
\prod\limits_{i\in k_\bu(\phi_*\eta)}\!\!\!\!\!\!\!\!\dom(\fb_{\al_i(\phi_*\eta)})
\!\times\!\prod\limits_{\Ga_i\in L_\bu(\phi_*\eta)}\!\!\!\!\!\!\!\!\!\!\dom(\Ga_i) \ar[ll]}$$
then commutes.

By Lemma~5.1($\fo_{\os}8$) in~\cite{RealWDVV3}, which applies for all $n$ odd,
with $k\!=\!k_{\bu}(\eta)$, $l\!=\!|L_\bu(\eta)|$, and $|L^*|\!=\!1$ or 
Proposition~5.1 in~\cite{Jake}, 
the sign of~$\phi_{\M}$ is $(-1)^{\ep_{\os}^+(\eta)}$, where 
\BE{fMsign_e}\ep_{\os}^+(\eta)=
\frac{\mu_Y^{\om}(\be_\bu(\eta)\!)}{2}\!+\!\big(k_\bu(\eta)\!+\!1\big)\!+\!\big|L_\bu(\eta)\big|
%\!+\!\blr{w_2(\os),\fd_Y(\be_\bu(\eta)\!)}
\!+\!\binom{k_\bu(\eta)}{2};\EE
% the extra $w_2(\os)$ above arises due 
the extra binomial coefficient arises due to different orderings
of the  tangent spaces at the boundary marked points here vs.~\cite{RealWDVV3}.
The sign of the diffeomorphism $\si$ is $(-1)^{\binom{k_{\bu}(\eta)}2}$,
while the diffeomorphism~$\si_{\fb}$ is orientation-preserving. 
The signs of the diffeomorphisms $\phi^{L_{\bu}(\eta)}$ and~$\phi_{\Ga}$
are $(-1)^{n|L_{\bu}(\eta)|}$ and $\ep_{\Ga}(\eta)$, respectively.
Combining this with Lemma~\ref{fibprodflip_lmm}, \eref{sgncond_e}, and
the second equality in~\eref{BCpseudo_e3}, 
we conclude~that the sign of diffeomorphism~$\phi_{\eta}^+$ is $(-1)^{\ep}$,
where
\BE{fMsign_e3}\begin{split}
\ep\!-\!1 &=\frac{\mu_Y^{\om}(\be_\bu(\eta)\!)}{2}
\!+\!k_\bu(\eta)\!-\!\frac12\sum_{\Ga_i\in L_\bu(\eta)}\!\!\!\!\!\!
\big(\codim\,\Ga_i\!-\!2\big)\\
&=\frac12\dim(\al)-\frac12\big(n\!-\!3\!-\!(n\!+\!1)k_{\bu}(\eta)\!\big)
\!-\!\frac12\sum_{i=1}^{k_{\bu}(\eta)}\!\!\!\big(\dim(\al_i(\eta)\!)\!+\!2\big).
\end{split}\EE
Along~\eref{Rfbcond_e} applied with $\al'\!=\!\al_i(\eta)\!\prec\!\al$,
this establishes the first claim. 

The proof of the second claim is almost identical, with the signs of the analogues 
of the diffeomorphisms~$\si$ and~$\si_{\fb}$ interchanged.
In this case, we use~\eref{BCpseudo_e3b} and~\eref{Rfbbcond_e} 
instead of~\eref{BCpseudo_e3} and~\eref{Rfbcond_e}.
\end{proof}

\begin{proof}[\bf{\emph{Proof of~\eref{psisot2_e2}}}]
Since $\phi_*$ is an involution on $\cD_{\om}^{\phi}(\al)$ and $\cD_{\om}^{\phi}(\wt\al)$,
the two statements of~\eref{psisot2_e2} follow from the two statements of
Lemma~\ref{flipsign_lmm}.
\end{proof}

\begin{proof}[\bf{\emph{Proof of Theorem~\ref{WelReal_thm}}}]
Since $n\!=\!3$, Definition~\ref{bndch_dfn}\ref{BC0_it} and~\eref{Rfbcond_e} 
imply that $\fb_{\al'}\!=\!\eset$ unless $\al'\!=\!(0,\pt,\eset)$ for some $\pt\!\in\!K(\al)$.
The open GW-invariants~\eref{tauinvdfn_e} are thus the sums of 
the signed cardinalities of the fiber products as in~\eref{JSinvdfn_e2} 
and~\eref{JSinvdfn_e2b} over $\eta\!\in\!\cD_{\om}^{\phi}(\al)$ 
satisfying~\eref{etaptsdfn_e}.
These sums are counts of (single) disks passing through the collection~$L(\al)$ of constraints
in~$X$ and $|K(\al)|$ points in~$Y$ traversed  by the boundaries of the disks in any order
and thus are precisely the disk counts~\eref{Solinvdfn_e} defined in~\cite{Jake}.
\end{proof}

\begin{proof}[\bf{\emph{Proof of Theorem~\ref{PenkaReal_thm}}}]
Let $\al'\!\in\!\cC_{\om;\al}^{\phi}(Y)$ and $\eta\!\in\!\cD_{\om}^{\phi}(\al')$.
If $k_{\bu}(\eta)\!\neq\!0$, $\fbb_{\eta}\!=\!\eset$ because $\fb_{\al''}\!=\!\eset$
for all $\al''\!\in\!\cC_{\om;\al}^{\phi}(Y)$.
We suppose $k_{\bu}(\eta)\!=\!0$ and apply the proof of Lemma~\ref{flipsign_lmm}
with
the disk moduli spaces oriented by the $\phi$-orientation as in~\cite{Penka2}.
This proof now applies~with 
$$\ep_{\os}^+(\eta)=1\!+\!\big|L_\bu(\eta)\big|  \quad\hbox{and}\quad
\ep\!-\!1=-\frac12\sum_{\Ga_i\in L_\bu(\eta)}\!\!\!\!\!\!
\big(\codim\,\Ga_i\!-\!2\big)$$
in~\eref{fMsign_e} and~\eref{fMsign_e3}, respectively.
Since $\deg\ga_i\!\equiv\!2$ mod~4 for all $\Ga_i\!\in\!L(\al)$,
the diffeomorphism~$\phi_{\eta}^+$ is thus orientation-reversing in this case.
It follows that $\fbb_{\al'}\!=\!0$.
This establishes the first claim of the theorem.

Since $\fb_{\al'}\!=\!0$ for all $\al'\!\in\!\cC_{\om;\al}^{\phi}(Y)$,
the open GW-invariant on the left-hand side of~\eref{PenkaReal_e}
is the signed cardinality of the fiber product as in~\eref{JSinvdfn_e2b} with
$$\eta=\big(\be(\al),0,L(\al),()\!\big)\in \cD_{\om}^{\phi}(\al).$$  
This signed cardinality is precisely the disk count on the right-hand side 
of~\eref{PenkaReal_e} defined in~\cite{Penka2} times~$2^{1-l}$;
the scaling factor appears for the same reason as in~\eref{WelvsSol_e}.
\end{proof}

\vspace{.2in}

{\it Department of Mathematics, Stony Brook University, Stony Brook, NY 11794\\
xujia@math.stonybrook.edu}

\end{document}